\documentclass[11pt]{article}
\usepackage{amssymb,amsbsy,amsmath,amsfonts,amssymb,amscd,amsthm}
\usepackage{fullpage}

\usepackage[colorlinks=true, linkcolor=black, urlcolor=blue, citecolor=black]{hyperref}
\usepackage{cancel}
\usepackage{booktabs}
\usepackage{multirow}
\usepackage{diagbox}
\usepackage{bm}
\usepackage{soul}

\numberwithin{equation}{section}
\usepackage[T1]{fontenc}
\usepackage[toc,page]{appendix}
\usepackage{hyperref}

\usepackage{url}
\usepackage{courier}
\usepackage[capitalise]{cleveref}
\crefname{equation}{}{}
\usepackage{mathtools}
\usepackage{xcolor}
\usepackage{graphicx}
\graphicspath{{./}{figures/}} 
\usepackage{tikz}
\usetikzlibrary{matrix}
\usepackage{algpseudocode}
\allowdisplaybreaks

\crefname{lemma}{lemma}{lemmas}
\crefname{proposition}{proposition}{propositions}

\newtheorem{theorem}{Theorem}[section]
\newtheorem{lemma}{Lemma}[section]
\newtheorem{proposition}{Proposition}[section]
\newtheorem{corollary}{Corollary}[section]

\theoremstyle{definition}
\newtheorem{definition}{Definition}[section]

\theoremstyle{remark}
\newtheorem{remark}{Remark}[section]

\newcommand{\EE}{\mathbb{E}}

\newcommand{\R}{\mathbb{R}}

\def\su{{\mathsf{u}}}
\def\sv{{\mathsf{v}}}

\def\i{{(\iota)}}

\newcommand{\eps}{\varepsilon}

\newcommand{\ux}{{\underline x}}

\newcommand{\beq}{\begin{equation}}
\newcommand{\eeq}{\end{equation}}

\def\sz{{\boldsymbol{z}}}
\def\sq{{\boldsymbol{q}}}
\newcommand{\ld}{\lambda}
\newcommand{\f}{\frac}
\newcommand{\e}{\epsilon}
\newcommand{\g}{\gamma}

\newcommand{\p}{\psi}

\newcommand{\z}{\zeta}

\newcommand{\hx}{\hat{x}}
\newcommand{\lc}{\left(}
\newcommand{\rc}{\right)}

\theoremstyle{plain}

\def\sideremark#1{\ifvmode\leavevmode\fi\vadjust{
		\vbox to0pt{\hbox to 0pt{\hskip\hsize\hskip1em
				\vbox{\hsize3cm\tiny\raggedright\pretolerance10000
					\noindent #1\hfill}\hss}\vbox to8pt{\vfil}\vss}}}

\titlepage
\title{Continuous-Time Heterogeneous Agent Models \\with Recursive Utility and Preference for Late Resolution}
\author{Yves Achdou\thanks{Universit\'{e} de Paris Cit\'{e} and Sorbonne Universit\'{e}, CNRS, Laboratoire Jacques-Louis Lions,  achdou@ljll.univ-paris-diderot.fr}, Qing Tang\thanks{School of Mathematics and Physics, China University of Geosciences (Wuhan),  tangqingthomas@gmail.com}}
\date{\small
Accepted for publication in \emph{Mathematical Models and Methods in Applied Sciences} (2027).\\
DOI: \href{https://doi.org/10.1142/S0218202527500011}{10.1142/S0218202527500011}
}

\begin{document}

\maketitle

\begin{abstract} 

We consider continuous-time heterogeneous agent models with recursive utility (Epstein--Zin utility) cast as mean field games, in which agents prefer late resolution of uncertainty. The model leads to a system coupling a pair of Hamilton--Jacobi--Bellman equations with state constraints and Fokker--Planck--Kolmogorov equations. We investigate the existence of solutions to the mean field game system and discuss some important qualitative features of the model. 
\end{abstract}


\section{Introduction}
	In this paper we extend the study of continuous-time heterogeneous agent models done in Achdou, Han, Lasry, Lions, and Moll \cite{achdou2022income} by addressing the Epstein--Zin recursive utility. The continuous-time formulation of the Aiyagari--Bewley--Huggett models \cite{huggett1993risk,aiyagari1994uninsured,ljungqvist2018recursive}, classical in recursive macroeconomics, and the related system of partial differential equations can be studied in the light of the mathematical theory of mean field games \cite{achdou2021mean,MR2295621}. Such models involve a large number of \textit{ex ante} identical but \textit{ex post} heterogeneous agents in an incomplete market setting. Each agent maximizes her/his utility over time with consumption decisions, facing debt limits and idiosyncratic income risks. She/he solves the stochastic optimal control problem, taking as given an equilibrium interest rate $r$:
	\begin{equation}\label{control_problem}
	\begin{aligned}
	v(x,y)={}&\max_{c_\tau}\EE \left[\int_t^{\infty} f(c_\tau,v_\tau)\, d\tau \big\vert \boldsymbol{x}_t=x,\,\boldsymbol{y}_t=y\right],\, \text{subject to}\\
      &\left\{
        \begin{array}[c]{rcl}
        \quad d\boldsymbol{x}_\tau &=& (r \boldsymbol{x}_\tau+\boldsymbol{y}_{\tau} - c_\tau)d\tau,\,\,\tau>t,\\
         \quad \boldsymbol{x}_\tau  &\geq& \underline{x}.    
        \end{array}
      \right.  
      \end{aligned}
	\end{equation}
	Here, $\boldsymbol{x}_\tau$ and $\boldsymbol{y}_{\tau}$ respectively stand for the agent's wealth and labor income at time $\tau$. The control variable is the consumption $c_\tau$ and the flow $f$ depends also on the future value $v_\tau=v(\boldsymbol{x}_\tau,\boldsymbol{y}_{\tau})$. It is assumed that $\boldsymbol{y}_{\tau}$ is a Poisson process with two states $y_1<y_2$. The Epstein--Zin recursive utility is defined as:
       \begin{equation}
         \label{EZ}
       f(c,v)=\frac{\rho}{1-\p^{-1}}\frac{c^{1-\p^{-1}}-((1-\g)v)^\theta}{((1-\g)v)^{\theta-1}},\qquad \theta=\frac{1-\p^{-1}}{1-\g},
       \end{equation}
where $\rho$ is the subjective discount rate. It is assumed that $\p$, the elasticity of intertemporal substitution (EIS) and $\g$, the risk aversion parameter, are both positive and do not take the value $1$. This type of nonexpected, recursive utility was proposed in \cite{EZ1989} in discrete time and continuous time models were introduced in \cite{duffie1992stochastic,duffie1992pde}. In the present paper, we use the same notation as in more recent economic literature \cite{pindyck2013economic,wang2016optimal}. A key feature of the Epstein--Zin utility \eqref{EZ} is the separation between risk aversion $\gamma$ and EIS $\psi$. The time-additive separable CRRA utility is a special case of recursive utility in which $\g=\psi^{-1}$ and 
\beq\label{Eq: flow CRRA}
f(c,v)=\frac{\rho c^{1-\g}}{1-\g}-\rho v.
\eeq
It was observed in \cite[p. 952]{EZ1989} that the attitude towards the timing of the resolution of uncertainty is pinned down by the constant $\g \p$: early (late) resolution is preferred if $\g \p> (<) 1$. With $\g \p=1$, the agent is indifferent to the timing of uncertainty resolution. 
In this paper, we only consider the case $\frac{1-\p^{-1}}{1-\g}>1$ with $\g > 1$, and the results can be extended to the case $\frac{1-\p^{-1}}{1-\g}>1$ with $\g < 1$. The extension to the case $\g \p=1$ is easy. Our assumption thus implies $\g \p<1$. In this context, it is more straightforward to use the viscosity solution theory and have a complete theory. Therefore, the theoretical results in this paper are only applicable to the case when agents prefer late resolution. \par 
We use the shorthand notation $v_t=v(\boldsymbol{x}_t,\boldsymbol{y}_t)$ when needed. 
Since $\boldsymbol{y}_t$ takes two values it is convenient to set $v_j(x)=v(x,y_j)$.
 The agent's optimization problem \eqref{control_problem} leads to the weakly coupled system of Hamilton-Jacobi-Bellman (HJB) equations for the value functions $v_j(x)$:
	\begin{equation}\label{eq:HJB}
		0= \max_{c\ge 0} \left\{ f(c,v_j) +(rx +y_j- c)Dv_j \right\}+\lambda_j(v_{\bar \jmath}(x)-v_j(x)),\,\, j\in \{1,2\}\quad\text{and}\,\,\bar \jmath=3-j,
	\end{equation}
	with a state constraint boundary condition at $\ux$.	In the spirit of ``Bewley models'', the interest rate $r$ is determined in equilibrium and depends on the aggregate wealth and labor in the economy. To close the model, we need to consider the distribution of agents by wealth and income. We introduce a stationary Fokker-Planck-Kolmogorov (FPK) equation, which describes the invariant measure $m_j$ of agents with income $y_j$. The measure $m_j$ has a density $g_j$ and possibly exhibits a Dirac mass at $\ux$ weighted by $\mu_j$, and the density satisfies: 
	\begin{equation}\label{eq:FP}
		- \frac{\partial}{\partial x} \left[ (rx + y_j - c_j(x)) g_j(x) \right]+\lambda_{\bar \jmath}g_{\bar \jmath}(x) -\lambda_jg_j(x)=0 .
	\end{equation}
The aggregate capital supply and labor are denoted by
\beq\label{def K N}
\begin{aligned}
&K[m]=\int_{x\geq \ux} xdm_1+\int_{x\geq \ux} xdm_2=\int_{x\geq \ux} xg_1(x)dx+\int_{x\geq \ux} xg_2(x)dx+\sum_{j\in \{1,2\}}\mu_j\ux,\\
& N[m]=\int_{x\geq \ux} y_1dm_1+\int_{x\geq \ux} y_2dm_2=\int_{x\geq \ux} y_1g_1(x)dx+\int_{x\geq \ux}y_2g_2(x)dx+\sum_{j\in \{1,2\}}\mu_jy_j.
\end{aligned}
\eeq
In the Huggett model \cite{huggett1993risk}, each agent can borrow or lend at interest rate $r$ while the aggregate capital is fixed to the value $B$. The recursive equilibrium in a stationary Huggett model is then summarized by the system
\begin{equation}\label{MFG}
\left\{\begin{aligned}
&(i)\qquad &&0 = \max_{c\ge 0} \left\{ f(c,v_j) +(r^*x+y_j - c)Dv_j(x) \right\}+\lambda_j(v_{\bar \jmath}(x)-v_j(x)),\\
&&&c_{j}^*(x)=\mathop{\arg \max}\limits_{c\ge 0} \left\{ f(c,v_j) +(r^*x+y_j - c)Dv_j(x) \right\},\\
&(ii)&&- \frac{\partial}{\partial x} \left[ (rx+y_j - c_{j}^*(x)) g_j(x) \right]+\lambda_{\bar \jmath}g_{\bar \jmath}(x) -\lambda_jg_j(x)=0,\\
&&&\int_{x\geq \ux} g_1(x)dx+\int_{x\geq \ux} g_2(x)dx+\sum_{j\in \{1,2\}}\mu_j=1,
\end{aligned}\right.
\end{equation}
with the equilibrium $r$ determined by the implicit coupling condition
\beq \label{Huggett}
 (iii_H)\qquad \mathcal{K}[r^*]=K[m^{(r^*)}]=B.
 \eeq
In the Aiyagari model \eqref{Aiyagari}, the asset $x$ is interpreted as homogeneous physical capital. The total factor of productivity (TFP) and the capital depreciation rate are respectively denoted by $A$ and $\delta$. The production in the economy is described using the Cobb-Douglas production function $\textbf{F}$ such that $\textbf{F}(K,N)=AK^{\alpha}N^{1-\alpha}$ with $0<\alpha<1$ and where $K$ is the aggregate capital and $N$ is the aggregate labor. The profit of the producer is $\textbf{F}(K,N)-(r+\delta)K-wN$ where $w$ is the wage. Here we normalize the wage to unity, i.e. $w=1$. Then, at any given interest rate $r$, the producer chooses a level of capital demand $\mathcal{K}_d[r]$ such that 
\begin{equation}\label{F}
\mathcal{K}_d[r]=\mathop{\arg\max}\limits_{K}\left\{\textbf{F}(K,N)-(r+\delta)K\right\}.
\end{equation}
It can be deduced from state constraints that under any given interest rate $r$,
\beq\label{normal N}
\mu_j+\int_{x> \ux}g_{j}(x)dx=\frac{\lambda_{\bar \jmath}}{\lambda_j+\lambda_{\bar \jmath}}, \quad N[m]=\frac{y_{\bar \jmath}\lambda_j}{\lambda_j+\lambda_{\bar \jmath}}+\frac{y_j\lambda_{\bar \jmath}}{\lambda_j+\lambda_{\bar \jmath}}.
\eeq
The stationary Aiyagari model is described by an equilibrium interest rate $r^*$ such that $\mathcal{K}[r^*]=\mathcal{K}_d[r^*]$, i.e. supply equals demand for capital. This leads one to supplement system \eqref{MFG} with the coupling condition: 
\begin{equation}\label{Aiyagari}
(iii_A)\qquad r^*=A\alpha \left(\frac{\mathcal{K}[r^*]}{N}\right)^{\alpha-1}-\delta=A\alpha \left(\frac{K[m^{(r^*)}]}{N}\right)^{\alpha-1}-\delta. 
\end{equation}
Continuous-time consumption-saving models with recursive utility have been used for studying financial consequence of disasters \cite{pindyck2013economic}, dynamic portfolio choice \cite{dimmock2024endowment} and economics of climate change \cite{hong2023mitigating}. In particular, such a model with state constraints (incomplete market setting) and stochastic income process has been considered in \cite{wang2016optimal}. \par
Discrete-time recursive utility models have been widely used for macroeconomic asset pricing \cite{fernandez2021estimating,guvenen2009parsimonious,hansen2008consumption} and were considered in a more general mathematical framework of abstract dynamic programming \cite{Sargent2025,Sargent_Stachurski_2025}. There is a vast literature on the existence and uniqueness of the value function for recursive utility process in discrete time \cite{bloise2024not,BOROVICKA2020}. The connection between discrete-time and continuous-time models with recursive utilities has been addressed in \cite{kraft2014stochastic}.\par

It was proposed in \cite{achdou2022income} that the suitable notion of solution to the HJB equation in these heterogeneous agent systems is the {\it{constrained viscosity solution}} \cite{capuzzo1990hamilton}. Since then, several works  \cite{camilli2025semi,hofer2025price,shigeta2023existence} have contributed to the mathematical analysis of such models with time-additive utilities. In this paper, we extend the {\it{constrained viscosity solution}} theory to HJB equations with recursive utility. The finite difference scheme of \cite{achdou2022income}, originally developed for models with CRRA utility, has been extended to the recursive utility setting \cite{AchdouTang2026} and is employed for the numerical examples in \cref{Sec: numerical}. In addition, a recently proposed semi-Lagrangian scheme \cite{camilli2025semi} can be adapted to the present framework and yields comparable results. \par
This paper is organized as follows. In \cref{Sec: Preliminaries} we give some preliminaries and recall useful notions in convex analysis. \cref{Sec: HJB} is devoted to the analysis of HJB equations. We discuss the existence and uniqueness of the {\it{constrained viscosity solution}}, and prove that the latter is the value function of the recursive optimal control problem. Higher regularity of the solution is then obtained using in particular arguments from convex analysis. We establish some key features of the optimal saving policies. In \cref{Sec: FPK} we consider the solution to the FPK equation given $r<\rho$. In \cref{Sec: MFG} we first study the HJB equation and saving policies with $r=\rho$ and prove the nonexistence of an invariant measure in this case. This leads to the blow up of the aggregate wealth as $r\to \rho$. We then address the existence of an equilibrium for the Aiyagari model. Finally, we give some numerical examples in \cref{Sec: numerical}.
\section{Preliminaries}\label{Sec: Preliminaries}
The assumptions that follow will be made in the whole paper:
\beq\label{asp1}
\g>1,\,\,0<\p<1,\,\,\g \p< 1.
\eeq
$$
\rho\geq r,\quad \ux<0,\quad \rho \ux+y_j>0.
$$
An equivalent statement to \eqref{asp1} would be $\g>1$ and $\theta>1$. \par
 With Epstein--Zin utility \eqref{EZ}, we can rewrite the HJB equation \eqref{eq:HJB}
	\beq\label{HJB}
	\begin{aligned}
	\frac{\rho}{\theta} v_j(x)={}& \max_{c\ge 0} \left\{ \mathcal{F}(c,v_j) +(rx +y_j- c)Dv_j \right\}+\lambda_j(v_{\bar \jmath}(x)-v_j(x))\\
	={}& H(x,y_j,v_j,Dv_j)+\lambda_j(v_{\bar \jmath}(x)-v_j(x)),
	\end{aligned}
	\eeq
	where we use the notation
	\beq\label{EZ flow}
	f(c,v_j)=\mathcal{F}(c,v_j)-\frac{\rho}{\theta} v_j,\qquad \mathcal{F}(c,v_j)=\frac{\rho}{1-\p^{-1}}\frac{c^{1-\p^{-1}}}{((1-\g)v_j)^{\theta-1}}.
	\eeq
The Hamiltonian in \Cref{HJB} is
	\begin{equation}\label{def H}
  H(x,y,v,p)=\left\{
      \begin{aligned}
       (rx +y)p+\frac{\rho^{\p}}{\p-1} p^{1-\p}((1-\g)v)^{\frac{1-\g \p}{1-\g}},\quad \quad & \text{if }\quad p\geq 0,\\
       +\infty, \quad \quad  & \text{if }\quad p< 0.
     \end{aligned}
   \right.
\end{equation}
We summarize the notation in \Cref{tab1}. 
\begin{table}[h!]
    \centering
    \caption{Symbols}
    \label{tab1}
    \begin{tabular}{l c c}
        \toprule
        Subjective discount factor & $\rho>0$ \\
        Risk aversion  & $\gamma>1$  \\
        Elasticity of intertemporal substitution (EIS) &$0<\psi<1$\\
        A parameter arising in Epstein--Zin utility &$\theta:=\frac{1-\p^{-1}}{1-\g}>1$\\
        A convenient parameter &$b:= \rho\left[\frac{r+\p(\rho-r)}{\rho}\right]^{\frac{1}{1-\p}}$\\
        Scaled discount &$\z=\rho/\theta$\\
        Aggregator & $f(c,v)$\\
        Modified aggregator & $ \mathcal{F}(c,v)=f(c,v)+\z v$\\
        Borrowing limit & $\ux$ \\
        Equilibrium interest rate & $r^*< \rho$\\
        \bottomrule
    \end{tabular}
\end{table}\par
It is clear with \eqref{asp1} that for any $c>0$, $\mathcal{F}(c,v)$ is decreasing in $v$ since 
\beq\label{F decreasing}
\mathcal{F}_v(c,v)<0.
\eeq
For all $c>0$ and $v<0$, the second order derivatives of the aggregator $f(c,v)$ and the determinants of Hessian matrices satisfy
\beq\label{Eq: f concave}
\begin{aligned}
&f_{cc}(c,v)=\mathcal{F}_{cc}(c,v)=\frac{-\rho c^{-\p^{-1}-1}}{\psi((1-\g)v)^{\theta-1}}<0,\\
&f_{vv}(c,v)=\mathcal{F}_{vv}(c,v)=(\g-\psi^{-1})\frac{\rho c^{1-\p^{-1}}}{((1-\g)v)^{\theta+1}}<0,\\
&
\det
\begin{pmatrix}
f_{cc}(c,v) & f_{cv}(c,v) \\
f_{vc}(c,v) & f_{vv}(c,v)
\end{pmatrix}
=
\rho^2\g (\p^{-1}-\g)\frac{c^{-2\p^{-1}}}{((1-\g)v)^{2\theta}}
>0.
\end{aligned}
\eeq
Therefore $f(c,v)$ and $\mathcal{F}(c,v)$ are jointly concave in $(c,v)$.\par
	The optimal consumption (away from the borrowing limit) is 
	\beq\label{c optimal}
	c_j=\mathop{\arg\max}\limits_{c\geq 0} \left\{ \mathcal{F}(c,v_j) -cDv_j \right\}=\rho^{\p}(Dv_j)^{-\p}((1-\g)v_j)^{\frac{1-\g \p}{1-\g}}.
	\eeq
We observe that $H(x,y,v,p)$ defined in \Cref{def H} is convex in $p$ for fixed $(x,y,v)$ with $p>0$,
\beq\label{H min}
\min_{p>0}H(x,y,v,p)=\rho\frac{(rx+y)^{1-\p^{-1}}}{1-\p^{-1}}\big((1-\g)v\big)^{\frac{\p^{-1}-\g}{1-\g}}.
\eeq
Moreover, since $rx+y>0$ and $p^{1-\p}$ is sublinear, we infer from \Cref{def H} the coercivity condition
\beq\label{H coerc}
H(x,y,v,p)\to +\infty \quad \text{when}\quad  p\to +\infty.
\eeq
We denote by $H_v(x,y,v,p)$ and $H_{vv}(x,y,v,p)$ the first and second order derivatives of $H(x,y,v,p)$ with respect to $v$. We use $H_{vp}(x,y,v,p)$ to denote the cross second order derivative of $H(x,y,v,p)$ with respect to $v$ and $p$. We obtain from \eqref{asp1}:
\begin{itemize}
\item $H(x,y,v,p)$ is decreasing in the $v-$ variable:
\beq\label{Eq:Hv<0}
H_v(x,y,v,p)=\frac{\rho^{\p}(1-\g \p)}{\p-1} p^{1-\p}((1-\g)v)^{\frac{\g(1-\p)}{1-\g}}< 0,
\eeq
\item
$H_{v}(x,y,v,p)$ is decreasing in both the $v-$ and $p-$ variables: 
\beq\label{Hvp}
H_{vp}(x,y,v,p)=-\rho^{\p}(1-\g \p)p^{-\p}((1-\g)v)^{\frac{\g(1-\p)}{1-\g}}< 0.
\eeq
\beq\label{Hvv}
H_{vv}(x,y,v,p)=-\g \rho^{\p}(1-\g \p) p^{1-\p}((1-\g)v)^{\frac{\g(1-\p)}{1-\g}-1}< 0.
\eeq
\end{itemize}
From \cref{Hvv}, we deduce that $H(x,y,v,p)$ is strictly concave in the $v$ variable. \par

We will also use the following elementary but important inequality whose proof  is left to the reader.
\begin{lemma}\label{Lemma: b}
For $\rho>r$, the parameter $b$ introduced in \Cref{tab1} is such that $r<b<\rho$. 
\end{lemma}

Let us recall some notions from convex analysis. 
\begin{definition}
Consider $\phi\in C[\ux,+\infty)$. The super- and the subdifferential of $\phi$ at $x$ are given by
\beq\label{Eq: Sub Sup Diff}
\begin{aligned}
(1)\quad D^+\phi(x):=\left\{p\in \R:\,\,\limsup_{z\to x,\,z\geq \ux}\frac{\phi(z)-\phi(x)-p(z-x)}{\vert z-x\vert} \leq 0\right\},\\
(2)\quad D^-\phi(x):=\left\{p\in \R:\,\,\liminf_{z\to x,\,z\geq \ux}\frac{\phi(z)-\phi(x)-p(z-x)}{\vert z-x\vert} \geq 0\right\}.
\end{aligned}
\eeq
\end{definition}
$D^+\phi(x)$ and $D^-\phi(x)$ are closed and convex subsets of $\R$. Moreover, it is clear that if $\phi$ is nondecreasing in a neighborhood of $x$ then $D^+\phi(x)$ and $D^-\phi(x)$ are subsets of $\mathbb{R}_+$. \par
For a locally Lipschitz function $\phi$, we define the set (cf. \cite[p. 63]{bardi1997optimal})
\beq
D^*\phi(x)=\left\{p\in \R: p=\lim_{n\to +\infty}D\phi(x_n),\,\,x_n\to x\right\},
\eeq
and we denote by ${\rm{co}}D^*\phi(x)$ the convex hull of $D^*\phi(x)$. We observe that if $\phi$ is nondecreasing in a neighborhood of $x$ then all elements of $D^*\phi(x)$ are nonnegative. 
\begin{definition}\label{def:semiconcave}
We say that $\phi: (\ux,+\infty)\to \mathbb{R}$ is semiconcave if there exists a constant $C\geq 0$ such that for all $x,z\in (\ux,+\infty)$ we have 
$$
\frac{1}{2} \phi(x)+\frac{1}{2}\phi(z)\leq \phi \lc \frac{x+z}{2}\rc+\frac{1}{4}C\vert x-z\vert^2.
$$
A function $\phi$ is semiconvex if $-\phi$ is semiconcave.
\end{definition}
It is obvious that any concave function is also semiconcave.
\begin{lemma}\label{lemma:semiconcave}
(\cite[Proposition 4.7]{bardi1997optimal}) Let $\phi$ be a semiconcave function in $(\ux,+\infty)$. Then $D^+\phi(x)={\rm{co}}D^*\phi(x)$ for all $x\in (\ux,+\infty)$. Moreover, if $D^+\phi(x)$ is a singleton then $\phi$ is differentiable at $x$. 
\end{lemma}
The following result is a direct consequence of \cite[Theorem 25.7]{tyrrell1970convex} and used in \cite{shigeta2023existence}. It allows one to prove the continuous dependence of the consumption and saving policies upon the interest rate $r$ by studying the convergence of the value functions, see also \cite{achdou2023mean,camilli2025semi}.
\begin{lemma}\label{Lemma: Dv converge}
Let $v^{\i}$, $\iota \in \mathbb{N}$ be a sequence of strictly concave, differentiable and bounded functions defined on $(\ux,+\infty)$. If $v^{\i}(x)\to v(x)$ for all $x$ as $\iota \to +\infty$, then $Dv^{\i}$ converges to $Dv$ locally uniformly. 
\end{lemma}

\section{Analysis of the Hamilton--Jacobi--Bellman equation}\label{Sec: HJB}
Let us recall the definition of viscosity solution for the HJB system \eqref{HJB}, which extends naturally the definition used in \cite{camilli2025semi} when the utility is of CRRA type.
\begin{definition}\hfill\label{def vis sol}
	\begin{itemize}
		\item[1.] An upper semicontinuous (u.s.c.) function $v=(v_1,v_2)$ is said to be a viscosity subsolution of \eqref{HJB} at $x$, if whenever $\varphi$ is a smooth function and $v_j-\varphi$ has a local maximum at $x$, then
	\[\frac{\rho}{\theta} v_j(x)\le H(x,y_j,v_j(x),D\varphi(x) )+\lambda_j (v_{\bar \jmath}(x)-v_j(x)).\]	
		\item[2.] A lower semicontinuous (l.s.c.) function $v=(v_1,v_2)$ is said to be a viscosity supersolution of \eqref{HJB} at $x$, if whenever $\varphi$ is a smooth function and $v_j-\varphi$ has a local  minimum at $x$, then
	\[\frac{\rho}{\theta} v_j(x)\ge H(x,y_j,v_j(x),D\varphi(x) )+\lambda_j (v_{\bar \jmath}(x)-v_j(x)).\]		
	\end{itemize}
A continuous function $v$ is said to be a constrained viscosity solution to  system \eqref{HJB} if $v$ is a viscosity supersolution in $(\ux,\infty)$ and a viscosity subsolution in $[\ux,\infty)$.
\end{definition}

It is useful to state the following equivalent definition using sub- and superdifferentials. 
\begin{definition}\hfill\label[definition]{def vis sol2}\\
	\begin{itemize}
		\item[1.] An upper semicontinuous (u.s.c.) function $v=(v_1,v_2)$ is a viscosity subsolution of \eqref{HJB} at $x$, if 
	\[\frac{\rho}{\theta} v_j(x)\le H(x,y_j,v_j(x),p)+\lambda_j (v_{\bar \jmath}(x)-v_j(x)) \quad \forall p\in D^+v_j(x).\]	
		\item[2.] A  lower semicontinuous (l.s.c.) function $v=(v_1,v_2)$ is a viscosity supersolution of \eqref{HJB} at $x$, if 
			\[\frac{\rho}{\theta} v_j(x)\ge H(x,y_j,v_j(x),p)+\lambda_j (v_{\bar \jmath}(x)-v_j(x)) \quad \forall p\in D^-v_j(x).\]			\end{itemize}
\end{definition}

We next state a strong comparison principle for viscosity sub- and supersolutions of \eqref{HJB}.
\begin{proposition}\label{comparison}
Assume that $\mathsf{u}=(\mathsf{u}_1,\mathsf{u}_2)$ and $\mathsf{v}=(\mathsf{v}_1,\mathsf{v}_2)$ are bounded viscosity sub- and supersolution of system \eqref{HJB}.We extend $\mathsf{v}_j$ at $\ux$ by setting 
$
\mathsf{v}_j(\ux)=\liminf_{\substack{z\rightarrow \ux,\, z>\ux}} \mathsf{v}_j(z) .
$
Then $\mathsf{u}\leq \mathsf{v}$ in $[\ux,+\infty)$, i.e. $\mathsf{u}_j\le \mathsf{v}_j$ in $[\ux,+\infty)$ for $j=1,2$.
\end{proposition}

The proof is contained in \ref{Comparison proof}. Note that this result holds also if $\g \p=1$, see \cite{camilli2025semi}.\par
Next we obtain some explicit sub- and supersolution of system \eqref{HJB}. To gain intuition, we first consider the decoupled system, i.e. the case when $\ld_j=0$, $j\in \{1,2\}$:
\begin{equation}\label{eq:HJB1d}
\frac{\rho}{\theta} v_j(x) = H(x, y_j, v_j,Dv_j).
\end{equation}

\begin{lemma}\label{prop:sub super sol}
Consider
\beq\label{Def:b}
\check{\su}_j=\frac{(rx+y_j)^{1-\gamma}}{1-\gamma},\quad \check{\sv}_j=\frac{(b(x+y_j/r))^{1-\gamma}}{1-\gamma},
\eeq
where $b$ is defined in \cref{tab1}. The functions $\check{\su}_j$ and $\check{\sv}_j$ are respectively a subsolution of \eqref{eq:HJB1d} in $[\ux,+\infty)$ and a supersolution of \eqref{eq:HJB1d} in $(\ux,+\infty)$.
\end{lemma}
\begin{proof}
$\check{\sv}_j(x)$ is a classical solution to \Cref{eq:HJB1d} in $(\ux,+\infty)$. Hence $\check{\sv}_j$ is a supersolution of \eqref{eq:HJB1d} in $(\ux,+\infty)$. \par
It is straightforward to verify that for all $x>\ux$, $\check{\su}_j$ satisfies 
$$
\frac{\rho}{\theta} \check{\su}_j\leq H(x, y_j, \check{\su}_j,D\check{\su}_j).
$$
Let $\varphi$ be a smooth function such that $\check{\su}_j-\varphi$ has a local maximum at $\ux$. It follows from \Cref{H min} that
$$
H(\ux, y_j, \check{\su}_j(\ux),D\varphi(\ux))\geq \rho\frac{(r\ux+y_j)^{1-\p^{-1}}}{1-\p^{-1}}\big((1-\g)\check{\su}_j(\ux)\big)^{\frac{\p^{-1}-\g}{1-\g}}=\rho \frac{(r\ux+y_j)^{1-\g}}{1-\p^{-1}}.
$$
On the other hand, we have 
$$
\frac{\rho}{\theta} \check{\su}_j(\ux)=\rho \frac{1-\g}{1-\p^{-1}}\frac{(r\ux+y_j)^{1-\g}}{1-\g}=\rho \frac{(r\ux+y_j)^{1-\g}}{1-\p^{-1}},
$$
hence $\frac{\rho}{\theta} \check{\su}_j(\ux)\leq H(\ux, y_j, \check{\su}_j(\ux),D\varphi(\ux))$ and $\check{\su}_j$ is a subsolution of \eqref{eq:HJB1d} in $[\ux,+\infty)$. 
\end{proof}
We can then build the sub- and supersolutions of \eqref{eq:HJB} by taking advantage of \Cref{prop:sub super sol}. 
\begin{proposition}\label{Prop:sub-super}
Assume $r>0$. The functions 
$$
\begin{aligned}
(\check{\su}_1,\check{\su}_2)={}&\lc \frac{(rx+y_1)^{1-\gamma}}{1-\gamma}, \frac{(rx+y_1)^{1-\gamma}}{1-\gamma}\rc\\
{\rm{and}}\quad (\check{\sv}_1,\check{\sv}_2)={}&\lc \frac{(b(x+y_2/r))^{1-\gamma}}{1-\gamma},\frac{(b(x+y_2/r))^{1-\gamma}}{1-\gamma}\rc
\end{aligned}
$$
 are respectively a subsolution of \eqref{HJB} in $[\ux,+\infty)$ and a supersolution of \eqref{HJB} in $(\ux,+\infty)$.
\end{proposition}
\begin{proof}
We first consider $(\check{\su}_1,\check{\su}_1)$. It is clear from \Cref{prop:sub super sol} that for $j=1$,
$$
H(x,y_1,\check{\su}_1,D\check{\su}_1)+\lambda_1(\check{\su}_1(x)-\check{\su}_1(x))\geq \frac{\rho}{\theta} \check{\su}_1(x),\quad \forall x> \ux. 
$$
For $x=\ux$, the argument is the same as in the proof of \Cref{prop:sub super sol}. For $j=2$ and $x>\ux$: 
$$
\begin{aligned}
&H(x,y_2,\check{\su}_1,D\check{\su}_1)+\lambda_2(\check{\su}_1(x)-\check{\su}_1(x))\\
={}&H(x,y_1,\check{\su}_1,D\check{\su}_1)+\lambda_2(\check{\su}_1(x)-\check{\su}_1(x))+(y_2-y_1)D\check{\su}_1(x) 
\geq \z \check{\su}_1(x). 
\end{aligned}
$$
Let  $\varphi$ be a smooth function such that $\check{\su}_1-\varphi$ has a local maximum at $\ux$. The inequality  
$
D\varphi(\ux)\geq D\check{\su}_1(\ux)>0
$
implies
$$
H(\ux,y_2,\check{\su}_1(\ux),D\varphi(\ux))+\lambda_1(\check{\su}_1(\ux)-\check{\su}_1(\ux))>\z \check{\su}_1(\ux).
$$
We argue similarly with $(\check{\sv}_2,\check{\sv}_2)$ and notice that for the supersolution property we only need to consider $x>\ux$. 
\end{proof}
\begin{remark} Next we use \Cref{def vis sol} and \Cref{Lemma: b} to show why $\check{\sv}_j$ is not a subsolution to \eqref{eq:HJB1d}, even though it is a classical solution in $(\ux,+\infty)$. Note that $D\check{\sv}_j(\ux)=b^{1-\gamma}(\ux+y_j/r)^{-\g}$, we consider the test function 
$$
\varphi(x)=\rho b^{\frac{1}{\p}-\g}r^{-1/\p}\frac{(x+y_j/r)^{1-\g}}{1-\g},\qquad D\varphi(x)=\rho b^{\frac{1}{\p}-\g}r^{-1/\p}(x+y_j/r)^{-\g}.
$$
Since $b>r$, $\rho>r$ and $\p>0$ we have first $r^{-1/\p}>b^{-1/\p}$, then 
$
\rho b^{\frac{1}{\p}-\g}r^{-1/\p}>b^{1+\frac{1}{\p}-\g}b^{-1/\p}>b^{1-\g}.
$
This implies $D\varphi(\ux)>D\check{\sv}_j(\ux)$. Therefore $\ux$ is a local maximum of $\check{\sv}_j-\varphi$. Now
$$
H(\ux,y_j,\check{\sv}_j(\ux),D\varphi(\ux))=\frac{\rho b^{\frac{1}{\p}-\g}r^{1-\frac{1}{\p}}(\ux+y_j/r)^{1-\g}}{1-\frac{1}{\p}}.
$$
Because 
$
\frac{z^{1-\frac{1}{\p}}}{1-\frac{1}{\p}}
$
is an increasing function in $(0,+\infty)$ and $b>r$, we have
 $$
 \frac{\rho b^{1-\g}(\ux+y_j/r)^{1-\g}}{1-\frac{1}{\p}}>\frac{\rho b^{\frac{1}{\p}-\g}r^{1-\frac{1}{\p}}(\ux+y_j/r)^{1-\g}}{1-\frac{1}{\p}}.
 $$
On the other hand,
 $$
  \frac{\rho b^{1-\g}(\ux+y_j/r)^{1-\g}}{1-\frac{1}{\p}}= \frac{\rho b^{1-\g}(\ux+y_j/r)^{1-\g}}{\theta(1-\g)}=\frac{\rho}{\theta} \check{\sv}_j(\ux),
 $$
 hence we have obtained:
$
\frac{\rho}{\theta} \check{\sv}_j(\ux)>H(\ux,y_j,\sv_j(\ux),D\varphi(\ux)).
$
Therefore, $\check{\sv}_j$ is not a subsolution at $\ux$.
\end{remark}
Similar explicit calculations show that $\check{\su}_j<\check{\sv}_j$, in agreement with the comparison principle in \Cref{comparison}. 

\begin{proposition}\label{Prop: Exist}
Assume $0<r\leq \rho$. There exists a unique bounded viscosity solution $v=(v_1,v_2)$ to \eqref{HJB}. 
\end{proposition}
\begin{proof}
We consider the following regularized problem. Given a constant $R>r\ux+y_2$, set 
$$
H_R(x,y_j,v_j,Dv_j)=\max_{0\leq c\leq R} \left\{ \mathcal{F}(c,v_j) +(rx +y_j- c)Dv_j \right\},
$$
and consider the system
\beq\label{Eq: HJB R}
	\begin{aligned}
	\frac{\rho}{\theta} v_j(x)=H_R(x,y_j,v_j,Dv_j)+\lambda_j(v_{\bar \jmath}(x)-v_j(x)),
	\end{aligned}
	\eeq
with state constraints. We proceed in several steps. \\
{\it{Step 1.}} We claim that for each $R$, the comparison principle is satisfied for the sub- and supersolution of \cref{Eq: HJB R}. The proof is similar to that of \Cref{comparison} that guarantees the uniqueness of solution if it exists. We observe that 
$$
\max_{0\leq c\leq R} \left\{ \mathcal{F}(c,\check{\sv}_j) +(rx +y_j- c)D\check{\sv}_j \right\}\leq \max_{c\geq 0} \left\{ \mathcal{F}(c,\check{\sv}_j) +(rx +y_j- c)D\check{\sv}_j \right\},
$$
hence $(\check{\sv}_1,\check{\sv}_2)$ is a supersolution of \cref{Eq: HJB R}. On the other hand, by defining
$$
\check{\su}^R_1(x)=\check{\su}^R_2(x)
=\left\{
      \begin{aligned}
       \quad \check{\su}_1(x) \quad & \text{if }\quad rx+y_1\leq R,\\
       \quad \frac{R^{1-\g}}{1-\g} \quad  & \text{if }\quad rx+y_1> R,
     \end{aligned}
   \right.
$$
we can use a similar argument as for \Cref{Prop:sub-super} to show that $(\check{\su}^R_1(x),\check{\su}^R_2(x))$ is a subsolution of \cref{Eq: HJB R}. From the comparison principle, we deduce that for a state constrained viscosity solution $\lc v^R_1,v^R_2\rc $ of \cref{Eq: HJB R}, $\check{\su}^R_j(x)\leq v^R_j(x)\leq \check{\sv}_j(x)$ holds for all $x\geq \ux$. The existence of a solution $\lc v^R_1,v^R_2\rc $ to \cref{Eq: HJB R} follows from Perron's method. We observe that $\lc v^R_1,v^R_2\rc$ is the value function of a stochastic optimal control problem and that $v^R_j$ is concave. Following the same argument as in \Cref{Prop: v C1} below, $v^R_j$ belongs to $C^1(\ux,+\infty)$. \\
{\it{Step 2.}} We now show that $Dv^R_j\geq 0$. Suppose that there exists $x^*$ such that $Dv^R_j(x^*)=\eta< 0$, then for all $x\geq x^*$, $Dv^R_j(x)\leq \eta$ and 
$$
\max_{0\leq c\leq R} \left\{ \mathcal{F}(c,v^R_j) +(rx +y_j- c)Dv^R_j \right\}=\mathcal{F}(R,v^R_j) +(rx +y_j- R)Dv^R_j .
$$
From \cref{F decreasing} and $\check{\su}_j(\ux)\leq v^R_j(x)$, we deduce $\mathcal{F}(R,v^R_j)\leq \mathcal{F}(R,\check{\su}_j(\ux))$. Since $Dv^R_j(x)\leq \eta$ for all $x\geq x^*$, $(rx +y_j- R)Dv^R_j\to -\infty$ as $x\to +\infty$. Therefore, $H_R\lc x,y_j,v^R_j,Dv^R_j\rc\to -\infty$ as $x\to +\infty$, which contradicts the boundedness of $v^R_j$. We have proved that there does not exist $x^*$ such that $Dv^R_j(x^*)< 0$. \\
\indent Similarly, from \cref{F decreasing} and $R>r\ux+y_2$, and the state constraint at $\ux$,
$$
H_R\lc \ux,y_j,v^R_j(\ux),Dv^R_j(\ux)\rc= H\lc \ux,y_j,v^R_j(\ux),Dv^R_j(\ux)\rc \geq H\lc \ux,y_j,\check{\sv}_j(\ux),Dv^R_j(\ux)\rc,
$$
we deduce from the coercivity of $H$ that $Dv^R_j(\ux)$ is bounded independently of $R$. The upper bound on $Dv^R_j$ then follows from the concavity of $v^R_j$. \\
{\it{Step 3.}} Suppose $R'>R$, we deduce from 
$$
H_R\lc x,y_j,v^R_j,Dv^R_j\rc\leq H_{R'}\lc x,y_j,v^R_j,Dv^R_j\rc,
$$
that $\lc v^R_1,v^R_2\rc $ is a subsolution of the system
\beq\label{Eq: HJB R'}
	\begin{aligned}
	\frac{\rho}{\theta} v_j(x)=H_{R'}(x,y_j,v_j,Dv_j)+\lambda_j(v_{\bar \jmath}(x)-v_j(x)).
	\end{aligned}
	\eeq
The comparison principle then gives $v^R_j\leq v^{R'}_j$. Therefore, $v^R_j$ depends on $R$ in a monotone increasing manner for $R$ sufficiently large. Since $\lc v^R_1,v^R_2\rc $ is bounded above by the supersolution $(\check{\sv}_1,\check{\sv}_2)$ independently of $R$, we infer that $\lc v^R_1,v^R_2\rc $ converges pointwise to the limit $\lc v_1,v_2\rc $ as $R\to +\infty$. By the uniform boundedness of $Dv^R_j$ and the Arzel\`a–Ascoli theorem, the sequence $\lc v^R_1,v^R_2\rc $ converges locally uniformly to $\lc v_1,v_2\rc $. Moreover the functions $v_j$ are continuous, strictly increasing, concave and tend to $0$ as $x\to \infty$. From the stability of viscosity solutions, we conclude that $(v_1,v_2)$ is a constrained viscosity solution to (2.2), in fact the unique one from the comparison principle. 
\end{proof}

\begin{proposition}\label{Prop:Lip}
The solution $v_j$ is locally Lipschitz. 
\end{proposition}
\begin{proof}
We consider some $x^*>\ux$ and its neighborhood $(x^*-\eps,x^*+\eps)$ such that $x^*-\eps>\ux$. We consider an interval $[x_1,x_2]$ such that $\ux<x_1<x^*-\eps<x^*+\eps<x_2<+\infty$ and then proceed in two steps. \\
{\bf{Step 1}}. We show that there exists $\hat{C}>0$, $\hat{C}$ depends on $x_1, x_2$ such that
\beq\label{Eq: Lip contra}
\frac{\rho}{\theta} v_j(x)< H(x,y_j,v_j(x),\hat{C})+\lambda_j (v_{\bar \jmath}(x)-v_j(x))\quad \forall x\in [x_1,x_2].
\eeq
We observe from the comparison principle and \eqref{Eq:Hv<0} that for all $\hat{C}>0$,
$$
H(x_1,y_1,\check{\sv}_2(x_2),\hat{C})\leq H(x,y_j,v_j(x),\hat{C}),
$$
and $v_{\bar \jmath}(x)\geq \check{\su}_1(x)> \check{\su}_1(x_1)$,
$$
\frac{\rho}{\theta} v_j(x)-\lambda_j (v_{\bar \jmath}(x)-v_j(x)) <-\lambda_jv_{\bar \jmath}(x)< -\lambda_j\check{\su}_1(x_1).
$$
Moreover, $H(x_1,y_1,\check{\sv}_2(x_2),C)$ is a monotone increasing function of $C$ if 
$$
C>\rho (rx_1+y_1)^{-\p^{-1}}\big((1-\g)\check{\sv}_2(x_2)\big)^{\frac{\p^{-1}-\g}{1-\g}}.
$$
Therefore, there exists $\hat{C}>0$ such that 
$
-\lambda_j\check{\su}_1(x_1)<H(x_1,y_1,\check{\sv}_2(x_2),\hat{C}),
$
hence \eqref{Eq: Lip contra} also holds. \\
{\bf{Step 2}}. We now consider the minimization problem, for some $C_0>\hat{C}$, $C_0$ may depend on $x^*$ and $\eps$, and $x\in (x^*-\eps,x^*+\eps)$,
\beq\label{Eq: loc min z}
\min_{z\in [x_1,x_2]}v_j(z)+C_0\vert z-x\vert.
\eeq
From the boundedness of $v_j$, $C_0$ can be chosen large enough such that neither $x_1$ nor $x_2$ can be the minimizer of \eqref{Eq: loc min z}. 
We claim that the unique minimizer in \eqref{Eq: loc min z} is $\hat{z}=x$. Otherwise, if $\hat{z}\neq x$ then $\vert z-x\vert$ is differentiable at $\hat{z}$ and, from the inequality for supersolution in \Cref{def vis sol} it follows
$$
\frac{\rho}{\theta} v_j(\hat{z})\geq H(\hat{z},y_j,v_j(\hat{z}),-C_0\frac{\hat{z}-x}{\vert \hat{z}-x\vert})+\lambda_j (v_{\bar \jmath}(\hat{z})-v_j(\hat{z})).
$$
This is in contradiction with \eqref{Eq: Lip contra} if $\hat{z}<x$, or with \eqref{def H} if $\hat{z}>x$. Therefore, $v_j(z)+C_0\vert z-x\vert \geq v_j(x)$ for all $z\in (x^*-\eps,x^*+\eps)$. By symmetry we can show $v_j(x)+C_0\vert x-z\vert \geq v_j(z)$.
\end{proof}

\begin{proposition}\label{Prop: v concave}
The viscosity solution $v=(v_1,v_2)$ is the value function of the optimal control problem \eqref{control_problem}. Furthermore, the function $v_j(x)$ is strictly concave in $x$.
\end{proposition}

\begin{proof}
We consider the optimal control problem, involving the function $v$ found in \Cref{Prop:Lip}: 
\begin{equation}\label{control_problem2}
\begin{aligned}
	\mathbf{V}(x,y)={}&\max_{c_\tau}\EE \left[\int_t^{\infty} f(c_\tau,v_\tau)\, d\tau \big\vert \boldsymbol{x}_t=x,\,\boldsymbol{y}_t=y\right],\\
     & \left\{
        \begin{array}[c]{rcl}
         \quad d\boldsymbol{x}_\tau &=& (r \boldsymbol{x}_\tau+\boldsymbol{y}_{\tau} - c_\tau)dt,\,\,\tau>t,\\
          \quad \boldsymbol{x}_\tau  &\geq& \underline{x}.    
        \end{array}
      \right.  
 \end{aligned}
	\end{equation}
The value function of \eqref{control_problem} is a fixed point of \eqref{control_problem2}. From the dynamic programming principle, the value function $\mathbf{V}$ of \eqref{control_problem2} is a viscosity solution of :
\beq\label{Eq: HJBV}
	\frac{\rho}{\theta} \mathbf{V}_j(x)
	=H(x,y_j,v_j,D\mathbf{V}_j)+\lambda_j(\mathbf{V}_{\bar \jmath}(x)-\mathbf{V}_j(x)).
\eeq
From the local Lipschitz continuity of $v$, we can then show the viscosity solution to \eqref{Eq: HJBV} is unique. Moreover, from \Cref{Prop:Lip}, we know $\mathbf{V}=v$ is itself a viscosity solution of \eqref{Eq: HJBV}. Therefore, the function $v$ is a solution of \eqref{control_problem}.\par 
From \Cref{Eq: f concave}, $f(c,v_j)$ is concave in $(c,v_j)$. The concavity of the value function then follows from \cite{duffie1992stochastic}.
\end{proof}
\begin{proposition}\label{Prop: v C1}
The function $v_j$ belongs to $C^1(\ux,+\infty)$. 
\end{proposition}
\begin{proof}
Our strategy is similar to that in \cite[Section 5.2, Proposition 5.7]{bardi1997optimal}. The function $v_j$ is concave and therefore semiconcave. Since $v_j$ is differentiable almost everywhere in $(\ux,+\infty)$, the set $D^*v_j(x)$ is nonempty and closed for any $x>\ux$. \par
To prove that $v_j$ is differentiable in $(\ux,+\infty)$, we only need to establish that $D^+v_j$ is a singleton. From the {\textit{semiconcavity}} of $v_j$,  $D^+v_j={\rm{co}}D^*v_j$ and we only need to prove that $D^*v_j$ is a singleton. \par
Suppose by contradiction $p^1\neq p^2\in D^*v_j$. Then there exist subsequences $\{x_n\}$, $\{z_k\}$ such that $v_j$ is differentiable at each $\{x_n\}$ and $\{z_k\}$, and 
$$
x=\lim_{n\to +\infty}x_n=\lim_{k\to +\infty}z_k,\quad p^1=\lim_{n\to +\infty}Dv_j(x_n),\quad p^2=\lim_{k\to +\infty}Dv_j(z_k).
$$
\beq
\begin{aligned}
\frac{\rho}{\theta} v_j(x_n)-H(x_n,y_j,v_j(x_n),Dv_j(x_n))&=\lambda_j (v_{\bar \jmath}(x_n)-v_j(x_n)),\\
\frac{\rho}{\theta} v_j(z_k)-H(z_k,y_j,v_j(z_k),Dv_j(z_k))&=\lambda_j (v_{\bar \jmath}(z_k)-v_j(z_k)).
\end{aligned}
\eeq
Using the supersolution found in \Cref{prop:sub super sol} and the comparison principle in \Cref{comparison}, we deduce that $v_j(x)$ is bounded away from $0$ for any $x<+\infty$. This allows us to take advantage of the continuity properties of $H$ and the continuity of $v_j$, $j\in \{1,2\}$ to infer that 
\beq
\frac{\rho}{\theta} v_j(x)-H\left(x,y_j,v_j(x),p^1\right)=\frac{\rho}{\theta} v_j(x)-H\left(x,y_j,v_j(x),p^2\right)=\lambda_j (v_{\bar \jmath}(x)-v_j(x)) .
\eeq
Set $\bar{p}=(p^1+p^2)/2$, from the strict convexity of $H$ we infer 
\beq\label{Eq: C1 H>}
\begin{aligned}
\frac{\rho}{\theta} v_j(x)-H\left(x,y_j,v_j(x),\bar{p}\right)>{}&\frac{\rho}{\theta} v_j(x)-\frac{1}{2}H\left(x,y_j,v_j(x),p^1\right)-\frac{1}{2}H\left(x,y_j,v_j(x),p^2\right)\\
={}&\lambda_j (v_{\bar \jmath}(x)-v_j(x)) .
\end{aligned}
\eeq
On the other hand $\bar{p}\in {\rm{co}}D^*v_j$, hence from \Cref{lemma:semiconcave} $\bar{p}\in D^+v_j$. Therefore, by \Cref{def vis sol2}
$$
\frac{\rho}{\theta} v_j(x)-H\left(x,y_j,v_j(x),\bar{p}\right)\leq \lambda_j (v_{\bar \jmath}(x)-v_j(x)), 
$$
in contradiction with \eqref{Eq: C1 H>}, therefore $D^+v_j={\rm{co}}D^*v_j$ is a singleton. \par
Finally, $Dv_j$ is continuous in $(\ux,+\infty)$, by the upper semicontinuity of the multivalued map $D^+v_j$ for the semiconcave function $v_j$. The coercivity of the Hamiltonian and the concavity of $v_j$ imply that $Dv_j$ is uniformly bounded in $(\ux,+\infty)$.
\end{proof}
\begin{proposition}\label{Dv UC}
The function $Dv_j$ is uniformly continuous in $[\ux,R]$ for any $R>\ux$. Moreover, $Dv_j$ is bounded on $[\ux,R]$ by a constant depending only on $\g$, $\p$, $(y_1,y_2)$ and $(\lambda_1,\lambda_2)$. 
\end{proposition}

\begin{proof}
From the concavity of $v_j$, $Dv_j(x)$ is monotone increasing as $x\to \ux$. There exists $Dv_j(\ux^+)$ such that $Dv_j(\ux^+)=\lim_{x\to \ux,x>\ux}Dv_j(x)$. Moreover, we define $$Dv_j(\ux)=\lim_{\epsilon \to 0,\epsilon>0}\frac{v_j(\ux+\epsilon)-v_j(\ux)}{\epsilon}$$ if it exists. Since $v_j$ is $C^1$ in $(\ux,R]$ and continuous in $[\ux,R]$, we obtain that $Dv_j(\ux)$ exists and $Dv_j(\ux)=Dv_j(\ux^+)$ by the mean value theorem. Therefore, $Dv_j$ is continuous on the interval $[\ux,R]$ and we obtain the uniform continuity by Heine-Borel theorem. 
\end{proof}

\begin{corollary}\label{col: c min}
Assume $0<r\leq \rho$. There exists a constant $c_{\min}>0$, independent of $r$, such that $c_j(x)>c_{\min}$ for all $x\geq \ux$ and $j\in \{1,2\}$.  
\end{corollary}
\begin{proof}
From \eqref{c optimal} and \Cref{Dv UC}, the first order condition for the optimal control \eqref{c optimal} holds 
 pointwise. \Cref{Prop:Lip} and \Cref{Dv UC} imply that $(Dv_j(x))^{-\p}$ is bounded below for all $x\geq \ux$. By comparison with the subsolution $\check{\su}_1$, see \Cref{Prop:sub-super}, we obtain that
 $$
 ((1-\g)v_j(x))^{\frac{1-\g \p}{1-\g}}\geq ((1-\g)\check{\su}_1(x))^{\frac{1-\g \p}{1-\g}}\geq (rx+y_1)^{1-\g \p}.
 $$
 If $\ux\geq 0$, then $(rx+y_1)^{1-\g \p}\geq y_1^{1-\g \p}$. If $\ux< 0$, then $(rx+y_1)^{1-\g \p}\geq (\rho \ux+y_1)^{1-\g \p}$. Therefore, we can find a lower bound $c_{\min}$ uniform w.r.t. $r$. 
\end{proof}
Now we can consider the continuous dependence of the consumption and saving policies upon the interest rate $r$. 
\begin{proposition}\label{Prop: stability r}
We denote by $v^\i$ the solution to system \eqref{eq:HJB} corresponding to an interest rate $r^\i$ with $0\leq r^\i\leq \rho$. For $r^\i \to r$, the sequence $v^\i$ converges in $C^1[\ux,R]$ to $v$ for any $R>\ux$. 
\end{proposition}
\begin{proof}
From \eqref{Eq:Hv<0}, \Cref{Prop:sub-super} and \Cref{Prop:Lip}, $H_v(x,y_j,v^\i_j,Dv^\i_j)$ is uniformly bounded for all $r^\i$ such that $0\leq r^\i\leq \rho$. By the stability of constrained viscosity solution $v^\i$ converges to $v$ uniformly. Since $v_j^\i$ is strictly concave (\Cref{Prop: v concave}), the local uniform convergence of $D v^\i_j$ to $Dv_j$ follows from \Cref{Lemma: Dv converge}. 
\end{proof}
We denote by $c_j^{\i}$ and $s_j^{\i}$ the consumption and saving policies with interest rate $r^\i$. From \Cref{Prop: stability r}, the sequences $c_j^{\i}$ and $s_j^{\i}$ converge locally uniformly to $c_j$ and $s_j$ as $r^\i\to r$ .

The semiconvexity of the value function can be obtained with control theoretic arguments, as in \cite{cannarsa2004semiconcave}.  
\begin{proposition}\label{Prop: v semiconvex}
The value function $(v_1,v_2)$ of \eqref{control_problem} is locally semiconvex.
\end{proposition}

\begin{proof}
From the stationarity of the control problem \eqref{control_problem}, we can take $t=0$ in \eqref{control_problem} and for the rest of the proof, we denote by $t$ the generic time variable. There exists $T$ such that the optimal trajectory $\boldsymbol{x}^*_t$ starting from $(x,y)$ does not reach $\ux$ for all $y\in \{y_1,y_2\}$ and all realizations of $\boldsymbol{y}_t$. From \Cref{col: c min}, for all $t\geq 0$:
$$
c\lc \boldsymbol{x}^*_t,\boldsymbol{y}_t\rc =\rho^{\p}\lc Dv\lc \boldsymbol{x}^*_t,\boldsymbol{y}_t\rc \rc^{-\p}\lc (1-\g)v\lc \boldsymbol{x}^*_t,\boldsymbol{y}_t\rc \rc^{\frac{1-\g \p}{1-\g}}>c_{\min}. 
$$
For $0<h<x-\ux$ and $\eps$ that will be chosen soon, consider the trajectories
\beq
	\left\{
\begin{aligned}
\quad &d\boldsymbol{x}^+_t=r\boldsymbol{x}^+_t+\boldsymbol{y}_t-c\lc \boldsymbol{x}^*_t,\boldsymbol{y}_t\rc-\eps e^{rt},\quad &&\boldsymbol{x}^+_0=x+h,\\
&d\boldsymbol{x}^-_t=r\boldsymbol{x}^-_t+\boldsymbol{y}_t-c\lc \boldsymbol{x}^*_t,\boldsymbol{y}_t\rc+\eps e^{rt},&&\boldsymbol{x}^-_0=x-h.
\end{aligned}
\right.
\eeq
We have
$$
\begin{aligned}
d\lc \boldsymbol{x}^+_t-\boldsymbol{x}^*_t\rc=r\lc \boldsymbol{x}^+_t-\boldsymbol{x}^*_t\rc-\eps e^{rt},\\
d\lc \boldsymbol{x}^-_t-\boldsymbol{x}^*_t\rc=r\lc \boldsymbol{x}^-_t-\boldsymbol{x}^*_t\rc+\eps e^{rt},
\end{aligned}
$$
and by taking $\eps=2h/T$ we obtain
$$
\begin{aligned}
& \boldsymbol{x}^+_t-\boldsymbol{x}^*_t=e^{rt}\lc \boldsymbol{x}^+_0-x\rc-\eps \int_0^t e^{r(t-\tau)}e^{r\tau}d\tau=he^{rt}-\eps te^{rt}=he^{rt}\lc 1-\frac{2t}{T}\rc,\\
& \boldsymbol{x}^-_t-\boldsymbol{x}^*_t=-he^{rt}+\eps te^{rt}=-he^{rt}\lc 1-\frac{2t}{T}\rc.
\end{aligned}
$$
By construction, $\boldsymbol{x}^+_{T/2}=\boldsymbol{x}^-_{T/2}=\boldsymbol{x}^*_{T/2}$. Moreover, 
\beq\label{Eq: x-x+}
\boldsymbol{x}^-_t<\boldsymbol{x}_t<\boldsymbol{x}^+_t\,\,{\text{and}}\,\,\frac{\boldsymbol{x}^+_t-\boldsymbol{x}^-_t}{2}=he^{rt}\lc 1-\frac{2t}{T}\rc, \quad \forall t<T/2.
\eeq
Consider the positive value $\delta$, depending on $x$:
\beq\label{Eq: delta time}
\delta:=\min_{\tau \in [0,T/2]}(x^*(\tau)-\ux)>0,
\eeq
and choose $h$ such that 
\beq\label{Eq:def h}
\max_{t \in [0,T/2]}he^{rt}\lc 1-\frac{2t}{T}\rc<\delta\quad {\rm{and}}\quad c_{\min}-\frac{2he^{rT}}{T}>\frac{c_{\min}}{2}.
\eeq
Without loss of generality we take $\boldsymbol{y}_0=y_j$. From the dynamic programming principle
$$
\begin{aligned}
v_j(x)={}&\mathbb{E}\left[\int_0^{T/2}f\lc c\lc \boldsymbol{x}^*_t,\boldsymbol{y}_t\rc, v\lc \boldsymbol{x}^*_t,\boldsymbol{y}_t\rc \rc dt +v(\boldsymbol{x}^*_{T/2},\boldsymbol{y}_{T/2})\right],\\
v_j(x+h)\geq{}& \mathbb{E}\left[\int_0^{T/2}f\lc c\lc \boldsymbol{x}^*_t,\boldsymbol{y}_t\rc-\eps e^{rt}, v\lc \boldsymbol{x}^+_t,\boldsymbol{y}_t\rc\rc dt +v(\boldsymbol{x}^*_{T/2},\boldsymbol{y}_{T/2})\right],\\
 v_j(x-h)\geq{}& \mathbb{E}\left[\int_0^{T/2}f\lc c\lc \boldsymbol{x}^*_t,\boldsymbol{y}_t\rc+\eps e^{rt}, v\lc \boldsymbol{x}^-_t,\boldsymbol{y}_t\rc\rc dt+v(\boldsymbol{x}^*_{T/2},\boldsymbol{y}_{T/2})\right].
\end{aligned}
$$
Using the notation from \eqref{EZ flow}, we have 
\beq\label{v semi con}
\begin{aligned}
&v_j(x+h)-2v_j(x)+v_j(x-h)\\
\geq{}&\mathbb{E}\int_0^{T/2}\underbrace{ \cdots}_{(I)} dt- \frac{\rho}{\theta}\mathbb{E}\int_0^{T/2}\lc v\lc \boldsymbol{x}^+_t,\boldsymbol{y}_t\rc-2v\lc \boldsymbol{x}^*_t,\boldsymbol{y}_t\rc+v\lc \boldsymbol{x}^-_t,\boldsymbol{y}_t\rc\rc dt\\
>{}&\mathbb{E}\int_0^{T/2}\underbrace{ \cdots}_{(I)}dt,\\
(I)={}&\mathcal{F}\lc c\lc \boldsymbol{x}^*_t,\boldsymbol{y}_t\rc-\eps e^{rt}, v\lc \boldsymbol{x}^+_t,\boldsymbol{y}_t\rc\rc -2\mathcal{F}\lc c\lc \boldsymbol{x}^*_t,\boldsymbol{y}_t\rc, v\lc \boldsymbol{x}^*_t,\boldsymbol{y}_t\rc \rc\\
&+\mathcal{F}\lc c\lc \boldsymbol{x}^*_t,\boldsymbol{y}_t\rc+\eps e^{rt}, v\lc \boldsymbol{x}^-_t,\boldsymbol{y}_t\rc\rc,
\end{aligned}
\eeq
where for the last inequality we used \eqref{Eq: x-x+} and the concavity of $v$.\par
Next, by using again the concavity of $v$ and $\boldsymbol{x}^*_t=(\boldsymbol{x}^+_t+\boldsymbol{x}^-_t)/2$, 
$$
v\lc \boldsymbol{x}^*_t,\boldsymbol{y}_t\rc=v\lc (\boldsymbol{x}^+_t+\boldsymbol{x}^-_t)/2,\boldsymbol{y}_t\rc>\frac{ v\lc \boldsymbol{x}^+_t,\boldsymbol{y}_t\rc+ v\lc \boldsymbol{x}^-_t,\boldsymbol{y}_t\rc}{2}.
$$
From \eqref{F decreasing} $-\mathcal{F}(c,v)$ is increasing in $v$, we have 
$$
 -\mathcal{F}\lc c\lc \boldsymbol{x}^*_t,\boldsymbol{y}_t\rc, v\lc \boldsymbol{x}^*_t,\boldsymbol{y}_t\rc \rc> -\mathcal{F}\lc c\lc \boldsymbol{x}^*_t,\boldsymbol{y}_t\rc, \frac{ v\lc \boldsymbol{x}^+_t,\boldsymbol{y}_t\rc+ v\lc \boldsymbol{x}^-_t,\boldsymbol{y}_t\rc}{2}\rc.
$$
$$
\begin{aligned}
 (I)>{}&\mathcal{F}\lc c\lc \boldsymbol{x}^*_t,\boldsymbol{y}_t\rc-\eps e^{rt}, v\lc \boldsymbol{x}^+_t,\boldsymbol{y}_t\rc\rc -2\mathcal{F}\lc c\lc \boldsymbol{x}^*_t,\boldsymbol{y}_t\rc, \frac{ v\lc \boldsymbol{x}^+_t,\boldsymbol{y}_t\rc+ v\lc \boldsymbol{x}^-_t,\boldsymbol{y}_t\rc}{2}\rc\\
&+\mathcal{F}\lc c\lc \boldsymbol{x}^*_t,\boldsymbol{y}_t\rc+\eps e^{rt}, v\lc \boldsymbol{x}^-_t,\boldsymbol{y}_t\rc\rc.
\end{aligned}
$$
where $c\lc \boldsymbol{x}^*_t,\boldsymbol{y}_t\rc-\eps e^{rt}>0$ from the choice of $h$ in \eqref{Eq:def h}. From \eqref{Eq:def h} $x+\delta\geq \boldsymbol{x}^+_t$ with $\delta$ defined in \eqref{Eq: delta time}. From the monotonicity of $v$ and by comparison with the supersolution  $\check{\sv}_2$, we have
$$
v\lc \boldsymbol{x}^+_t,\boldsymbol{y}_t\rc \leq 
v\lc x+\delta, \boldsymbol{y}_t\rc \leq \check{\sv}_2(x+\delta).
$$
From \Cref{col: c min} we obtain that for all 
\beq\label{Eq: (c,v)}
c\in [c\lc \boldsymbol{x}^*_t,\boldsymbol{y}_t\rc-\eps e^{rt},c\lc \boldsymbol{x}^*_t,\boldsymbol{y}_t\rc+\eps e^{rt}],\quad v\in [v\lc \boldsymbol{x}^-_t,\boldsymbol{y}_t\rc,v\lc \boldsymbol{x}^+_t,\boldsymbol{y}_t\rc],
\eeq
$$
\begin{aligned}
&\mathcal{F}_{cc}(c,v)\geq \frac{-\rho c_{\min}^{-\p^{-1}-1}}{\psi((1-\g)\check{\sv}_2(x+\delta))^{\theta-1}}=C_1,\\
& \mathcal{F}_{vv}(c,v)\geq (\g-\psi^{-1})\frac{\rho c_{\min}^{1-\p^{-1}}}{((1-\g)\check{\sv}_2(x+\delta))^{\theta+1}}=C_2,\\
&0<\mathcal{F}_{cv}(c,v) \leq (\psi^{-1}-\g)\frac{\rho c_{\min}^{-\p^{-1}}}{((1-\g)\check{\sv}_2(x+\delta))^{\theta}}=C_3.
\end{aligned}
$$
There exists 
$
C>\max\left\{-C_1,\,-C_2,\,C_3\right\}
$
such that the term  $(I)$ in \eqref{v semi con} satisfies 
$$
(I)>-C\lc \vert \eps e^{rt}\vert^2+\left\vert \frac{v\lc \boldsymbol{x}^+_t,\boldsymbol{y}_t\rc-v\lc \boldsymbol{x}^-_t,\boldsymbol{y}_t\rc}{2}\right\vert^2\rc.
$$
Since $\eps=2h/T$ and $t\leq T/2$ in \cref{v semi con} we have $\vert \eps e^{rt}\vert^2=\frac{4e^{2rt}h^2}{T^2}\leq \frac{4e^{\rho T}h^2}{T^2}$. 
From \Cref{Dv UC} and \cref{Eq: x-x+}, we have 
$$
\left\vert \frac{v\lc \boldsymbol{x}^+_t,\boldsymbol{y}_t\rc-v\lc \boldsymbol{x}^-_t,\boldsymbol{y}_t\rc}{2}\right\vert^2\leq \|Dv_j\|^2_{L^\infty}\left\vert \frac{\boldsymbol{x}^+_t-\boldsymbol{x}^-_t}{2}\right\vert^2\leq e^{\rho T}\|Dv_j\|^2_{L^\infty}h^2.
$$
 Therefore from \eqref{v semi con} we obtain 
 $$
v_j(x+h)-2v_j(x)+v_j(x-h)>-\frac{CTe^{\rho T}}{2}\lc \frac{4}{T^2}+\|Dv_j\|^2_{L^\infty}\rc h^2
 $$
 
 and we obtain the local semi-convexity of $v_j$ from \Cref{def:semiconcave}. 
\end{proof}
 From \Cref{Prop: v concave} and \Cref{Prop: v semiconvex} $v_j$ is strictly concave and locally semiconvex. Then, from \cite[Lemma 10.7]{fleming2006controlled}:
\begin{proposition}\label{Prop:v W2}
The function $v_j$ is $W^{2,\infty}_{loc}$ in $(\ux,+\infty)$. 
\end{proposition}

Below, we prove the expected fact that the value of the productive agents is larger than that of the unproductive ones: 
\begin{proposition}\label{Prop: v2 v1}
We have $v_2> v_1$.
\end{proposition}
\begin{proof}
Let $v_2-v_1$ achieve its minimum at $\hx$. Assume by contradiction that $v_2(\hx)-v_1(\hx)\leq 0$. We distinguish two cases.\par
{\textit{Case 1}}:  $\hx>\ux$. From \Cref{Prop: v C1}, we deduce that $Dv_2(\hx)=Dv_1(\hx)$ and 
$$
\begin{aligned}
&\lc \frac{\rho}{\theta}+\lambda_1+\lambda_2\rc(\underbrace{v_2(\hx)-v_1(\hx)}_{\leq 0})\\
={}&(y_2-y_1)Dv_2(\hx)+\underbrace{H(\hx,y_1,v_2(\hx),Dv_2(\hx))-H(\hx,y_1,v_1(\hx),Dv_1(\hx))}_{>0},
\end{aligned}
$$
the desired contradiction. \par
{\textit{Case 2}}:  $\hx=\ux$. Since $v_2-v_1$ is $C^1$ up to the boundary and achieves its minimum at $\ux$, $Dv_2(\ux)\geq Dv_1(\ux)$. From the state constraint condition,  
$$
H_p(\ux,y_1,v_1(\ux),Dv_1(\ux))=s_1(\ux)\geq 0,
$$
and, from the convexity of $H(x,y_j,v_j,p)$ in the $p$-variable and $Dv_2(\ux)\geq Dv_1(\ux)$, 
\beq\label{HDv2-HDv1}
H(\ux,y_1,v_1(\ux),Dv_2(\ux))-H(\ux,y_1,v_1(\ux),Dv_1(\ux))\geq s_1(\ux)\lc Dv_2(\ux)-Dv_1(\ux)\rc\geq 0.
\eeq
Subtracting the two HJB equations, we obtain
\begin{align*}
&\lc \frac{\rho}{\theta}+\lambda_1+\lambda_2\rc (v_2(\ux)-v_1(\ux))\\
={}& (y_2-y_1)Dv_2(\hx)+H(\ux,y_1,v_2(\ux),Dv_2(\ux))-H(\ux,y_1,v_1(\ux),Dv_1(\ux)),
\end{align*}
where
\begin{align*}
&H(\ux,y_1,v_2(\ux),Dv_2(\ux))-H(\ux,y_1,v_1(\ux),Dv_1(\ux))\\
={}&\underbrace{H(\ux,y_1,v_2(\ux),Dv_2(\ux))-H(\ux,y_1,v_1(\ux),Dv_2(\ux))}_{\geq 0}\\
&+\underbrace{H(\ux,y_1,v_1(\ux),Dv_2(\ux))-H(\ux,y_1,v_1(\ux),Dv_1(\ux))}_{\geq 0,\,\Cref{HDv2-HDv1}}.
\end{align*}
This is in contradiction with the assumption $v_2(\ux)\leq v_1(\ux)$. 
\end{proof}
The following proposition states that the savings of the unproductive agents are negative for all values of $x>\ux$. This result is well known in the model involving a CRRA utility but its proof is more difficult in the present case, because one has to handle the dependence of the Hamiltonian (hence the optimal consumption policy away from the borrowing limit, see \eqref{c optimal}) on the value $v_j$. 
\begin{proposition}\label{Prop: s1<0}
The optimal saving policy $s_1$ has the following properties: $s_1(x)< 0$ for all $x> \ux$ and $s_1(\ux)=0$. 
\end{proposition}
\begin{proof}
We argue by contradiction and suppose $s_1(\hx)\geq 0$ for some $\hx>\ux$. We may first suppose that $s_2(\hx)\neq 0$ and $s_1(\hx)>0$, this implies that the functions $v_j$ are $C^2$ in a neighborhood of $\hx$ and that $-\infty<D^2v_j(\hx)<0$. Differentiating the HJB equations at $\hx$ leads to
\beq\label{Eq: Dv1}
\begin{aligned}
&\lc \rho-r\rc Dv_1(\hx)+\lambda_1\lc Dv_1(\hx)-Dv_2(\hx)\rc\\
={}&s_1(\hx)D^2v_1(\hx)+\lc 1-\frac{1}{\theta}\rc \rho Dv_1(\hx)+H_v(\hx,y_1,v_1(\hx),Dv_1(\hx))Dv_1(\hx),
\end{aligned}
\eeq
and 
\beq\label{Eq: Dv2}
\begin{aligned}
&\lc \rho-r\rc Dv_2(\hx)+\lambda_2\lc Dv_2(\hx)-Dv_1(\hx)\rc\\
={}&s_2(\hx)D^2v_2(\hx)+\lc 1-\frac{1}{\theta}\rc \rho Dv_2(\hx)+H_v(\hx,y_2,v_2(\hx),Dv_2(\hx))Dv_2(\hx).
\end{aligned}
\eeq
If on the contrary, $s_1(\hx)= 0$ or $s_2(\hx)= 0$, then from \Cref{Prop:v W2} we know that the functions $Dv_j$ are differentiable almost everywhere and $D^2v_j$ are essentially bounded in a neighborhood of $\hx$. Therefore, $s_j(x)D^2v_j(x)$ has a sense at almost every $x$ in a neighborhood of $\hx$ and we can pass to the limit. On the other hand, from the continuity of $v_j$ and $Dv_j$ with respect to $x$ and the continuity of $H_v(x,y_j,v_j,Dv_j)$ w.r.t. $v_j$ and $Dv_j$, $\lim_{x\to \hx}Dv_j(x)=Dv_j(\hx)$, 
$$
\lim_{x\to \hx}H_v(x,y_j,v_j(x),Dv_j(x))Dv_j(\hx)=H_v(\hx,y_j,v_j(\hx),Dv_j(\hx))Dv_j(\hx).
$$
Hence, with the convention $s_j(\hx)D^2v_j(\hx)=\lim_{x\to \hx}s_j(x)D^2v_j(x)$, we can still write \eqref{Eq: Dv1} and \eqref{Eq: Dv2} if $s_2(\hx)= 0$ or $s_1(\hx)= 0$.\par
Since $s_1(\hx)\geq 0$ we deduce
$$
\begin{aligned}
\rho^{\p}(Dv_1(\hx))^{-\p}((1-\g)v_1(\hx))^{\frac{1-\g \p}{1-\g}}=c_1(\hx)\leq r\hx+y_1,\\
(Dv_1(\hx))^{1-\p}\geq \rho^{1-\p}((1-\g)v_1(\hx))^{\frac{(1-\g \p)(1-\p)}{(1-\g)\p}}(r\hx+y_1)^{1-\frac{1}{\p}}.
\end{aligned}
$$
This implies
\beq
\begin{aligned}
H_v(\hx,y_1,v_1(\hx),Dv_1(\hx))\leq  {}&\underbrace{-\rho \lc 1-\frac{1}{\theta}\rc ((1-\g)v_1(\hx))^{\frac{\p^{-1}-1}{1-\g}}(r\hx+y_1)^{1-\frac{1}{\p}}}_{decreasing\,\, in\,\, v_1(\hx)}\\
\leq {}&-\rho \lc 1-\frac{1}{\theta}\rc.
\end{aligned}
\eeq
For the last inequality we used the comparison with the subsolution in \eqref{Def:b}:
$
v_1(\hx)\geq \frac{(r\hx+y_1)^{1-\gamma}}{1-\gamma}.
$
This implies 
\beq\label{Ineq: Hv hx}
\lc 1-\frac{1}{\theta}\rc \rho Dv_1(\hx)+H_v(\hx,y_1,v_1(\hx),Dv_1(\hx))Dv_1(\hx)\leq 0.
\eeq
If $s_1(\hx)>0$, then $s_1(\hx)D^2v_1(\hx)\leq 0$. Moreover, since $r\leq \rho$ and $Dv_1\geq 0$, there holds $\lc \rho-r\rc Dv_1(\hx)\geq 0$. With $s_1(\hx)D^2v_1(\hx)\leq 0$, $\lc \rho-r\rc Dv_1(\hx)\geq 0$ and 
\eqref{Ineq: Hv hx}, we deduce from \eqref{Eq: Dv1} that:
\beq\label{Ineq: Dv1 Dv2}
Dv_1(\hx)\leq Dv_2(\hx).
\eeq
Then, from $v_1(\hx)<v_2(\hx)$ in \Cref{Prop: v2 v1} we infer (notice that $H_v(x,y,v,p)$ in fact does not depend on $y$):
\beq\label{Eq: Dv1 r}
H_v(\hx,y_2,v_2(\hx),Dv_2(\hx))<H_v(\hx,y_1,v_1(\hx),Dv_1(\hx))\leq -\rho \lc 1-\frac{1}{\theta}\rc,
\eeq
hence 
$$
\lc 1-\frac{1}{\theta}\rc \rho Dv_2(\hx)+H_v(\hx,y_2,v_2(\hx),Dv_2(\hx))Dv_2(\hx)<0.
$$
Next, since $\rho\geq r$, $Dv_2(\hx)>0$ and $Dv_1(\hx)\leq Dv_2(\hx)$, 
$$\lc \rho-r\rc Dv_2(\hx)+\lambda_2\lc Dv_2(\hx)-Dv_1(\hx)\rc\geq 0.$$
 Therefore, from \eqref{Eq: Dv2} and \eqref{Eq: Dv1 r} we deduce $s_2(\hx)D^2v_2(\hx)>0$. But we know that $D^2v_2(\hx)<0$, so $s_2(\hx)<0$. This gives
\beq\label{Eq: c2}
c_2(\hx)>r\hx+y_2.
\eeq
Recall that the assumption $s_1(\hx)\geq 0$ is equivalent to
\beq\label{Eq: c1}
 c_1(\hx)\leq r\hx+y_1.
 \eeq
 With $c_j(\hx)>c_{\min}>0$, we deduce from \eqref{Eq: c2} and \eqref{Eq: c1} that
\beq\label{Eq: c2 c1}
 \frac{c_2(\hx)}{c_1(\hx)}>\frac{r\hx+y_2}{r\hx+y_1}.
\eeq
From $s_2(\hx)<0$, we infer from \cref{HJB} that
$$
\frac{\rho}{\theta}v_2(\hx)=\mathcal{F}(c_2(\hx),v_2(\hx)) +\underbrace{s_2(\hx)Dv_2(\hx)}_{<0}+\lambda_2(\underbrace{v_{1}(\hx)-v_2(\hx)}_{<0})< \mathcal{F}(c_2(\hx),v_2(\hx)),
$$
which yields
$
v_2(\hx)<\frac{(c_2(\hx))^{1-\gamma}}{1-\gamma}.
$
From $s_1(\hx)\geq 0$, we infer from \cref{HJB} that
$$
\frac{\rho}{\theta}v_1(\hx)=\mathcal{F}(c_1(\hx),v_1(\hx)) +\underbrace{s_1(\hx)Dv_1(\hx)}_{\geq 0}+\lambda_1(\underbrace{v_{2}(\hx)-v_1(\hx)}_{> 0})> \mathcal{F}(c_1(\hx),v_1(\hx)),
$$
which yields
$
v_1(\hx)>\frac{(c_1(\hx))^{1-\gamma}}{1-\gamma}.
$
We can then deduce from \cref{Ineq: Dv1 Dv2} that 
$$
\begin{aligned}
 \frac{c_2(\hx)}{c_1(\hx)}\underbrace{=}_{\eqref{c optimal}}\lc \frac{Dv_2(\hx)}{Dv_1(\hx)}\rc^{-\p}\lc \frac{(1-\g)v_2(\hx)}{(1-\g)v_1(\hx)}\rc^{\frac{1-\g \p}{1-\g}}\leq \lc \frac{c_2(\hx)}{c_1(\hx)}\rc^{1-\g \p}.
 \end{aligned}
$$
This yields ${c_2(\hx)}\leq {c_1(\hx)}$, in contradiction with \cref{Eq: c2 c1}. Therefore, we conclude $s_1(x)<0$ for all $x>\ux$. Finally, from the state constraint $s_1(\ux)\geq 0$ and the continuity of $s_1$, we obtain $s_1(\ux)=0$. 
\end{proof}
\begin{corollary}\label{Cor: Dv1 Dv2}
We have $Dv_1(\ux)>Dv_2(\ux)$.
\end{corollary}
\begin{proof}
{\it{Step 1}}. We assume $s_2(\ux)>0$. In this case, we can differentiate the HJB equation for $v_2$ at $\ux$. From \Cref{Prop: s1<0}, we can also differentiate the HJB equation satisfied by $v_1$ at $x>\ux$ and pass to the limit as $x\to \ux$,
\beq\label{Eq: Euler ux}
\begin{aligned}
&\lc \rho-r\rc Dv_1(\ux)+\lambda_1\lc Dv_1(\ux)-Dv_2(\ux)\rc\\
={}&\lim_{x\to \ux}s_1(x)D^2v_1(x)+\lc 1-\frac{1}{\theta}\rc \rho Dv_1(\ux)+H_v(\ux,y_1,v_1(\ux),Dv_1(\ux))Dv_1(\ux).
\end{aligned}
\eeq
 By subtracting the resulting equations, we obtain
\beq\label{Eq: Dv1-Dv2 1}
\begin{aligned}
&(\rho-r+\ld_1+\ld_2)(Dv_1(\ux)-Dv_2(\ux))\\
={}& \lim_{x\to \ux}s_1(x)D^2v_1(x)-s_2(\ux)D^2v_2(\ux)+\lc 1-\frac{1}{\theta}\rc \rho (Dv_1(\ux)-Dv_2(\ux))\\
&+H_v(\ux,y_1,v_1(\ux),Dv_1(\ux))Dv_1(\ux)-H_v(\ux,y_2,v_2(\ux),Dv_2(\ux))Dv_2(\ux).
\end{aligned}
\eeq
Since $s_1(x)<0$ we know that $\lim_{x\to \ux}s_1(x)D^2v_1(x)\geq 0$. On the other hand, $s_2(\ux)> 0$ implies $-s_2(\ux)D^2v_2(\ux)>0$ and finally
$$
\lim_{x\to \ux}s_1(x)D^2v_1(x)-s_2(\ux)D^2v_2(\ux)>0.
$$
By rearranging \eqref{Eq: Dv1-Dv2 1} we have 
\beq\label{Eq: Dv1-Dv2 2}
\begin{aligned}
&(\rho-r+\ld_1+\ld_2)(Dv_1(\ux)-Dv_2(\ux))\\
&-\lc \lc 1-\frac{1}{\theta}\rc \rho+H_v(\ux,y_1,v_1(\ux),Dv_1(\ux))\rc (Dv_1(\ux)-Dv_2(\ux))\\
={}& \lim_{x\to \ux}s_1(x)D^2v_1(x)-s_2(\ux)D^2v_2(\ux)\\
&+\lc H_v(\ux,y_1,v_1(\ux),Dv_1(\ux))-H_v(\ux,y_2,v_2(\ux),Dv_2(\ux))\rc Dv_2(\ux),
\end{aligned}
\eeq

\beq\label{Eq: Hv ux}
\begin{aligned}
H_v(\ux,y_1,v_1(\ux),Dv_1(\ux))\leq {}&-\rho \lc 1-\frac{1}{\theta}\rc ((1-\g)v_1(\ux))^{\frac{\p^{-1}-1}{1-\g}}(r\ux+y_1)^{1-\frac{1}{\p}}\\
\leq {}&-\rho \lc 1-\frac{1}{\theta}\rc.
\end{aligned}
\eeq
Next we consider 
$$
\begin{aligned}
&H_v(\ux,y_1,v_1(\ux),Dv_1(\ux))-H_v(\ux,y_2,v_2(\ux),Dv_2(\ux))\\
={}&H_v(\ux,y_1,v_1(\ux),Dv_1(\ux))-H_v(\ux,y_2,v_2(\ux),Dv_1(\ux))+H_v(\ux,y_2,v_2(\ux),Dv_1(\ux))\\
&-H_v(\ux,y_2,v_2(\ux),Dv_2(\ux)).
\end{aligned}
$$
From the mean value theorem, there exists $\chi_1$, $v_1(\ux)<\chi_1<v_2(\ux)$ such that
$$
\begin{aligned}
&H_v(\ux,y_1,v_1(\ux),Dv_1(\ux))-H_v(\ux,y_2,v_2(\ux),Dv_2(\ux))\\
={}&H_{vv}(\ux,y_1,\chi_1,Dv_1(\ux))(v_1(\ux)-v_2(\ux)).
\end{aligned}
$$
From \eqref{Hvv}, $v_1(\ux)<v_2(\ux)$ and $Dv_2(\ux)\geq 0$ we have
$$
H_{vv}(\ux,y_1,\chi_1,Dv_1(\ux))(v_1(\ux)-v_2(\ux))Dv_2(\ux)\geq 0.
$$
For some $\xi \in (0,1)$ and $\chi_2=\xi Dv_1(\ux)+(1-\xi)Dv_2(\ux)$,
$$
\begin{aligned}
&H_v(\ux,y_2,v_2(\ux),Dv_1(\ux))-H_v(\ux,y_2,v_2(\ux),Dv_2(\ux))\\
={}&H_{vp}(\ux,y_2,v_2(\ux),\chi_2)(Dv_1(\ux)-Dv_2(\ux)).
\end{aligned}
$$
From \eqref{Hvp}, $H_{vp}(\ux,y_2,v_2(\ux),\chi_2)Dv_2(\ux)\leq 0$. Combining \eqref{Eq: Dv1-Dv2 2} and the observation above yields $Dv_1(\ux)>Dv_2(\ux)$.\par
{\it{Step 2}}. Now we consider the case $s_2(\ux)=0$. Similarly to the proof of \Cref{Prop: s1<0}, we deduce 
$v_2(\ux)\leq \frac{(r\ux+y_2)^{1-\gamma}}{1-\gamma}$.
From $c_2(\ux)=r\ux+y_2$ and \cref{c optimal}, we infer $Dv_2(\ux)\leq \rho (r\ux+y_2)^{-\g}$. From $s_1(\ux)=0$, $v_1(\ux)\geq \frac{(r\ux+y_1)^{1-\gamma}}{1-\gamma}$ and \cref{c optimal}, we infer $Dv_1(\ux)\geq \rho (r\ux+y_1)^{-\g}$. Therefore, we conclude $Dv_1(\ux)>Dv_2(\ux)$.
\end{proof}
A consequence of the above results is that the value of the unproductive agents is singular at the borrowing limit.
\begin{corollary}\label{Cor: D^2v1}
We have 
\beq
\lim_{x\to \ux}D^2v_1(x)=-\infty.
\eeq
\end{corollary}
\begin{proof}
We deduce from \eqref{Eq: Hv ux} that
$$
\begin{aligned}
\lc 1-\frac{1}{\theta}\rc \rho Dv_1(\ux)+H_v(\ux,y_1,v_1(\ux),Dv_1(\ux))Dv_1(\ux)\leq 0.
\end{aligned}
$$
With this inequality,  $\rho\geq r$, \Cref{Cor: Dv1 Dv2} and \eqref{Eq: Euler ux} lead to
$$
0<(\rho-r)Dv_1(\ux)+\ld_1(Dv_1(\ux)-Dv_2(\ux))\leq \lim_{x\to \ux}s_1(x)D^2v_1(x).
$$
This implies the desired result. 
\end{proof}

The following proposition provides a sufficient condition under which the savings of productive agents are negative for all $x > \ux$. In this regime, productive agents decumulate capital regardless of their wealth.
\begin{proposition}\label{Prop:s2boundary}
Suppose $r>0$ and 
\beq\label{rho-r}
\begin{aligned}
\lc \frac{\rho}{\theta}-r\rc(r\ux+y_2)^{-1/\psi}+\lambda_2\lc \underbrace{(r\ux+y_2)^{-1/\psi}-(r\ux+y_1)^{-1/\psi}}_{<0}\rc\geq 0.
\end{aligned}
\eeq
Then $s_2(x)<0$ for all $x>\ux$ and $s_2(\ux)=0$. 
\end{proposition}
\begin{proof}
We argue by contradiction and suppose that there exists $\hx>\ux$ such that $s_2(\hx)\geq 0$. Then 
$$
c_2(\hx)\leq r\hx+y_2\quad {\text{and}}\quad Dv_2(\hx)\geq \rho(r\hx+y_2)^{-1/\psi}((1-\g)v_2(\hx))^{\frac{1-\g \psi}{\psi(1-\g)}}. 
$$
From \Cref{Prop: s1<0}, $s_1(\hx)<0$, hence
$$
c_1(\hx)>r\hx+y_1\quad {\text{and}}\quad Dv_1(\hx)<\rho(r\hx+y_1)^{-1/\psi}((1-\g)v_1(\hx))^{\frac{1-\g \psi}{\psi(1-\g)}}. 
$$

With \Cref{Prop:v W2}, if $s_2(\hx)> 0$ we differentiate the HJB equation for $v_2$ at $\hx$ to obtain
\beq\label{Eq: Dv2 2}
 \begin{aligned}
&\lc \frac{\rho}{\theta}-r\rc Dv_2(\hx)+\lambda_2\lc Dv_2(\hx)-Dv_1(\hx)\rc\\
={}&s_2(\hx)D^2v_2(\hx)+H_v(\hx,y_2,v_2(\hx),Dv_2(\hx))Dv_2(\hx)\leq 0.
 \end{aligned}
\eeq
{If $s_2(\hx)=0$ we obtain \cref{Eq: Dv2 2} by differentiating in a neighborhood of $\hx$ and pass to the limit.}\par
{Let us consider the case $\frac{\rho}{\theta}-r+\lambda_2 \leq 0$. In this case,}
\beq\label{Ineq: hx}
\lc \frac{\rho}{\theta}-r\rc(r\hx+y_2)^{-1/\psi}+\lambda_2\lc \underbrace{(r\hx+y_2)^{-1/\psi}-(r\hx+y_1)^{-1/\psi}}_{<0}\rc< 0,
\eeq
hence
\beq\label{Ineq: ux}
\frac{\frac{\rho}{\theta}-r+\lambda_2}{\lambda_2}<\lc \frac{r\hx+y_2}{r\hx+y_1}\rc^{1/\p}<\lc \frac{r\ux+y_2}{r\ux+y_1}\rc^{1/\p},
\eeq
where for the last inequality we used $y_2>y_1$, $r\hx>r\ux$, $r\ux+y_j>0$ and $1/\p>0$. It is easy to see that \eqref{Ineq: ux} contradicts \eqref{rho-r}. \par
Now let us consider the case when $\frac{\rho}{\theta}-r+\lambda_2 > 0$. From \Cref{Eq: Dv2 2}, we obtain that
$$
\lc \frac{\rho}{\theta}-r+\lambda_2 \rc (r\hx+y_2)^{-1/\psi}\lc\frac{v_2(\hx)}{v_1(\hx)}\rc^{\frac{1-\g \psi}{\psi(1-\g)}}-\lambda_2 (r\hx+y_1)^{-1/\psi}<0.
$$
Since $v_1(\hx)<v_2(\hx)<0$ and $\frac{1-\g \psi}{\psi(1-\g)}<0$, we observe that
$
\frac{v_2(\hx)}{v_1(\hx)}<1$, thus
$ \lc\frac{v_2(\hx)}{v_1(\hx)}\rc^{\frac{1-\g \psi}{\psi(1-\g)}}>1$.
We recover \eqref{Ineq: hx} and then \eqref{Ineq: ux}, in contradiction with \eqref{rho-r}. We have therefore proved that $s_2(x)<0$ for all $x>\ux$ and conclude $s_2(\ux)=0$, from the state constraint condition and the continuity of $s_2$. 
\end{proof}

By contrast with \Cref{Prop:s2boundary}, the following proposition contains a sufficient condition for the savings of the productive agents at the borrowing limit to be positive. 

\begin{proposition}\label{Prop:s2boundary 2}
Suppose $0\leq r\leq \rho$ and 
\beq\label{rho-r2}
\begin{aligned}
\lc \rho-r\rc(r\ux+y_2)^{-\g}+\lambda_2\lc \underbrace{(r\ux+y_2)^{-\g}-(r\ux+y_1)^{-\g}}_{<0}\rc< 0.
\end{aligned}
\eeq
Then $s_2(\ux)>0$. 
\end{proposition}
\begin{proof}
We argue by contradiction. Suppose $s_2(\ux)=0$, then $s_1(\ux)=0$ and $s_2(\ux)=0$ yield 
\beq\label{Eq: Dv1 Dv2 ux}
 \begin{aligned}
& Dv_1(\ux)=\rho(r\ux+y_1)^{-1/\p}((1-\g)v_1(\ux))^{\frac{1-\g \p}{\p(1-\g)}},\\
& Dv_2(\ux)=\rho(r\ux+y_2)^{-1/\p}((1-\g)v_2(\ux))^{\frac{1-\g \p}{\p(1-\g)}}.
 \end{aligned}
\eeq
From the continuity of $s_2$, we may define 
 $$
 \delta^+=\max \{\delta\geq 0: \quad s_2(x)= 0\,\,\forall x\in [\ux,\ux+\delta]\},
 $$
 and consider different cases.\par 
Case ($I$): If $\delta^+>0$, then for $x\in (\ux,\ux+\delta^+]$
 $$
 Dv_2(x)=\rho(rx+y_2)^{-1/\p}((1-\g)v_2(x))^{\frac{1-\g \p}{\p(1-\g)}}
 $$
We infer from \cref{HJB} and $v_1(\ux)\leq v_2(\ux)$ that
$$
\frac{\rho}{\theta}v_2(\ux)=\mathcal{F}(c_2(\ux),v_2(\ux)) +\underbrace{s_2(\ux)Dv_2(\ux)}_{=0}+\lambda_2(\underbrace{v_{1}(\ux)-v_2(\ux)}_{<0})\leq \mathcal{F}(c_2(\ux),v_2(\ux)),
$$
which yields
$$
v_2(\ux)\leq \frac{(r\ux+y_2)^{1-\gamma}}{1-\gamma}.
$$
Moreover, for $x\in (\ux,\ux+\delta^+]$ we have 
 $$
 \begin{aligned}
 D^2v_2(x)={}&-\frac{r\rho}{\p}(rx+y_2)^{-\frac{1}{\p}-1}((1-\g)v_2(x))^{\frac{1-\g \p}{\p(1-\g)}}\\
 {}&+\frac{\rho(1-\g \p)}{\p}(rx+y_2)^{-1/\p}((1-\g)v_2(x))^{\frac{1-\g \p}{\p(1-\g)}-1}Dv_2(x)\\
\geq  {}&-\frac{r\rho}{\p}(r x+y_2)^{-1-\g}\geq -\frac{r\rho}{\p}(r \ux+y_2)^{-1-\g},
\end{aligned}
 $$
Hence, $\lim_{x\to \ux}s_2(x)D^2v_2(x)=0$ and
$$
\begin{aligned}
\lc \rho-r\rc Dv_2(\ux)+\lambda_2\lc Dv_2(\ux)-Dv_1(\ux)\rc
=\lc 1-\frac{1}{\theta}\rc \rho Dv_2(\ux)+H_v(\ux,y_2,v_2(\ux),Dv_2(\ux))Dv_2(\ux).
\end{aligned}
$$
We can infer with \cref{Eq: f concave} and $v_2(\ux)\leq \frac{(r\ux+y_2)^{1-\gamma}}{1-\gamma}$ that 
$$
H_v(\ux,y_2,v_2(\ux),Dv_2(\ux))= \mathcal{F}_v(c_2(\ux),v_2(\ux))\geq  \mathcal{F}_v\lc c_2(\ux),\frac{(r\ux+y_2)^{1-\gamma}}{1-\gamma}\rc\geq -\lc 1-\frac{1}{\theta}\rc \rho,
$$
hence
$$
\lc \rho-r\rc Dv_2(\ux)+\lambda_2\lc Dv_2(\ux)-Dv_1(\ux)\rc \geq 0.
$$
From \Cref{prop:sub super sol} and the comparison principle (\Cref{comparison}) we know 
$$
v_1(\ux)\geq \frac{(r\ux+y_1)^{1-\gamma}}{1-\gamma}.
$$
We then infer from \cref{Eq: Dv1 Dv2 ux} that 
$$
Dv_2(\ux)\leq \rho(r\ux+y_2)^{-\g},\quad Dv_1(\ux)\geq \rho(r\ux+y_1)^{-\g}.
$$
Therefore, 
\beq\label{Eq: Dv2 bdy}
\begin{aligned}
\lc \rho-r\rc(r\ux+y_2)^{-\g}+\lambda_2\lc (r\ux+y_2)^{-\g}-(r\ux+y_1)^{-\g}\rc\geq 0,
\end{aligned}
\eeq
and this yields a contradiction with \eqref{rho-r2}.
\par
Case ($II$): Now we consider the case $\delta^+=0$.\par
Case ($II1$): We first consider the case that $\delta^+=0$ and there exists a sufficiently small $\eps>0$ such that either $s_2(x)>0$ or $s_2(x)<0$ in $(\ux,\ux+\eps)$.
  \begin{itemize}
 \item[(1):]
 Suppose $s_2(x)<0$ in $(\ux,\ux+\eps)$ then $\liminf_{x\to \ux}s_2(x)D^2v_2(x)\geq 0$, arguing as above, we still obtain 
\eqref{Eq: Dv2 bdy} and hence a contradiction with \eqref{rho-r2}.
 \item[(2):]
 Suppose $s_2(x)>0$ in $(\ux,\ux+\eps)$, we claim that it is impossible that $\lim_{x\to \ux}D^2v_2(x)= -\infty$. If $\lim_{x\to \ux}D^2v_2(x)= -\infty$, there exists $\hat{x}\in (\ux,\ux+\eps)$ such that
\beq\label{Eq: Ds2<0}
 \begin{aligned}
 Ds_2(x)={}&r+\p \rho^{\p}(Dv_2(x))^{-\p-1}((1-\g)v_2(x))^{\frac{1-\g \p}{1-\g}}D^2v_2(x)\\
 &-(1-\g \p)\rho^{\p}(Dv_2(x))^{1-\p}((1-\g)v_2(x))^{\frac{1-\g \p}{1-\g}-1} <0,\,\, \forall x\in (\ux,\hx).
 \end{aligned}
\eeq
Therefore, $s_2(\hx)<0$, in contradiction with $s_2(x)>0$ for $x\in (\ux,\ux+\eps)$. 
Since $\lim_{x\to \ux}D^2v_2(x)= -\infty$ is not possible, we know there exists a sequence $x_n\to \ux$ such that $$\lim_{x_n\to \ux}s_2(x_n)D^2v_2(x_n)= 0$$ and we again obtain \eqref{Eq: Dv2 bdy}, leading to a contradiction with \eqref{rho-r2}.
 \end{itemize}
Case ($II2$): We consider the case that $\delta^+=0$ and there exists $x_n\to \ux$ such that $s_2(x_n)=0$. Again, we claim it is impossible that $\lim_{x_n\to \ux}D^2v_2(x_n)= -\infty$. Otherwise, for sufficiently small $x_n$ we can show, similarly to the calculation in \eqref{Eq: Ds2<0}, that $Ds_2(x)<0$ for all $x\in (\ux,x_n)$, hence $s_2(x_n)<0$. Having ruled out the possibility that $\lim_{x_n\to \ux}D^2v_2(x_n)= -\infty$, we can extract a bounded subsequence of $D^2v_2(x_n)$ and obtain \eqref{Eq: Dv2 bdy}, in contradiction with \eqref{rho-r2}.
\end{proof}
The next result deals with the behavior of $s_2$ as $x\to +\infty$. 
\begin{proposition}\label{Prop:s2 barx}
Suppose $\rho>r$ and $s_2(\ux)>0$. There exists $\bar{x}>\ux$ such that $s_2(\bar{x})=0$ and $s_2(x)<0$ for all $x\geq \bar{x}$.
\end{proposition}
\begin{proof}
Let us assume that for all $\hx$ sufficiently large, $s_2(\hx)>0$ and look for a contradiction. Differentiating the HJB equation for $v_2$ at $\hx$ leads to \eqref{Eq: Dv2}. From $s_2(\hx)>0$, we obtain $c_2(\hx)<r\hx+y_2$. The comparison with subsolution found in \Cref{Prop:sub-super} yields  
$$
v_2(\hx)\geq \frac{(r\hx+y_1)^{1-\gamma}}{1-\gamma}.
$$ 
From the monotonicity of $H_v$ (\eqref{Hvp} and \eqref{Hvv}), we obtain
$$
H_v(\hx,y_2,v_2(\hx),Dv_2(\hx))<-\rho \lc 1-\frac{1}{\theta}\rc \lc \frac{r\hx+y_2}{r\hx+y_1}\rc^{1-\p^{-1}}.
$$
We deduce from \eqref{Eq: Dv2} that
$$
\rho-r+\lambda_2<\rho \lc 1-\frac{1}{\theta}\rc\lc 1- \lc \frac{r\hx+y_2}{r\hx+y_1}\rc^{1-\p^{-1}}\rc+\frac{\lambda_2Dv_1(\hx)}{Dv_2(\hx)}.
$$
From \Cref{Prop: s1<0}, $c_1(\hx)>r\hx+y_1$, hence 
$$
\frac{Dv_1(\hx)}{Dv_2(\hx)}=\lc \frac{c_2(\hx)}{c_1(\hx)}\rc^{1/\p}\lc \frac{v_1(\hx)}{v_2(\hx)}\rc^{\frac{1-\g \p}{\p(1-\g)}}<\lc \frac{r\hx+y_2}{r\hx+y_1}\rc^{\p^{-1}},
$$
therefore
$$
\rho-r<\rho \lc 1-\frac{1}{\theta}\rc\lc 1- \lc \frac{r\hx+y_2}{r\hx+y_1}\rc^{1-\p^{-1}}\rc+\lambda_2 \lc \lc \frac{r\hx+y_2}{r\hx+y_1}\rc^{\p^{-1}}-1\rc.
$$
{
Passing to the limit we obtain
$$
\lim_{\hx\to +\infty}\rho \lc 1-\frac{1}{\theta}\rc\lc 1- \lc \frac{r\hx+y_2}{r\hx+y_1}\rc^{1-\p^{-1}}\rc+\lambda_2 \lc \lc \frac{r\hx+y_2}{r\hx+y_1}\rc^{\p^{-1}}-1\rc=0,
$$
a contradiction with $\rho>r$.} The existence of $\bar{x}$ then follows from the continuity of $s_2$. 
\end{proof}

\section{Analysis of the Fokker--Planck--Kolmogorov equation}\label{Sec: FPK}
Let us state the weak formulation of the system of Fokker-Planck-Kolmogorov equations:
\begin{equation}\label{Eq: FP m}
		- \frac{\partial}{\partial x} \left[ (rx + y_j - c_j(x)) m_j(x) \right]+\lambda_{\bar \jmath}m_{\bar \jmath}(x) -\lambda_jm_j(x)=0 .
\end{equation} 
\begin{definition}\label{Def: FPK sol weak}
The measure $m$ with $m=\sum_{j\in \{1,2\}}m_j\otimes\delta_{y_j}(y)$ is a weak solution to \cref{Eq: FP m} if for all test functions $(\phi_1,\phi_2)\in \left(C_c^1([\ux,+\infty))\right)^2$ and $j\in \{1,2\}$,
\begin{equation}\label{Eq: weak FPK}
\int_{x\geq \ux}\lambda_j \phi_j(x)dm_j-\int_{x\geq \ux}\lambda_{\bar \jmath}\phi_j(x)dm_{\bar \jmath}=\int_{x\geq \ux}s_j(x)D\phi_j(x)dm_j.
\end{equation}
\end{definition}
In the stationary regime, the singularities in the wealth distribution may appear only at $\ux$ or $\hat{x}>\ux$ such that $s_1(\hat{x})=0$ or $s_2(\hat{x})=0$. It has been proven in \Cref{Prop: s1<0} that $s_1$ does not vanish in $(\ux,+\infty)$. On the other hand, the asymptotics analysis at the points where $s_2$ vanishes show that, with the optimal strategy,  it takes an infinite  time to reach these points. Therefore, the wealth distribution of the population labeled $j$ has a density $g_j$ in $(\ux,+\infty)$ and may exhibit a Dirac mass at $\ux$, of the form $\mu_j \delta_{\ux}$. We can then write the probability measure as 
\beq\label{Eq: def m}
dm_j=g_jdx+\mu_j\delta_{\ux}.
\eeq
 We denote by $G_j(x)$ the cumulative distribution function: 
\beq\label{Eq: G}
G_j(x):=\mu_j +\int_{\ux}^xg_j(x)dx.
\eeq
We may rewrite \cref{Eq: weak FPK} as follows 
\begin{equation*}
 \begin{aligned}
&\int_{x>\ux}\left(\lambda_j g_j(x)-\lambda_{\bar \jmath} g_{\bar \jmath}(x)\right)\phi_j(x)dx+(\lambda_j\mu_j-\lambda_{\bar \jmath}\mu_{\bar \jmath})\phi_j(\ux)\\
={}&\int_{x>\ux}s_j(x)g_j(x)D\phi_j(x)dx+\mu_js_j(\ux)D\phi_j(\ux).
 \end{aligned}
\end{equation*}
\begin{lemma}
Given any $r$ such that $0\leq r<\rho$, let $m=\sum_{j\in \{1,2\}}m_j\otimes\delta_{y_j}(y)$ be a weak solution to the FPK equation. Then we have 
$$
\lim_{x\to +\infty}G_j(x)=\int_{\ux}^{+\infty}dm_j=\mu_j +\int_{\ux}^{+\infty}g_j(x)dx=\frac{\ld_{\bar \jmath}}{\ld_j+\ld_{\bar \jmath}}.
$$
\end{lemma}
\begin{proof}
By taking the test functions $\phi_j=1$ in \cref{Eq: weak FPK} we have
\beq\label{Eq: G1 G2}
\ld_1\int_{\ux}^{+\infty}dm_1=\ld_2\int_{\ux}^{+\infty}dm_2,
\eeq
we then obtain \eqref{Eq: G1 G2} from $\int_{\ux}^{+\infty}dm_1+\int_{\ux}^{+\infty}dm_2=1$. 
\end{proof}

Let $\bar{x}$ be defined as in \Cref{Prop:s2 barx}. We first observe $g_1(x)=g_2(x)=0$ for all $x$ such that $x>\bar{x}$, because the optimal policies consist of decumulating wealth when $x>\bar{x}$, so the stationary distribution must vanish in this region (see \cite{meyn2012markov}). 

\begin{lemma}\label{Lemma: sjgj}
Assume $r$ satisfies $0\leq r<\rho$ and $s_2(\ux)>0$. Then we have 
\beq\label{Eq: sjgj}
s_1(x)g_1(x)+s_2(x)g_2(x)=0\quad \forall x>\ux.
\eeq
\end{lemma}
\begin{proof}
From the FPK equation we can derive for all $x>\ux$,
$$
\frac{d}{dx}\lc s_1(x)g_1(x)+s_2(x)g_2(x)\rc =0.
$$
Moreover for $x>\bar{x}$, $s_1(x)g_1(x)+s_2(x)g_2(x)=0$, therefore we obtain \eqref{Eq: sjgj}.
\end{proof}

If $s_2(\ux)>0$, from \Cref{Prop:s2 barx} and the continuity of $s_2$, we know that there exists $\widehat{x}$
\beq\label{Eq: def hx}
\widehat{x}=\min_{x}\{x:\,\, s_2(x)=0\}.
\eeq
It is clear that $\widehat{x}\leq \bar{x}$ and $s_2(\widehat{x})=0$. Moreover, $s_2(x)>0$ for all $x<\widehat{x}$.
We recall that $\bar{x}$ in \Cref{Prop:s2 barx} is the last point where $s_2$ vanishes, i.e. $s_2(x)<0$ for all $x>\bar{x}$. 
\begin{proposition}\label{Prop: FPK}
Assume $s_2(\ux)>0$ and let $\widehat{x}$ be defined by \eqref{Eq: def hx}. Then, there exists $\kappa_2$ such that the densities are given by 
\beq\label{Eq: g2}
{\text{For all}}\,\, x\in (\ux,\widehat{x}),\quad g_2(x)=\frac{\kappa_2}{s_2(x)}\exp\lc \int_{\ux}^x\lc -\frac{\ld_1}{s_1(z)}-\frac{\ld_2}{s_2(z)}\rc dz\rc,
\eeq
\beq\label{Eq: g1}
{\text{For all}}\,\, x\in (\ux,\widehat{x}),\quad g_1(x)=-\frac{\kappa_2}{s_1(x)}\exp\lc \int_{\ux}^x\lc -\frac{\ld_1}{s_1(z)}-\frac{\ld_2}{s_2(z)}\rc dz\rc.
\eeq
\end{proposition}
\begin{proof}
For all $x\in (\ux,\widehat{x})$, $s_2(x)>0$ and $s_1(x)<0$, we can write
\beq\label{Eq: sg}
\frac{d}{dx}\lc s_2(x)g_2(x)\rc=\lc -\frac{\ld_1}{s_1(x)}-\frac{\ld_2}{s_2(x)}\rc \lc s_2(x)g_2(x)\rc.
\eeq
From $s_2(\ux)>0$ and $s_1(x)= O(\sqrt{x-\underline x})$ near $\ux$ (see \ref{Expansion ux}), we infer $-\frac {\lambda_1} {s_1(x)}-\frac {\lambda_2} {s_2(x)}$ is integrable in a neighborhood of $\ux$, which allows us to integrate (4.9) in $[\underline x, x)$  for any $x\in [\underline x, \widehat{x})$.  We obtain that there exists $\kappa_2>0$ such that for all $x\in  [\underline x, \widehat{x})$,
$$g_2(x) =\frac{\kappa_2} {s_2(x)}\exp\lc{\int_{\ux}^x-\frac {\lambda_1} {s_1(z)} -\frac {\lambda_2} {s_2(z)}dz}\rc.$$
It is clear that $\kappa_2=g_2(\underline x) s_2(\underline x)$.  We then deduce \cref{Eq: g1} in $( \underline x,\widehat{x})$ from \cref{Eq: sjgj}.
\end{proof}

\begin{proposition}\label{Prop:s2 hatx}
Assume $s_2(\ux)>0$ and let $\widehat{x}$ be defined by \eqref{Eq: def hx}. Then, $g_j(x)=0$ for all $x> \widehat{x}$. 
\end{proposition}
\begin{proof} The densities satisfy $g_j(x)=0$ on the set $\{x:\,\widehat{x}<x<\bar{x},\,s_2(x)<0\}$. This follows from \Cref{Prop: s1<0} and \Cref{Lemma: sjgj}.\par
 Next, we show the densities $g_j(x)=0$ on the set $\{x:\,\widehat{x}<x<\bar{x},\,s_2(x)>0\}$. This is an open set, i.e. a countable or finite union of disjoint intervals $I_k=(a_k,b_k)$. We know that, if $a_k\neq  \bar{x}$ then
\beq\label{Eq: s Ds=0}
s_2(a_k)=s_2(b_k)=0,\quad Ds_2(a_k)\geq 0.
\eeq
{If $s_2(a_k)=Ds_2(a_k)=0$, then $s_2(x)=o(x-a_k)$. Therefore, \cref{Eq: s Ds=0} implies $s_2(x)=O(x-a_k)$ as $x\to a_{k,+}$.} We now choose $\xi\in I_k$ and denote $C_k=s_2(\xi)g_2(\xi)$. If $g_2(\xi)>0$, then by using \cref{Eq: sg}, we obtain
 \beq\label{Eq: g Ik}
  \begin{aligned}
 &g_2(x)=\frac{C_k}{s_2(x)}\exp\lc \int_{\xi}^x\lc -\frac{\ld_1}{s_1(z)}-\frac{\ld_2}{s_2(z)}\rc dz\rc,\\
 &g_1(x)=-\frac{C_k}{s_1(x)}\exp\lc \int_{\xi}^x\lc -\frac{\ld_1}{s_1(z)}-\frac{\ld_2}{s_2(z)}\rc dz\rc.
  \end{aligned}
 \eeq
{We note that $\xi>a_k$, $s_1(x)$ is bounded away from $0$ if $x>\widehat{x}$. We thus infer that $\exp\lc \int_{\xi}^x\lc -\frac{\ld_1}{s_1(z)}\rc dz\rc$ is bounded and $\exp\lc \int_{\xi}^x\lc -\frac{\ld_2}{s_2(z)}\rc dz\rc\geq 0$ for all $x$ in a right neighborhood of $a_k$. We then deduce from $s_2(x)=O(x-a_k)$ that the density $g_2$ proposed in \cref{Eq: g Ik} is not integrable near $a_{k,+}$.} Therefore $C_k=0$ and $g_2(x)=0$ for all $x\in I_k$. Finally, \cref{Eq: sjgj} and $s_1(x)<0$ give $g_1(x)=0$ for all $x\in I_k$.\par
Finally, we consider $\check{x}$, such that $s_2(x)=0$ in a neighborhood of $\check{x}$. We deduce from \cref{Eq: sjgj}, $s_2(\check{x})=0$ and $s_1(\check{x})<0$ that $g_1(\check{x})=0$. From \cref{eq:FP}, we obtain $g_2(\check{x})=0$. 
\end{proof}
\begin{corollary}
Assume $s_2(\ux)>0$. We have 
$$
\kappa_2=\ld_1\mu_1. 
$$
\end{corollary}
\begin{proof}
Note that $\mu_2=0$ because $s_2(\ux)>0$. From \eqref{Eq: G1 G2} and $\mu_2=0$, we deduce
$$
-\ld_1\mu_1=\ld_1\int_{\ux}^{\widehat{x}}g_1(x)dx-\ld_2\int_{\ux}^{\widehat{x}}g_2(x)dx.
$$
By integrating the FPK equation for $g_2$, we then obtain
$$
-\ld_1\mu_1=\int_{\ux}^{\widehat{x}+1}\frac{d}{dx}\lc s_2(z)g_2(z)\rc dz= s_2(\widehat{x}+1)g_2(\widehat{x}+1)-\lim_{x\to \ux}s_2(x)g_2(x).
$$
From \Cref{Prop:s2 hatx}, $s_2(\widehat{x}+1)g_2(\widehat{x}+1)=0$. The result follows from \eqref{Eq: g2}.
\end{proof}

\begin{proposition}\label{Prop: Kr}
 There exists $\hat{r}>0$, for all $r$ such that $\hat{r}\leq r< \rho$, the aggregate wealth $\mathcal{K}[r]$ is positive and depends continuously on $r$. 
\end{proposition}
\begin{proof}
Let $r^\i \to r$ with $r< \rho$, we have $s^\i _1(x)<0$ and $s^\i_2(\ux)>0$. From \Cref{Prop: stability r} we know that $s^\i_j$ converges to $s_j$ locally uniformly. For any $x> \widehat{x}$ such that $s_2(x)<0$, we can obtain $s^\i_j(x)<0$ with $\iota$ sufficiently large. For any $M>\widehat{x}$, for $\iota$ sufficiently large, the measure $m^\i_j$ is supported in $[\ux,M]$. Hence we can extract a subsequence $(m^\i_1,m^\i_2)$ which converges weakly in the sense of measure and {the limit $(\tilde{m}_1,\tilde{m}_2)$ is supported in $[\ux,\widehat{x}]$.} By passing to the limit in the weak form of FPK equation \cref{Eq: weak FPK}, we obtain that for all test functions $(\phi_1,\phi_2)\in \left(C_c^1([\ux,M))\right)^2$,
\begin{equation}\label{Eq: weak sol FPK}
\int_{x\geq \ux}\lambda_j \phi_j(x)d\tilde{m}_j-\int_{x\geq \ux}\lambda_{\bar \jmath}\phi_j(x)d\tilde{m}_{\bar \jmath}=\int_{x\geq \ux}s_j(x)D\phi_j(x)d\tilde{m}_j.
\end{equation}
The solution of \cref{Eq: weak sol FPK} is unique, namely given $(m_1,m_2)$ defined by \cref{Eq: def m}. Therefore, the whole sequence $(m^\i_1,m^\i_2)$ converges weakly to $(m_1,m_2)$. Recalling \cref{Huggett}, we see that 
$$
\mathcal{K}[r^\i]=\sum_{j\in \{1,2\}}\int_{\ux}^Mxdm^\i_j,
$$
and conclude that $\mathcal{K}[r^\i]$ converges to $\mathcal{K}[r]$.
\end{proof}
With similar arguments as above, we can study the situation when $s_2(\underline{x})= 0$:
\begin{proposition}
If $s_2(\underline x)=0$, then $m_1= \frac {\lambda_2}{\lambda_1+\lambda_2} \delta_{\underline x}$ and $m_2= \frac {\lambda_1}{\lambda_1+\lambda_2} \delta_{\underline x}$.
\end{proposition}
\section{Existence of solution to the Mean Field Game system}\label{Sec: MFG}
\subsection{Nonexistence of invariant measures when $r=\rho$}
{In \cref{Sec: FPK} we have shown that the aggregate wealth depends continuously on the interest rate $r$. We now show that the aggregate wealth blows up as $r\to \rho$. For this purpose, we consider the limiting case $r=\rho$.  For $b$ defined in \cref{tab1} and $b=\rho$, the sub and supersolutions given in \Cref{Prop:sub-super} become:}
\beq\label{Eq: sub- super r=rho}
\begin{aligned}
&(\check{\su}_1,\check{\su}_1)=\lc \frac{(\rho x+y_1)^{1-\gamma}}{1-\gamma},\frac{(\rho x+y_1)^{1-\gamma}}{1-\gamma}\rc,\\
&(\check{\sv}_2,\check{\sv}_2)=\lc \frac{(\rho x+y_2)^{1-\gamma}}{1-\gamma},\frac{(\rho x+y_2)^{1-\gamma}}{1-\gamma}\rc.
\end{aligned}
\eeq
\begin{proposition}
If $r=\rho$, then $s_2(x)>0$ for all $x\geq \ux$. 
\end{proposition}
\begin{proof}
It is clear from \Cref{Prop:s2boundary 2} that $s_2(\ux)>0$ if $r=\rho$.
Next we argue by contradiction and suppose $s_2(\hx)\leq 0$ for some $\hx>\ux$. \par
{\it{Step 1}}. Suppose $s_2(\hx)< 0$. From \Cref{Prop:v W2}, we can differentiate \Cref{Eq: Dv1} and \Cref{Eq: Dv2} at $\hx$:
\beq\label{Eq: Dv1 rho}
 \begin{aligned}
&\lambda_1\lc Dv_1(\hx)-Dv_2(\hx)\rc\\
={}&s_1(\hx)D^2v_1(\hx)+\lc 1-\frac{1}{\theta}\rc \rho Dv_1(\hx)+H_v(\hx,y_1,v_1(\hx),Dv_1(\hx))Dv_1(\hx),
 \end{aligned}
\eeq
\beq\label{Eq: Dv2 rho}
 \begin{aligned}
&\lambda_2\lc Dv_2(\hx)-Dv_1(\hx)\rc\\
={}&s_2(\hx)D^2v_2(\hx)+\lc 1-\frac{1}{\theta}\rc \rho Dv_2(\hx)+H_v(\hx,y_2,v_2(\hx),Dv_2(\hx))Dv_2(\hx).
 \end{aligned}
\eeq
From \cref{Hvp}, \cref{Hvv}, \cref{Eq: sub- super r=rho} and $s_2(\hx)< 0$ we can infer
$$
H_v(\hx,y_2,v_2(\hx),Dv_2(\hx))\geq -\lc 1-\frac{1}{\theta}\rc \rho ,
$$
hence
\beq\label{Eq: Dv2 rho r}
\lc 1-\frac{1}{\theta}\rc \rho Dv_2(\hx)+H_v(\hx,y_2,v_2(\hx),Dv_2(\hx))Dv_2(\hx)\geq 0. 
\eeq
Inequality \eqref{Eq: Dv2 rho r}, with $s_2(\hx)D^2v_2(\hx)\geq 0$ and \Cref{Eq: Dv2 rho} lead to $Dv_2(\hx)\geq Dv_1(\hx)$. This and $v_2(\hx)> v_1(\hx)$ yield
$$
H_v(\hx,y_1,v_1(\hx),Dv_1(\hx))>H_v(\hx,y_2,v_2(\hx),Dv_2(\hx))\geq -\lc 1-\frac{1}{\theta}\rc \rho,
$$
hence 
$$
\lc 1-\frac{1}{\theta}\rc \rho Dv_1(\hx)+H_v(\hx,y_1,v_1(\hx),Dv_1(\hx))Dv_1(\hx)>0.
$$
From \Cref{Eq: Dv1 rho}, we obtain $s_1(\hx)D^2v_1(\hx)<0$. Since $D^2v_1(\hx)<0$, this is in contradiction with $s_1(\hx)<0$ (obtained in \Cref{Prop: s1<0}).\par
{\it{Step 2}}. Suppose $s_2(\hx)=0$, we can differentiate \Cref{Eq: Dv2} in a neighborhood of $\hx$. Moreover, from \Cref{Prop:v W2} we have $\lim_{x\to \hx}s_2(x)D^2v_2(x)=0$. We proceed similarly as in {\it{Step 1}} and omit the details. 
\end{proof}
The following proposition describes the asymptotics of the savings in the limit $x\to +\infty$.  The proof is in \ref{Expansion details}.
\begin{proposition}\label{Prop: sj asmp}
Assume that $r=\rho$. As $x\to +\infty$, 
\beq\label{Eq: sj asmp}
   \begin{aligned}
  &s_j(x)\\
  ={}&\frac{\ld_j(y_j-y_{\bar \jmath})}{\rho+\lambda_j+\lambda_{\bar \jmath}}+o(1)\\
  ={}&\frac{\g (1+\p)\lambda_j(y_{\bar \jmath}-y_j)^2}{2(\rho+\lambda_j+\lambda_{\bar \jmath})}\left[1-\frac{\lambda_j\lambda_{\bar \jmath}+\lambda_{\bar \jmath}^2+\rho \lambda_j}{(\rho+\lambda_j+\lambda_{\bar \jmath})^2}-\frac{\lambda_j}{(\rho+\lambda_j+\lambda_{\bar \jmath})}\right](\rho x+y_j)^{-1}\\
&+\frac{\lambda_j(y_j-y_{\bar \jmath})}{\rho+\lambda_j+\lambda_{\bar \jmath}}+o((\rho x+y_j)^{-1}).
   \end{aligned}
  \eeq
\end{proposition}

{\Cref{Prop: sj asmp} allows us to show that when $r=\rho$, there is no stationary measure. The following result is proved in \ref{Expansion details}.}
\begin{proposition}\label{Prop: no sol FPK}
Assume that $r=\rho$.  Let  $v_j$ solve the HJB equations (1.4). The FPK equation (4.1) does not have a solution. In other words, there is no stationary probability measure  when $r=\rho$.
\end{proposition}
In the next result, we show $\mathcal{K}(r)$ becomes unbounded as $r\to \rho$. 
\begin{corollary}\label{Prop: K blows}
For any constant $C_{\mathcal{K}}>0$, there exists $\bar{r}\in [0,\rho)$ such that the aggregate wealth 
\beq\label{Eq: K bar}
\mathcal{K}(\bar{r})>C_{\mathcal{K}}.
\eeq
\end{corollary}
\begin{proof}
Suppose \cref{Eq: K bar} does not hold and we consider a sequence $r^\i\to \rho$, $r^\i<\rho$, such that 
$$
\mathcal{K}(r^\i)\leq C_{\mathcal{K}}.
$$
It is easy to obtain for all $\iota$,
$$
\int_{x\geq 0} \vert x\vert dm^\i_1+\int_{x\geq 0} \vert x\vert dm^\i_2\leq C_{\mathcal{K}}+\vert \ux\vert,\quad \int_{x\geq \ux} \vert x\vert dm^\i_1+\int_{x\geq \ux} \vert x\vert dm^\i_2\leq C_{\mathcal{K}}+2\vert \ux\vert.
$$
It follows that the sequence of probability measures $m^\i$ is tight (cf. \cite[Proposition D.5.2 and Lemma D.5.3]{meyn2012markov}). From Prokhorov’s theorem, we can extract a weakly convergent subsequence of $m^\i$. By passing to the limit in the weak form of FPK equation \cref{Eq: weak FPK}, we obtain a solution to \cref{Eq: weak FPK} with $r=\rho$ and a density defined by \cref{Eq: def m}, in contradiction with \Cref{Prop: no sol FPK}.
\end{proof}

\subsection{Existence of solutions to the MFG system}
We now consider the existence of solutions to the Aiyagari model. 
\begin{theorem}
Suppose 
\beq\label{cond exis}
  \frac{\rho}{\theta \lambda_2}>\lc \frac{y_2}{y_1}\rc^{1/\psi}-1.
\eeq
There exists a solution to the Aiyagari model \eqref{MFG}-\eqref{Aiyagari} with the equilibrium interest rate $r^*$ such that $0<r^*<\rho$. 
\end{theorem}
\begin{proof}
We have shown, with \Cref{Prop:Lip}, the existence and uniqueness of the solution to the HJB equation and optimal saving policy $s_j^{(r)}$ for each $r$. We recall the notations for the aggregate asset $K[m]$ and $\mathcal{K}(r)$ given by 
\cref{def K N} and \cref{Huggett}. We rewrite the equilibrium condition for the Aiyagari model \eqref{Aiyagari} as 
$$
r^*=\mathcal{B}(r^*),\quad {\text{s.t.}}\,\, \mathcal{B}(r)=A \alpha \left( \frac{\mathcal{K}(r)}{N} \right)^{\alpha - 1} - \delta.
$$
It follows directly from \Cref{Prop: Kr} that $\mathcal{B}(r)$ depends continuously on $r$. Let us take 
$$
C_{\mathcal{K}}=\left(\frac{\delta}{A\alpha}\right)^{\frac{1}{\alpha-1}}N.
$$ 
From \Cref{Prop: K blows}, there exists $\bar{r}>0$ such that $\mathcal{K}(\bar{r})>C_{\mathcal{K}}$, hence $\bar{r}-\mathcal{B}(\bar{r})>0$.\par
From condition \eqref{cond exis}, there exists $\hat{r}>0$ sufficiently small such that
\beq\label{bar r}
 \frac{\rho-\hat{r}\theta}{\theta \lambda_2}> \lc \frac{\hat{r}\ux+y_2}{\hat{r}\ux+y_1}\rc^{1/\psi}-1.
\eeq
We use $s^{(\hat{r})}_j$ to denote the saving policies corresponding to $r=\hat{r}$. With \Cref{Prop: s1<0} and \Cref{Prop:s2boundary}, we have $s^{(\hat{r})}_j(\ux)=0$, $s^{(\hat{r})}_j(x)< 0$ for all $x>\ux$. This gives $\mathcal{K}(\hat{r})=\ux\leq 0$. From the continuity of $\mathcal{K}(r)$, given by \Cref{Prop: Kr}, there exists $r_0$, $\hat{r}\leq r_0<\rho$ such that $\mathcal{K}(r_0)=0$, hence $r_0-\mathcal{B}(r_0)=-\infty$. From the intermediate value theorem, there exists $r^*\in (r_0,\bar{r})$ such that $r^*-\mathcal{B}(r^*)=0$. 
\end{proof}
The analysis of the Huggett model is very similar and we only state the result.  

\begin{theorem}\label{Thm: Huggett}
Assume $\ux<B$ and \cref{cond exis}, then there exists a solution to the Huggett model \eqref{MFG}-\eqref{Huggett} with the equilibrium interest rate $r^*$ such that $0<r^*<\rho$. 
\end{theorem}

\section{Numerical examples}\label{Sec: numerical}
We report on some numerical tests with the Aiyagari model, fixing 
\beq
\begin{aligned}
&\text{Discount}: \rho=0.05,\,\,\text{income}: y_1=0.1,\,\,y_2=0.5,\\
&\text{Transition rates}: \lambda_1=0.4,\,\,\lambda_2=0.4,\,\,\text{Debt limit}:  \ux=-0.15.\\
&\text{TFP}: A=0.95,\,\,\alpha=0.35,\,\,\text{Depreciation}: \delta=0.1.
\end{aligned}
\eeq
A detailed description of the numerical methods is given in \cite{AchdouTang2026}. We considered four different cases. Hereafter, we enumerate the different cases and the related interest rate $r^*$ at equilibrium:
\begin{itemize}
\item {\textit{Test 1:}} $\g=2$, $\psi=0.8$, $r^*=0.034$.
\item {\textit{Test 2:}} $\g=2$, $\psi=0.4$, $r^*=0.0246$.
\item {\textit{Test 3:}} $\g=4$, $\psi=0.4$, $r^*=0.018$.
\item {\textit{Test 4:}} $\g=1.2$, $\psi=0.4$, $r^*=0.02737$.
\end{itemize}
It is important to notice that only Test 2 and Test 4 satisfy the assumption $\g \p<1$. Test 1 and Test 3 do not actually fall into the theoretical framework of the present paper. Yet, we observe that our numerical algorithms continue to perform well in these settings \cite{AchdouTang2026}. \par
For comparison, in the CRRA case with $\g=1/\p=2$ we obtain $r^*=0.027942$.\par
{\bf{Interpretations}}:
\begin{itemize}
\item Fixing $\g$, $r^*$ increases as $\psi$ increases. As agents are more willing to consume, the aggregate capital decreases. 
\item Fixing $\psi$, $r^*$ decreases as $\g$ increases. As agents are more risk averse, they favor more precautionary savings near $\ux$. \end{itemize}
We now plot the consumption and saving policies. For plotting the asset distributions, we truncated the upper range using a percentile-based threshold to avoid that the high values related to the Dirac mass hide the other ones. In the plots, we use solid lines for results from Test 1 and Test 3, dotted lines for Test 2 and Test 4 (\cref{c12,c34,s12,s34,g12,g34}). 
 \begin{figure}[h!]
   \centering
    \includegraphics[width=0.55\linewidth]{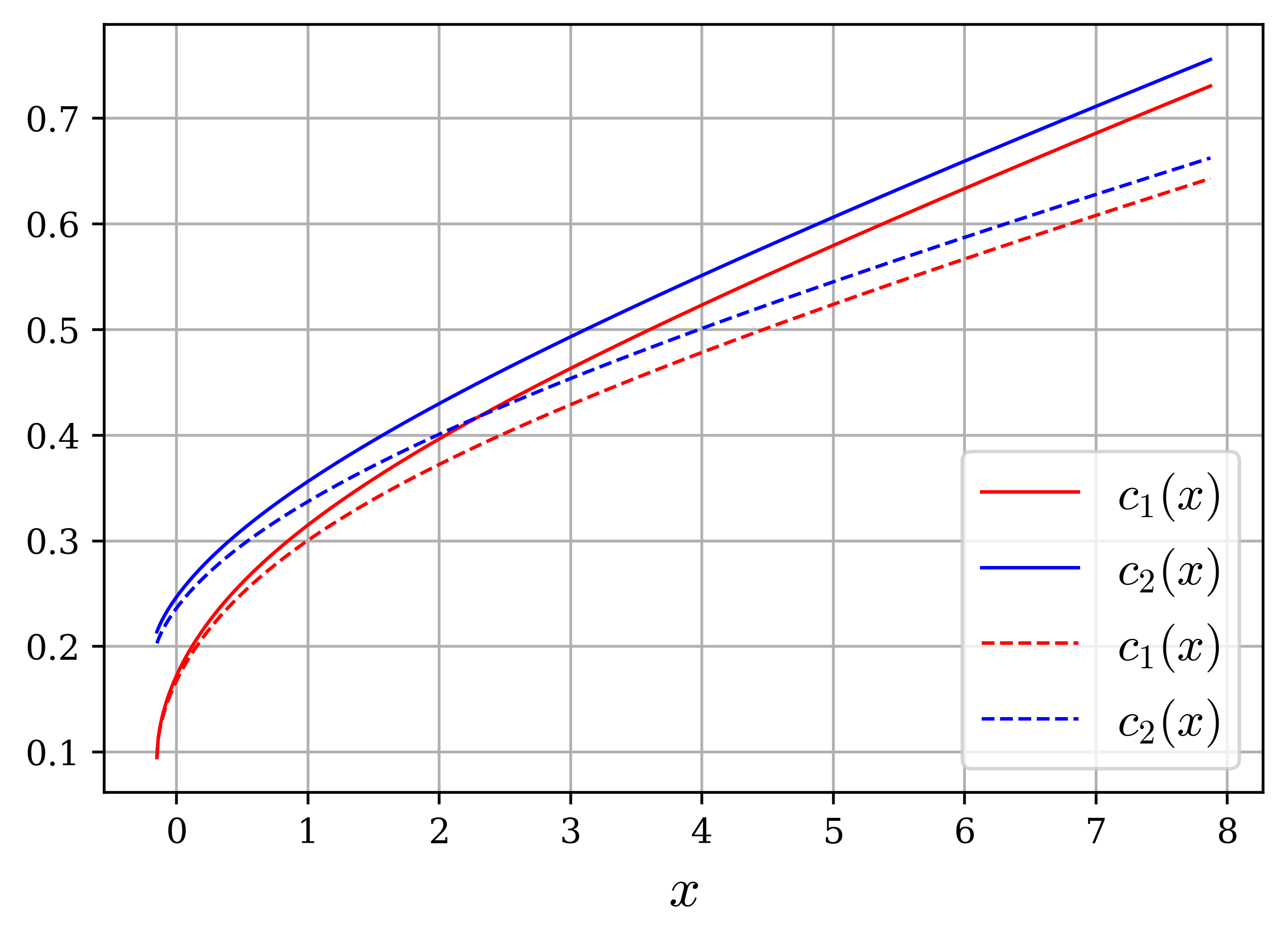}
        \caption{Consumption with $\g=2$: {\textit{Test 1}} and {\textit{Test 2}}}\label{c12}
\end{figure}

\begin{figure}[h!]
   \centering
    \includegraphics[width=0.55\linewidth]{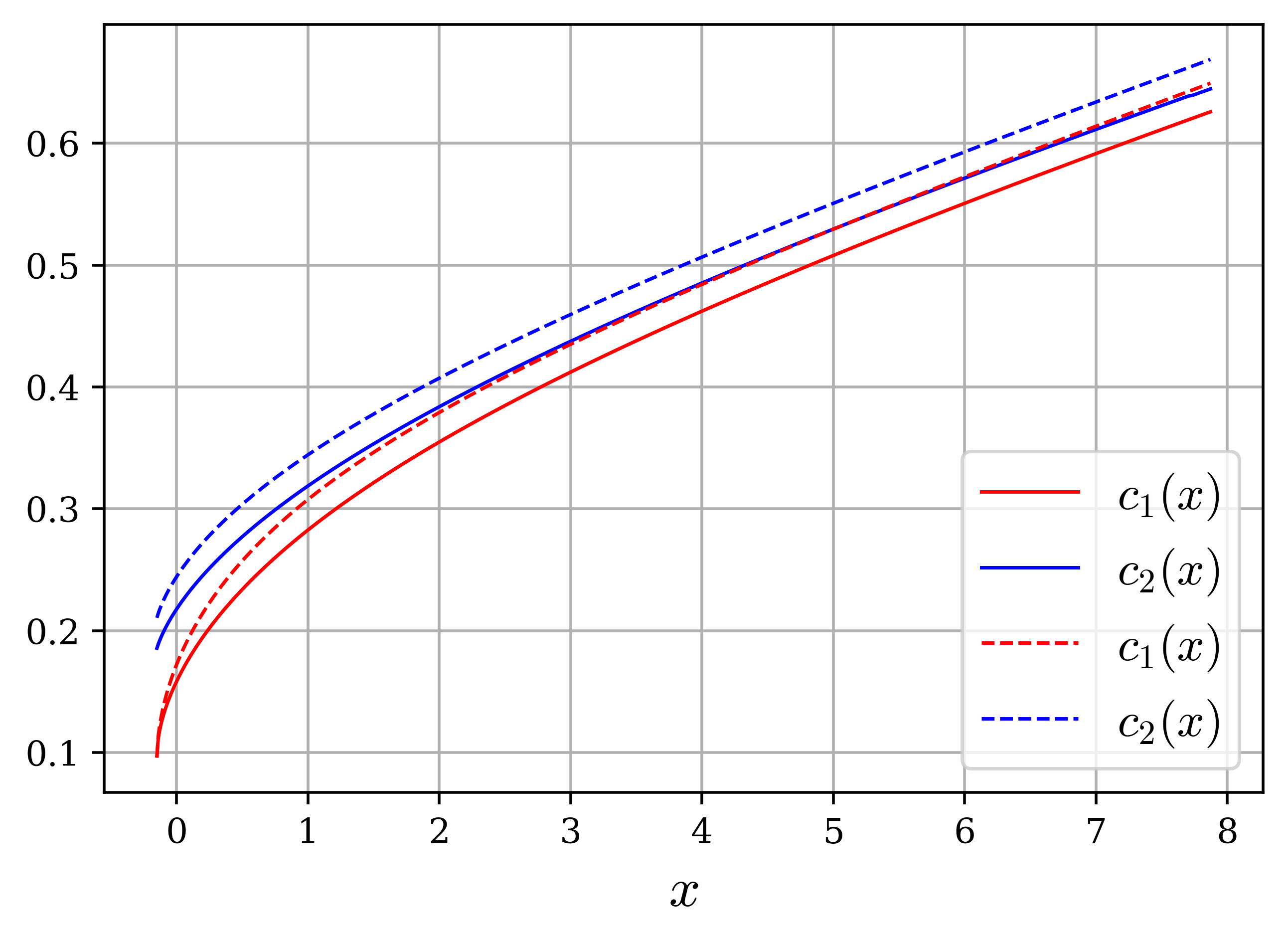}
      \caption{Consumption with $\psi=0.4$: {\textit{Test 3}} and {\textit{Test 4}}}\label{c34}
\end{figure}

\begin{figure}[h!]
   \centering
    \includegraphics[width=0.55\linewidth]{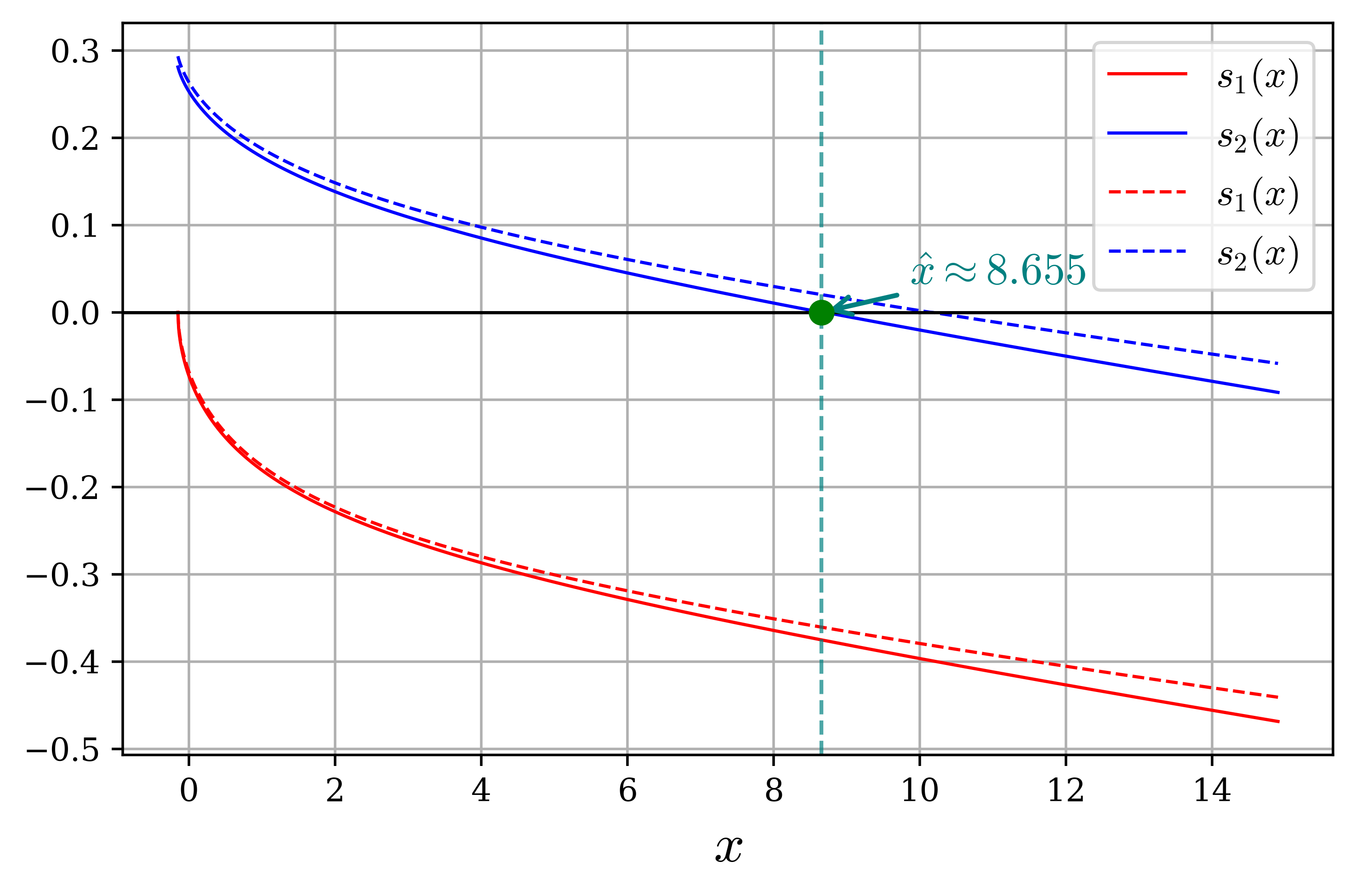}
        \caption{Saving with $\g=2$: {\textit{Test 1}} and {\textit{Test 2}}}\label{s12}
\end{figure}
\begin{figure}[h!]
   \centering
    \includegraphics[width=0.55\linewidth]{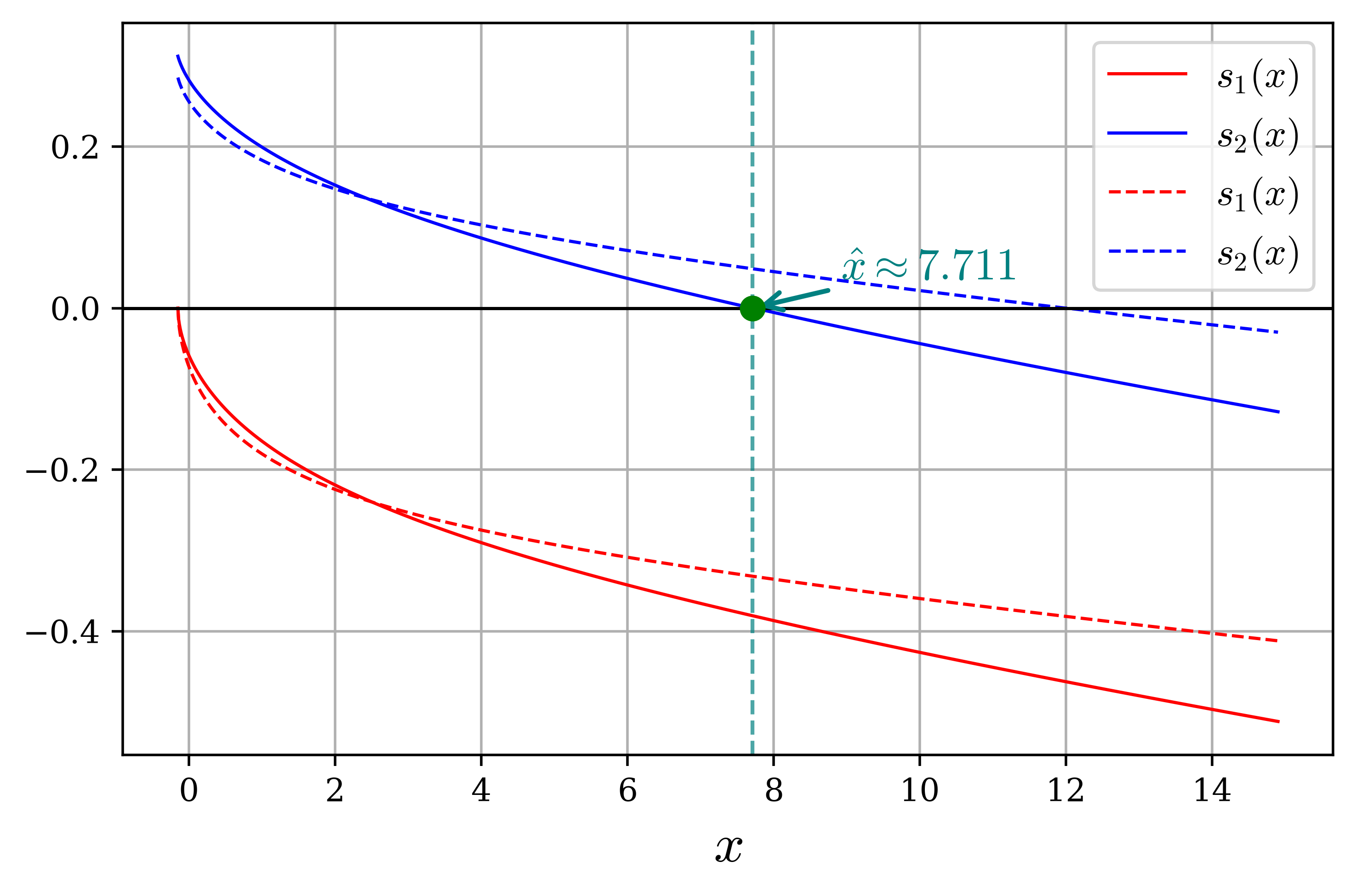}
    \caption{Saving with $\psi=0.4$: {\textit{Test 3}} and {\textit{Test 4}}}\label{s34}
\end{figure}

\begin{figure}[h!]
   \centering
    \includegraphics[width=0.55\linewidth]{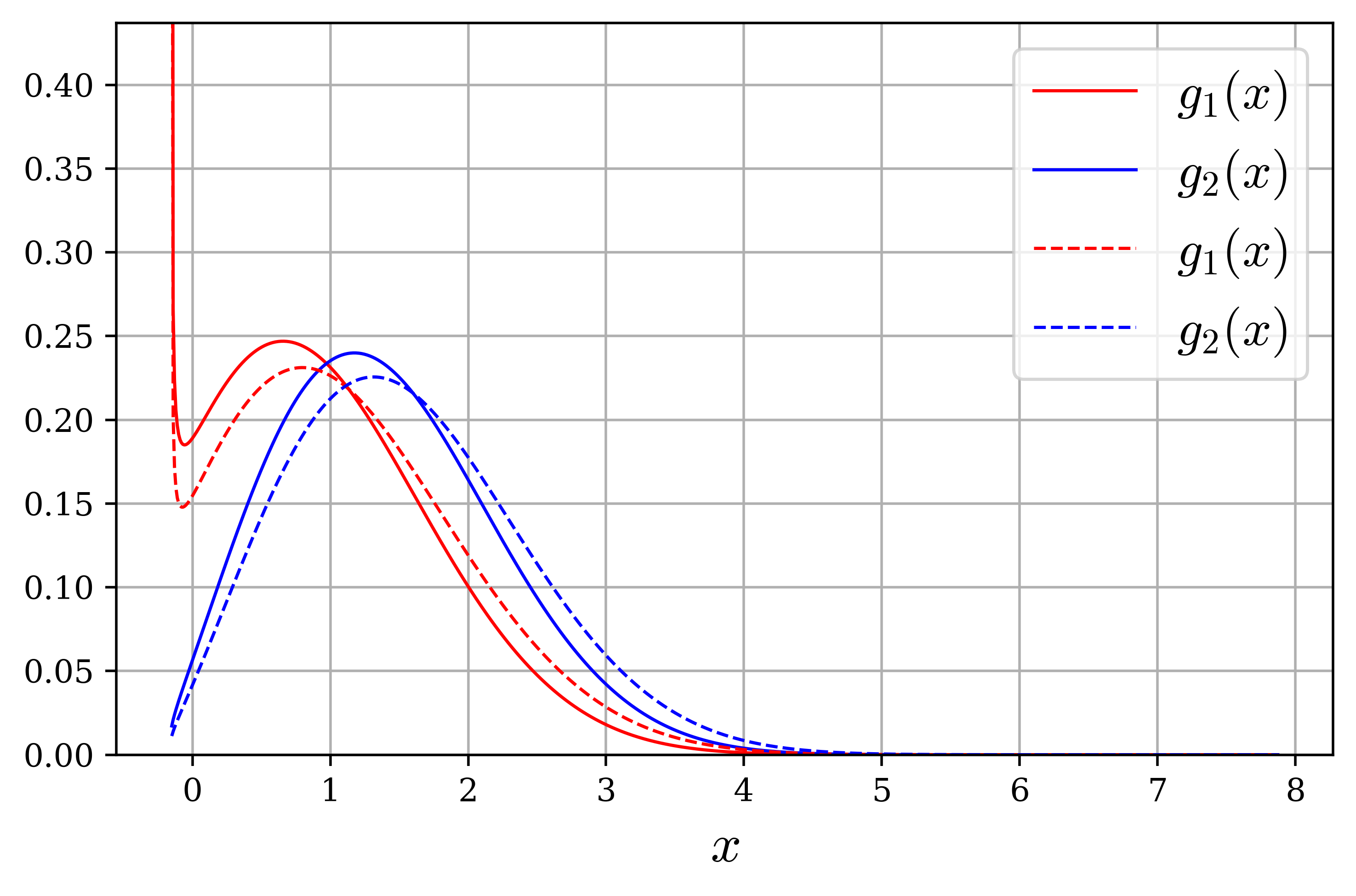}
    \caption{Asset distribution with $\g=2$: {\textit{Test 1}} and {\textit{Test 2}}}\label{g12}
\end{figure}
\begin{figure}[!htbp]
   \centering
    \includegraphics[width=0.55\linewidth]{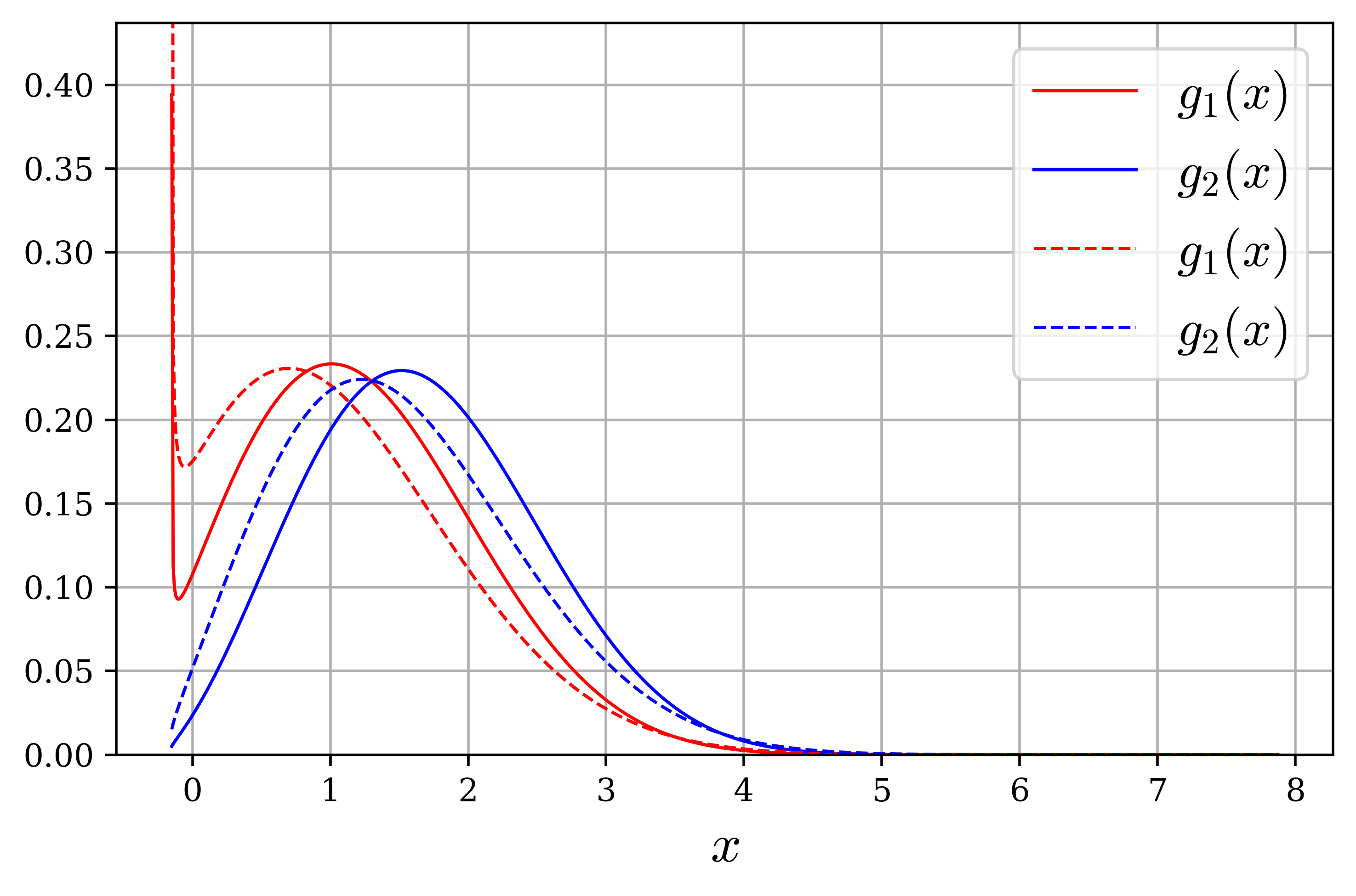}
     \caption{Asset distribution $\psi=0.4$: {\textit{Test 3}} and {\textit{Test 4}}}\label{g34}
\end{figure}

We observe in \cref{s34} that, given the same $\p$, agents in the model with higher risk aversion $\g$ (Test 3) exhibit higher savings when their asset level $x$ is close to the borrowing limit $\ux$. This occurs despite the fact that the equilibrium $r^*$ is substantially lower in Test 3 than in Test 4. Savings in Test 4 eventually exceed those in Test 3 as $x$ moves away from $\ux$. This pattern suggests that, near the borrowing constraint, the precautionary savings motive dominates the effect of interest rate differences, whereas the latter becomes more influential at higher asset levels.

\section{Discussion and future works}
In the near future, we plan to address the other situation, namely  $\gamma\psi>1$,   when the agents favor early resolution of uncertainty. Many of the arguments contained in the present paper will have to be modified, and we plan to build on the existing literature on recursive utilities in discrete time.\par
Another direction of our future research concerns the analysis proposed in \cref{Sec: MFG} and \ref{Expansion details} in the limit situation when $r=\rho$,  that is based on a second order expansion of the saving policies as $x\to +\infty$. This method is specific to the two states income process in the present paper. It would be desirable to find a robust method that would work for general income processes. This may require using alternative arguments from the literature on the stability of Markov chains, see e.g. \cite{meyn2012markov}.\par

Finally, the present paper considered models where there is no aggregate uncertainty impacting the economy globally. To consider aggregate uncertainty as in the Krusell-Smith model \cite{krusell1998income}, one needs to study a master equation on the space of probability measures\cite{cardaliaguet2019master}. The numerical simulations of the master equation in continuous-time heterogeneous agent models have been addressed in the recent literature \cite{bilal2023solving,gu2024global}. It would be interesting to investigate the master equation arising in models with recursive utility.

\begin{appendices}

\section{Proof of \Cref{comparison}}\label{Comparison proof}
\begin{proof}
We assume by contradiction that 

\beq\label{Eq:contra}
\max_j\sup_x(\mathsf{u}_j(x)-\mathsf{v}_j(x))=\delta>0.
\eeq
{\it{Step 1}}. First we consider the case when the $\sup$ in \eqref{Eq:contra} is achieved at $\ux$, i.e.
\beq\label{Eq:contra-bd}
\max_j\sup_x(\mathsf{u}_j(x)-\mathsf{v}_j(x))=\max_j(\mathsf{u}_j(\ux)-\mathsf{v}_j(\ux))=\delta.
\eeq
The supersolution is defined only in $(\ux,+\infty)$, but from its lower semicontinuity we can extend it to $\ux$ with $\mathsf{v}_j(\ux)=\liminf_{\substack{z\rightarrow \ux,\, z>\ux}} \mathsf{v}_j(z)$. {There exists a sequence $\zeta_k$ such that 
\beq\label{Eq: zeta_k}
\mathsf{v}_j(\zeta_k)\to \mathsf{v}_j(\ux)\quad {\text{as}}\quad \zeta_k\to \ux,\quad j\in \{1,2\}.
\eeq}
We denote $\e_k=\vert \zeta_k-\ux\vert$. Consider the   function
\beq
\psi_k(j,x,z)=\mathsf{u}_j(x)-\mathsf{v}_j(z)-\frac{\vert x-z\vert^2}{\e_k}-\left[\left(\frac{z-x}{\e_k}-1\right)_{-}\right]^2-\vert z-\ux\vert^2.
\eeq
Let $\psi_k$ attain its maximum at $(j_{k}, x_{k},z_{k})$. Since $\psi_k(j_{k},x_{k},z_{k})\geq \psi_k(j_{\bar{\jmath}_{k}},x_{k},z_{k})$, we obtain
\beq\label{v1 v2}
\mathsf{u}_{\bar{\jmath}_{k}}(x_{k})-\mathsf{u}_{j_{k}}(x_{k})\leq \mathsf{v}_{\bar{\jmath}_{k}}(z_{k})-\mathsf{v}_{j_{k}}(z_{k}).
\eeq
We now show 
\beq\label{psi_k}
\psi_k(j_{k},x_{k},z_{k})\geq \delta-o(1)>0. 
\eeq

From
$
\left[\left(\frac{\zeta_k-\ux}{\e_k}-1\right)_{-}\right]^2=0,
$
we have 
$$
\psi_k(j_{k}, x_{k},z_{k})\geq \max_j\psi_k(j,\ux,\zeta_k) =\max_j(\mathsf{u}_j(\ux)-\mathsf{v}_j(\zeta_k))-\vert \ux-\zeta_k\vert-\vert \zeta_k-\ux\vert^2.
$$
From \cref{Eq: zeta_k} we obtain
$
\max_j\psi_k(j,\ux,\zeta_k) =\delta-o(1)
$
and therefore \cref{psi_k}. From $\psi_k(j_{k}, x_{k},z_{k})> 0$ and the boundedness of $\mathsf{u}_{j_{k}}(x_{k})$ and $\mathsf{v}_{j_{k}}(z_{k})$ there exists a constant $C>0$ such that 
$
\frac{\vert x_{k}-z_{k}\vert^2}{\e_k}<C,
$
hence $x_{k}-z_{k}\to 0$ as $\e_k\to 0$. From \cref{psi_k} we deduce 
$$
\liminf_{k\to +\infty}(\mathsf{u}_{j_{k}}(x_{k})-\mathsf{v}_{j_{k}}(z_{k}))\geq \liminf_{k\to +\infty}\psi_k(j_{k},x_{k},z_{k})\geq \delta. 
$$
On the other hand, since $\mathsf{u}_{j_{k}}(x_{k})-\mathsf{v}_{j_{k}}(z_{k})$ is u.s.c., 
$$
\limsup_{k\to +\infty}(\mathsf{u}_{j_{k}}(x_{k})-\mathsf{v}_{j_{k}}(z_{k}))\leq \max_j(\mathsf{u}_j(\ux)-\mathsf{v}_j(\ux))=\delta,
$$
hence
$
\lim_{k\to +\infty}(\mathsf{u}_{j_{k}}(x_{k})-\mathsf{v}_{j_{k}}(z_{k}))=\delta. 
$
We then obtain 
\begin{equation*}
\frac{\vert x_{k}-z_{k}\vert^2}{\e_k}+
\left[\left(\frac{z_{k}-x_{k}}{\e_k}-1\right)_{-}\right]^2+
\vert z_{k}-\ux\vert^2\to 0,\qquad \text{as}\quad \e_k\to 0.
\end{equation*}
This gives 
$
z_{k}-x_{k}\geq \e_k-\e_ko(1),
$
which implies $z_{k}>\ux$, hence we can use the supersolution property at $z_{k}$. We set
$
\Lambda_k=\frac{2}{\e_k}\left(\frac{z_{k}-x_k}{\e_k}-1\right)_{-}.
$
Since $\psi_k(j_{k},x,z)$ achieves its maximum at $(x_{k},z_{k})$, we have
\beq\label{Eq: sub super ineq}
	\left\{\begin{aligned}
\quad  \frac{\rho}{\theta} \mathsf{u}_{j_{k}}(x_{k}) \leq {}&H\left(x_{k},y_{j_{k}},\mathsf{u}_{j_{k}}(x_{k}),\frac{2( x_{k}-z_{k} )}{\e_k}+\Lambda_k \right)+\lambda_{j_{k}}\left(\mathsf{u}_{\bar{\jmath}_{k}}(x_{k}) -\mathsf{u}_{j_{k}}(x_{k}) \right),\\
\quad \frac{\rho}{\theta}  \mathsf{v}_{j_{k}}(z_{k}) \geq {}& H\left(z_{k},y_{j_{k}},\mathsf{v}_{j_{k}}(z_{k}) ,\frac{2( x_{k}-z_{k} )}{\e_k}+\Lambda_k-2(z_{k}-\ux) \right)\\
&+\lambda_{j_{k}}\left(\mathsf{v}_{\bar{\jmath}_{k}}(z_{k})-\mathsf{v}_{j_{k}}(z_{k}) \right).
	\end{aligned}\right.
\eeq
From the coercivity of $H$ and boundedness of $\mathsf{v}_{j_{k}}$, it follows that $\frac{2( x_{k}-z_{k} )}{\e_k}+\Lambda_k-2(z_{k}-\ux)$ is positive and bounded. Subtracting the two inequalities in \cref{Eq: sub super ineq} and using \cref{v1 v2}, for $k$ sufficiently large, we have 
\beq\label{Eq: u-v1}
\begin{aligned}
&\frac{\rho}{\theta}  (\mathsf{u}_{j_{k}}(x_{k})-\mathsf{v}_{j_{k}}(z_{k}))\\
\leq {}&\underbrace{H\left(x_{k},y_{j_{k}},\mathsf{u}_{j_{k}}(x_{k}),\frac{2( x_{k}-z_{k} )}{\e_k}+\Lambda_k \right)-H\left(x_{k},y_{j_{k}},\mathsf{v}_{j_{k}}(z_{k}),\frac{2( x_{k}-z_{k} )}{\e_k}+\Lambda_k \right)}_{(I)<0}\\
&+H\left(x_{k},y_{j_{k}},\mathsf{v}_{j_{k}}(z_{k}),\frac{2( x_{k}-z_{k} )}{\e_k}+\Lambda_k \right)\\
&-H\left(z_{k},y_{j_{k}},\mathsf{v}_{j_{k}}(z_{k}),\frac{2( x_{k}-z_{k} )}{\e_k}+\Lambda_k-2(z_{k}-\ux) \right)\\
\leq {}& \underbrace{r(x_{k}-z_{k})\left(\frac{2( x_{k}-z_{k} )}{\e_k}+\Lambda_k\right)}_{(II)<0}+2(rz_{k}+y_{j_{k}})(z_{k}-\ux)\\
+&\left[\underbrace{\frac{\rho^{\p}}{\p-1}\left(\frac{2( x_{k}-z_{k} )}{\e_k}+\Lambda_k \right)^{1-\p}
-\frac{\rho^{\p}}{\p-1}\left(\frac{2( x_{k}-z_{k} )}{\e_k}+\Lambda_k-2(z_{k}-\ux) \right)^{1-\p}}_{(III)<0}\right]\\
&\cdot ((1-\g)\mathsf{v}_{j_{k}}(z_{k}))^{\frac{1-\g \p}{1-\g}}\\
\leq {}& 2(rz_{k}+y_{j_{k}})(z_k-\ux).
\end{aligned}
\eeq
In \eqref{Eq: u-v1}, $(I)$ follows from \eqref{Eq:Hv<0} and $
\lim_{k\to +\infty}(\mathsf{u}_{j_{k}}(x_{k})-\mathsf{v}_{j_{k}}(z_{k}))=\delta 
$, $(II)$ follows from $z_k>x_k$ and $(III)$ follows from the monotonicity of $(p^{1-\p})/(\p-1)$. We then obtain $\mathsf{u}_{j_{k}}(x_{k})-\mathsf{v}_{j_{k}}(z_{k})\to 0$ as $\e_k\to 0$, in contradiction with \eqref{psi_k}. We have shown that \eqref{Eq:contra-bd} cannot hold. \par
{\it{Step 2}}. Now we consider the case when the supremum of $\mathsf{u}_j(x)-\mathsf{v}_j(x)$ is not achieved at $\ux$, i.e. 
\beq\label{Eq:contra-bd2}
\max_j\sup_x(\mathsf{u}_j(x)-\mathsf{v}_j(x))=\delta>\max_j(\mathsf{u}_j(\ux)-\mathsf{v}_j(\ux)).
\eeq
Let us define the function $\mathtt{h}(x)=\f{1}{2}\ln (1+\vert x-\ux \vert^2)$. From \eqref{Eq:contra-bd2}, one can show that for a sufficiently small $\beta>0$, we can find $0<\delta_2\leq \delta$ such that  
 \beq\label{Eq: delta2}
 \max_j\sup_x(\mathsf{u}_j(x)-\mathsf{v}_j(x)-\beta \mathtt{h}(x))>\delta_2>\max_j(\mathsf{u}_j(\ux)-\mathsf{v}_j(\ux)).
 \eeq
We now consider the function
\beq
\Phi(j,x,z)=\mathsf{u}_j(x)-\mathsf{v}_j(z)-\frac{\vert x-z\vert^2}{\e}-\beta \mathtt{h}(x).
\eeq
Since $\Phi(j,x,z)$ is u.s.c., and $\beta \mathtt{h}(x)\to +\infty$ as $x\to +\infty$, we can assume that $\Phi(j,x,z)$ achieves its maximum at some $(j_\e,x_\e,z_\e)$ such that $x_\e,z_\e<+\infty$ if $\beta>0$. Now we fix $\beta$ such that \cref{Eq: delta2} holds and 
\beq\label{Eq: beta}
\beta(\rho+y_2)<\frac{\delta_2 \rho}{4\theta},
\eeq
where $\theta$ is defined in \cref{tab1}. It is obvious that 
$$
\Phi(j_\e,x_\e,z_\e)\geq \max_j\sup_x(\mathsf{u}_j(x)-\mathsf{v}_j(x)-\beta \mathtt{h}(x)),
$$
hence \cref{Eq: delta2} yields
 $$\Phi(j_\e,x_\e,z_\e)>\delta_2\,\,{\text{for all}}\, \,\eps.$$
The boundedness of $\mathsf{u}_j$ and $\mathsf{v}_j$ yields
$
\vert x_\e-z_\e \vert\to 0$ as $\e\to 0$.
Classically, we then infer that
$
\frac{\vert x_\e-z_\e \vert^2}{\e}\to 0$ as $\e\to 0$.
We then extract subsequences $x_{\e_k}$ and $z_{\e_k}$ that both converge to some $x^*$ and such that $j_{\e_k}$ converges to $j^*$. 
From the semicontinuity of $\mathsf{u}_j$ and $\mathsf{v}_j$, with \eqref{Eq: beta}, we get 
\beq\label{Eq:contra-3}
\su_{j^*}(x^*)-\sv_{j^*}(x^*)\geq \limsup_k(\mathsf{u}_{j_{\e_k}}(x_{\e_k})-\mathsf{v}_{j_{\e_k}}(z_{\e_k}))\geq \Phi(j_{\e_k},x_{\e_k},z_{\e_k})>\delta_2.
\eeq
It follows from \cref{Eq: delta2} and \cref{Eq:contra-3} that $x^*>\ux$. 
For brevity we write the sequence $(j_{k},x_k,y_k)$ instead of $(j_{\e_k},x_{\e_k},y_{\e_k})$. \par
From \Cref{def vis sol} we obtain
\begin{equation}\label{Eq: u-v2}
\begin{aligned}
&\frac{\rho}{\theta} (\su_{j_k}(x_k)-\sv_{j_k}(z_k) )\\
\leq {}&\underbrace{H\left(x_k,y_{j_{k}},\mathsf{u}_{j_{k}}(x_k),\f{2(x_k-z_k)}{\e_k}+\beta D\mathtt{h}(x_k)\right)-H\left(x_k,y_{j_k},\mathsf{v}_{j_{k}}(z_{k}),\f{2(x_k-z_k)}{\e_k}+\beta D\mathtt{h}(x_k)\right)}_{(I)<0}\\
&+H\left(x_{k},y_{j_{k}},\mathsf{v}_{j_{k}}(z_{k}),\f{2(x_k-z_k)}{\e_k}+\beta D\mathtt{h}(x_k)\right)-H\left(z_{k},y_{j_{k}},\mathsf{v}_{j_{k}}(z_{k}),\f{2(x_k-z_k)}{\e_k}\right)\\
\leq {}& r(x_{k}-z_{k})\left(\f{2(x_k-z_k)}{\e_k}\right)+\underbrace{\beta(rx_{k}+y_{j_{k}}) D\mathtt{h}(x_k)}_{(II)}\\
+&\left[\underbrace{\frac{\rho^{\p}}{\p-1}\left(\f{2(x_\e-z_\e)}{\e}+\beta D\mathtt{h}(x_k)\right)^{1-\p}
-\frac{\rho^{\p}}{\p-1}\left(\f{2(x_k-z_k)}{\e_k}\right)^{1-\p}}_{(III)<0}\right]\\
&\cdot((1-\g)\mathsf{v}_{j_{k}}(z_{k}))^{\frac{1-\g \p}{1-\g}}\\
\leq {}& \f{2\rho(x_k-z_k)^2}{\e_k}+\frac{\delta_2 \rho}{4\theta}.
\end{aligned}
\end{equation}
In \eqref{Eq: u-v2}, $(I)$ and $(III)$ are dealt with in the same way as in \eqref{Eq: u-v1}, cf. {\it{Step 1}}. It follows from \eqref{Eq: beta}, with $D\mathtt{h}(x_k)<1$ and $x_kD\mathtt{h}(x_k)\leq 1$, that $(II)<\frac{\delta_2 \rho}{4\theta}$.
Passing to the limit yields $\su_{j^*}(x^*)-\sv_{j^*}(z^*)<\delta_2$, in contradiction with \eqref{Eq:contra-3}. \par
Combining {\it{Step 1}} and {\it{Step 2}}, we have shown that \eqref{Eq:contra} cannot hold, leading to the desired result. 
\end{proof}

\section{Asymptotic analysis as $x\to +\infty$ in the case $r=\rho$}\label{Expansion details}
\subsection{Proof of \Cref{Prop: sj asmp}}
\begin{proof} From \Cref{Prop: v2 v1} and \cref{Eq: sub- super r=rho}, we know that if $r=\rho$ then 
\beq\label{Eq: v12 r=rho}
\frac{\lc \rho x+y_1\rc^{1-\g}}{(1-\g)}\leq v_1(x)<v_2(x)\leq \frac{\lc \rho x+y_2\rc^{1-\g}}{(1-\g)}\quad \forall x>\ux.
\eeq
From $s_1(x)<0$ when $x>\ux$, we deduce
$$
\rho x+y_1< \rho^\p(Dv_1(x))^{-\psi}((1-\g)v_1(x))^{\frac{1-\g \p}{1-\g}}\leq\rho^\p(Dv_1(x))^{-\psi}(\rho x+y_2)^{1-\g \p},
$$
\beq\label{Eq: Dv1 rho2}
0<Dv_1(x)<\rho (\rho x+y_1)^{-\g}\lc \frac{\rho x+y_2}{\rho x+y_1}\rc^{\p^{-1}-\g}.
\eeq
Similarly, 
\beq\label{Eq: Dv2 rho2}
Dv_2(x)>\rho (\rho x+y_2)^{-\g}\lc \frac{\rho x+y_1}{\rho x+y_2}\rc^{\p^{-1}-\g}.
\eeq
We then proceed in 2 steps.\\
{\it{Step 1}}: From \cref{Eq: v12 r=rho}, \cref{Eq: Dv1 rho2} and \cref{Eq: Dv2 rho2}, we may look for a first order expansion of $v_j$ in the form
$$
v_j(x)=\frac{\lc \rho x+y_j\rc^{1-\g}}{(1-\g)}+\sz_j(x),\,\,
\sz_j(x)=O(x^{-\g})\quad {\text{and}}\quad D\sz_j(x)=O(x^{-1-\g}).
$$
From 
$$
 \frac{\lc \rho x+y_2\rc^{1-\g}}{(1-\g)}-\frac{\lc \rho x+y_1\rc^{1-\g}}{(1-\g)}- \lc \rho x+y_1\rc^{-\g}(y_2-y_1)=O(x^{-1-\g}),
$$
we infer
 $$
 \begin{aligned}
 &\lc \frac{\rho}{\theta}+\lambda_j\rc v_j(x)-\lambda_jv_{\bar \jmath}(x)\\
 ={}& \frac{\rho\lc \rho x+y_j\rc^{1-\g}}{1-\p^{-1}}+\lambda_j\lc \rho x+y_j\rc^{-\g}(y_j-y_{\bar \jmath})+\lc\frac{\rho}{\theta}+\lambda_j\rc \sz_j(x)-\lambda_j\sz_{\bar \jmath}(x)+O(x^{-1-\g}).
 \end{aligned}
 $$
 On the other hand, from \Cref{def H} we deduce 
\beq
\begin{aligned}
&H(x,y_j,v_j(x),Dv_j(x))\\
={}&(\rho x+y_j)Dv_j(x)+\frac{\rho^{\p}}{\p-1}\lc Dv_j(x)\rc^{1-\p}((1-\g)v_j(x))^{\frac{1-\g \p}{1-\g}}\\
={}&\rho\lc \rho x+y_j\rc^{1-\g}+(\rho x+y_j)D\sz_j(x)\\
{}&+\underbrace{\frac{\rho^\p}{\p-1}\lc \rho\lc \rho x+y_j\rc^{-\g}+D\sz_j(x)\rc^{1-\p}((1-\g)v_j(x))^{\frac{1-\g \p}{1-\g}}}_{(I)}\\
={}&\frac{\rho\lc \rho x+y_j\rc^{1-\g}}{1-\p^{-1}}+\lc \frac{1}{\theta}-1\rc \rho \sz_j(x)+O(x^{-1-\g}).
\end{aligned}
\eeq
To obtain the expansion for ($I$), we observe
$$
((1-\g)v_j(x))^{\frac{1-\g \p}{1-\g}}= (\rho x+y_j)^{1-\g \p}\lc 1+(1-\g)\frac{\sz_j(x)}{(\rho x+y_j)^{1-\g}}\rc^{\frac{1-\g \p}{1-\g}},
$$
then, expanding to the order $O(x^{-1-\g})$, we obtain
 $$
 \begin{aligned}
& \frac{\rho^\p}{\p-1}\lc \rho\lc \rho x+y_j\rc^{-\g}+D\sz_j(x)\rc^{1-\p}(\rho x+y_j)^{1-\g \p}\lc 1+(1-\g)\frac{\sz_j(x)}{(\rho x+y_j)^{1-\g}}\rc^{\frac{1-\g \p}{1-\g}}\\
&\sim \frac{\rho}{\p-1}\lc \rho x+y_j\rc^{1-\g}\lc 1+\rho^{-1}\lc \rho x+y_j\rc^{\g}D\sz_j(x)\rc^{1-\p}\lc 1+(1-\g \p)\frac{\sz_j(x)}{(\rho x+y_j)^{1-\g}}\rc\\
&\sim \frac{\rho\lc \rho x+y_j\rc^{1-\g}}{\p-1}-\lc \rho x+y_j\rc D\sz_j(x)+\underbrace{\frac{1-\g \p}{\p-1}}_{=1/\theta-1}\rho\sz_j(x).
 \end{aligned}
  $$
  From the HJB equation for $v_j$ we obtain
  \beq\label{Eq:z}
 \lc \rho+\lambda_j\rc \sz_j(x)-\lambda_j\sz_{\bar \jmath}(x)= \lambda_j\lc \rho x+y_j\rc^{-\g}(y_{\bar \jmath}-y_j)+O(x^{-1-\g}).
  \eeq
Solving \eqref{Eq:z} up to terms of order $O(x^{-\g})$, we obtain that $\sz_j(x)\sim \widehat{\sz}_j(x)$ as $x\to +\infty$, where
  \beq\label{Eq: z_j}
  \widehat{\sz}_j(x)= \frac{\lambda_j \lc \rho x+y_j\rc^{-\g}(y_{\bar \jmath}-y_j)}{\rho+\lambda_j+\lambda_{\bar \jmath}}. \eeq
  Note that
  \beq\label{Eq: Dz_j}
  D\widehat{\sz}_j(x)= -\frac{\rho \g \lambda_j \lc \rho x+y_j\rc^{-1-\g}(y_{\bar \jmath}-y_j)}{\rho+\lambda_j+\lambda_{\bar \jmath}}.
  \eeq
  {\it{Step 2}}: Second order expansion. Let us set 
$$
v_j(x)=\frac{\lc \rho x+y_j\rc^{1-\g}}{(1-\g)}+\widehat{\sz}_j(x)+\sq_j(x),
$$
where, from the previous step, $\sq_j(x)=O(x^{-1-\gamma})$ as $x\to +\infty$. Unfortunately, we do not have yet any better information on $D \sq_j(x)$ than the estimate $D\sq_j(x)=O(x^{-1-\gamma})$.

The HJB equation for $v_j(x)$ becomes
\beq\label{Eq: HJB q_j}
\begin{aligned}
&\frac{\rho}{\theta}\lc\frac{\lc \rho x+y_j\rc^{1-\g}}{(1-\g)}+\widehat{\sz}_j(x)+\sq_j(x)\rc\\
={}&\lc \rho x+y_j\rc \lc \rho \lc \rho x+y_j\rc^{-\g}+D\widehat{\sz}_j(x)+D\sq_j(x)\rc\\
+{}&\underbrace{\frac{\rho \lc \rho x+y_j\rc^{1-\g}}{\p-1}\lc 1+\frac{\lc \rho x+y_j\rc^{\g}}{\rho}(D\widehat{\sz}_j(x)+D\sq_j(x))\rc^{1-\p}\lc 1+(1-\g)\frac{\widehat{\sz}_j(x)+\sq_j(x)}{(\rho x+y_j)^{1-\g}}\rc^{\frac{1-\g \p}{1-\g}}}_{(I)}\\
+{}&\lambda_j\lc \lc\frac{\lc \rho x+y_{\bar \jmath}\rc^{1-\g}}{(1-\g)}+\widehat{\sz}_{\bar \jmath}(x)+\sq_{\bar \jmath}(x)\rc-\lc\frac{\lc \rho x+y_j\rc^{1-\g}}{(1-\g)}+\widehat{\sz}_j(x)+\sq_j(x)\rc\rc.
\end{aligned}
\eeq
We aim at  simplifying \cref{Eq: HJB q_j} by using the a priori information on the behavior of $\sq_j$ at infinity.
We start with the term ($I$): we observe that
$$
\begin{aligned}
  &\left( 1 +\frac {(\rho x + y_j)^\gamma}{\rho} (D\widehat{\sz}_j(x)+D\sq_j(x) ) \right)^{1-\psi}\\
  = & 1+ \frac {1-\psi }{\rho} (\rho x + y_j)^\gamma (D\widehat{\sz}_j(x)+D\sq_j(x) )
    - \frac {(1-\psi)\psi }{2 \rho^2}(\rho x + y_j)^{2\gamma} (D\widehat{\sz}_j(x)+D\sq_j(x) )^2 \\
    {}&+ O(x^{-3}),
    \end{aligned}
 $$   
 and that
 $$
\begin{aligned}
  &\left(  1 + (1-\gamma) (\rho x + y_j)^{\gamma-1} (\widehat{\sz}_j(x)+\sq_j(x) ) \right)^{\frac {1-\gamma \psi} {1-\gamma}}\\
  =&{} 1+(1-\gamma\psi) (\rho x + y_j)^{\gamma-1}  (\widehat{\sz}_j(x)+\sq_j(x) )
  +\frac \gamma 2  (1-\gamma\psi)(1-\psi)
  (\rho x + y_j)^{2\gamma-2} \widehat{\sz}_j^2 (x) +O(x^{-3}).
\end{aligned}
$$
Hence the term ($I$) in \cref{Eq: HJB q_j} satisfies 
\begin{align*}
 (I)&=\frac {\rho \lc \rho x+y_j\rc^{1-\gamma}} {\psi-1}\cdot\\
{}&\left(
    \begin{array}[c]{l}
   \displaystyle    1 + \frac {1-\psi }{\rho} (\rho x + y_j)^\gamma (D\widehat{\sz}_j(x)+D\sq_j(x) )
      + (1-\gamma\psi)    (\rho x + y_j)^{\gamma-1}  (\widehat{\sz}_j(x)+\sq_j(x) )
      \\
    \displaystyle       - \frac {(1-\psi)\psi }{2 \rho^2}(\rho x + y_j)^{2\gamma} (D\widehat{\sz}_j(x)+D\sq_j(x) )^2 +\frac \gamma 2  (1-\gamma\psi)(1-\psi)
      (\rho x + y_j)^{2\gamma-2} \widehat{\sz}_j^2 (x)
     \\ 
       \displaystyle  +   \frac {1-\psi }{\rho} (1-\gamma\psi) 
      (\rho x + y_j)^{2\gamma-1}    (\widehat{\sz}_j(x)+\sq_j(x) )   (D\widehat{\sz}_j(x)+D\sq_j(x) )
      + O(x^{-3})
    \end{array}
\right)\\
  ={}&
                 \frac \rho {\psi-1}  \lc \rho x+y_j\rc^{1-\gamma}+
 \lc \frac{\rho}{\theta}-\rho\rc \lc \widehat{\sz}_j(x)+\sq_j(x)\rc\
                 -\lc \rho x+y_j\rc (D\widehat{\sz}_j(x)+D\sq_j(x))    \\
&+\underbrace{\frac{\p}{2\rho}\lc \rho x+y_j\rc^{\gamma +1} \vert D\widehat{\sz}_j(x)\vert^2}_{(I1)}
-\underbrace{(1-\g \p)\lc \rho x+y_j\rc^{\g}\widehat{\sz}_j(x)D\widehat{\sz}_j(x)}_{(I2)}\\
&-\underbrace{\frac{\rho(1-\g\p)\g}{2}\lc \rho x+y_j\rc^{\g-1}\vert \widehat{\sz}_j(x)\vert^2}_{(I3)}                       
                +\frac{\p\lc \rho x+y_j\rc^{\gamma +1}}{2\rho} \left(   \vert D{\sq}_j(x)\vert^2 + 2 D\widehat{\sz}_j(x) D{\sq}_j(x)\right)\\
                {}&-(1-\g \p)\lc \rho x+y_j\rc^{\g}\widehat{\sz}_j(x)D\sq_j(x)
                 +O(x^{-\gamma -2}).
\end{align*}
   From \cref{Eq: z_j} and \cref{Eq: Dz_j}, we obtain
$$
(I1)-(I2)-(I3)=\frac{\rho \g}{2}\lc 
\frac{\lambda_1}{\rho+\lambda_1+\lambda_2}\rc^2\lc \rho x+y_1\rc^{-1-\g}(y_2-y_1)^2.
 $$
We also notice that $\widehat{\sz}_1$ satisfies the equation
  $$
  \begin{aligned}
  &\lc \rho+\lambda_1\rc \widehat{\sz}_1(x)-\lambda_1\widehat{\sz}_2(x)\\
  ={}& \lambda_1\lc \rho x+y_1\rc^{-\g}(y_2-y_1)+ \frac{\ld_1\ld_2\g(\rho x+y_1)^{-1-\g}(y_2-y_1)^2}{\rho+\ld_1+\ld_2}+O(x^{-2-\g}),
  \end{aligned}
  $$
 and  we observe that
  $$
    \begin{aligned}
 &\frac{\lc \rho x+y_2\rc^{1-\g}}{(1-\g)}-\frac{\lc \rho x+y_1\rc^{1-\g}}{(1-\g)}\\
 ={}&\lc \rho x+y_1\rc^{-\g}(y_2-y_1)-\frac{\g}{2}\lc \rho x+y_1\rc^{-1-\g}(y_2-y_1)^2+O(x^{-2-\g}).
   \end{aligned}
   $$
By using the equalities above,  \cref{Eq: HJB q_j} becomes: 
\beq\label{Eq: sq1}
   \begin{aligned}
&\lc \rho+\lambda_1\rc\sq_1-\lambda_1\sq_2    -\frac{\p}{2\rho}\lc \rho x+y_1\rc^{\gamma +1} \left(   \vert D{\sq}_1(x)\vert^2 + 2 D\widehat{\sz}_1(x) D{\sq}_1(x)\right)\\
&+(1-\g \p)\lc \rho x+y_1\rc^{\g}\widehat{\sz}_1(x)D\sq_1(x)\\
={}&{-\frac{\lambda_1\g}{2}}\lc \rho x+y_1\rc^{-1-\g}(y_2-y_1)^2\lc 1-\frac{\rho \lambda_1}{(\rho+\lambda_1+\lambda_2)^2}-\frac{2\lambda_2}{\rho+\lambda_1+\lambda_2}\rc+O(x^{-2-\g}),
    \end{aligned}
\eeq
and
\beq\label{Eq: sq2}
   \begin{aligned}
     &\lc \rho+\lambda_2\rc\sq_2-\lambda_2\sq_1
     -\frac{\p}{2\rho}\lc \rho x+y_2\rc^{\gamma +1} \left(   \vert D{\sq}_2(x)\vert^2 + 2 D\widehat{\sz}_2(x) D{\sq}_2(x)\right)\\
&+(1-\g \p)\lc \rho x+y_2\rc^{\g}\widehat{\sz}_2(x)D\sq_2(x)\\
={}&{-\frac{\lambda_2\g}{2}}\lc \rho x+y_2\rc^{-1-\g}(y_2-y_1)^2\lc 1-\frac{\rho \lambda_2}{(\rho+\lambda_1+\lambda_2)^2}-\frac{2\lambda_1}{\rho+\lambda_1+\lambda_2}\rc+O(x^{-2-\g}).
    \end{aligned}
\eeq
    We then rescale $\sq_j$ as follows: $\sq_j(x)= -(\rho x+y_j)^{-\gamma-1} Q_j(x)$.
    The a priori information on $\sq_j$ yields that $Q_j$ are smooth functions and that $Q_j(x)= O(1)$ and  $DQ_j(x)= O(1)$ as $x\to +\infty$. We then obtain from \cref{Eq: sq1} and \cref{Eq: sq2} that
    \begin{equation}
      \label{eq:10000}
   \begin{aligned}
     &\lc \rho+\lambda_1\rc Q_1-\lambda_1 Q_2
     +\frac{\p}{2\rho}  \vert DQ_1(x)\vert^2
+\frac{ \lambda_1}{\rho+\lambda_1+\lambda_2}(y_2-y_1) DQ_1
\\
={}&{\frac{\lambda_1\g}{2}}(y_2-y_1)^2\lc 1-\frac{\rho \lambda_1}{(\rho+\lambda_1+\lambda_2)^2}-\frac{2\lambda_2}{\rho+\lambda_1+\lambda_2}\rc+O(x^{-1}),
    \end{aligned}
  \end{equation}
  and
    \begin{equation}
      \label{eq:10001}
 \begin{aligned}
     &\lc \rho+\lambda_2\rc Q_2-\lambda_2 Q_1    +\frac{\p}{2\rho}  \vert DQ_2(x)\vert^2
+\frac{ \lambda_2}{\rho+\lambda_1+\lambda_2}(y_1-y_2) DQ_2
     \\
={}&{\frac{\lambda_2\g}{2}}(y_1-y_2)^2\lc 1-\frac{\rho \lambda_2}{(\rho+\lambda_1+\lambda_2)^2}-\frac{2\lambda_1}{\rho+\lambda_1+\lambda_2}\rc+O(x^{-1}),
    \end{aligned}
  \end{equation}
  We notice that the unique smooth and bounded solution of the following system of Hamilton-Jacobi equations posed in the full lines $\R\times \R$:
  $$
   \begin{aligned}
     &\lc \rho+\lambda_1\rc \widehat Q_1-\lambda_1 \widehat Q_2
     +\frac{\p}{2\rho}  \vert D\widehat Q_1(x)\vert^2
+ \frac{ \lambda_1 (y_2-y_1)}{\rho+\lambda_1+\lambda_2} D\widehat Q_1(x)
\\
={}&{\frac{\lambda_1\g}{2}}(y_2-y_1)^2\lc 1-\frac{\rho \lambda_1}{(\rho+\lambda_1+\lambda_2)^2}-\frac{2\lambda_2}{\rho+\lambda_1+\lambda_2}\rc, \\ \\
   &\lc \rho+\lambda_2\rc \widehat Q_2-\lambda_2 \widehat Q_1    +\frac{\p}{2\rho}  \vert D\widehat Q_2(x)\vert^2
+\frac{ \lambda_2 (y_1-y_2)}{\rho+\lambda_1+\lambda_2} D\widehat Q_2(x)
     \\
={}&{\frac{\lambda_2\g}{2}}(y_1-y_2)^2\lc 1-\frac{\rho \lambda_2}{(\rho+\lambda_1+\lambda_2)^2}-\frac{2\lambda_1}{\rho+\lambda_1+\lambda_2}\rc.
    \end{aligned}
  $$
  is given by the constants
  $$
  \begin{aligned}
    \widehat Q_1& =
    {\frac{\g \lambda_1(y_2-y_1)^2}{2(\rho+\lambda_1+\lambda_2)}}\lc1-\frac{\lambda_1\lambda_2+\lambda_2^2+\rho \lambda_1}{(\rho+\lambda_1+\lambda_2)^2}\rc ,\\
    \widehat Q_2&={\frac{\g \lambda_2(y_2-y_1)^2}{2(\rho+\lambda_1+\lambda_2)}}\lc1-\frac{\lambda_1\lambda_2+\lambda_1^2+\rho \lambda_2}{(\rho+\lambda_1+\lambda_2)^2}\rc.
  \end{aligned}
  $$
  This indicates that  when $x\to +\infty$,  $\sq_j(x)\sim \widehat{\sq}_j(x)$ (for brevity, we omit the proof of this point), where
  \beq
  \widehat{\sq}_1(x)=-\frac{\g \lambda_1(y_2-y_1)^2}{2(\rho+\lambda_1+\lambda_2)}\lc1-\frac{\lambda_1\lambda_2+\lambda_2^2+\rho \lambda_1}{(\rho+\lambda_1+\lambda_2)^2}\rc \lc \rho x+y_1\rc^{-1-\g},
  \eeq
  \beq
  \widehat{\sq}_2(x)=-\frac{\g \lambda_2(y_2-y_1)^2}{2(\rho+\lambda_1+\lambda_2)}\lc1-\frac{\lambda_1\lambda_2+\lambda_1^2+\rho \lambda_2}{(\rho+\lambda_1+\lambda_2)^2}\rc \lc \rho x+y_2\rc^{-1-\g}.
  \eeq
  Moreover, injecting the information that  $Q_j(x)\sim \widehat{Q}_j(x)$  as $x\to +\infty$
  into \eqref{eq:10000}-\eqref{eq:10001}, we see that as $x\to \infty$,
  $$
  \begin{aligned}
     &
     \frac{\p}{2\rho}  \vert DQ_1(x)\vert^2
+ \frac{ \lambda_1}{\rho+\lambda_1+\lambda_2}(y_2-y_1) DQ_1(x)
=o(1),\\
  &
     \frac{\p}{2\rho}  \vert DQ_2(x)\vert^2
+ \frac{ \lambda_2}{\rho+\lambda_1+\lambda_2}(y_1-y_2) DQ_2(x)
=o(1).
    \end{aligned}
$$
Since $Q_1$ and $Q_2$ are smooth and tend to constants at $+\infty$, the latter system of equations implies that  $DQ_j(x)=o(1)$  as $x\to +\infty$, $j=1,2$. This implies that $D\sq_j(x)\sim D\widehat{\sq}_j(x)$ as $x\to \infty$, 
where
$$
\begin{aligned}
  D\widehat{\sq}_1(x)
  ={}&\frac{\rho \g(1+\g)\lambda_1(y_2-y_1)^2}{2(\rho+\lambda_1+\lambda_2)}\lc1-\frac{\lambda_1\lambda_2+\lambda_2^2+\rho \lambda_1}{(\rho+\lambda_1+\lambda_2)^2}\rc \lc \rho x+y_1\rc^{-2-\g},\\
  D\widehat{\sq}_2(x)
  ={}&\frac{\rho \g(1+\g)\lambda_2(y_2-y_1)^2}{2(\rho+\lambda_1+\lambda_2)}\lc1-\frac{\lambda_1\lambda_2+\lambda_1^2+\rho \lambda_2}{(\rho+\lambda_1+\lambda_2)^2}\rc \lc \rho x+y_2\rc^{-2-\g}.
  \end{aligned}
  $$
  
  We now derive the asymptotic behavior of the optimal consumption as $x\to +\infty$:
 \begin{align*}
 &c_1(x)\\
= {}&\rho^{\p}\lc \rho\lc \rho x+y_1\rc^{-\g}+D\widehat{\sz}_1(x)+D\sq_1(x)\rc^{-\p}((1-\g)v_1(x))^{\frac{1-\g \p}{1-\g}}\\
= {}& \lc \rho x+y_1\rc\left[1-\frac{\p\lc \rho x+y_1\rc^{\g}}{\rho}\left(D\widehat{\sz}_1(x)+D\widehat{\sq}_1(x)\right)+\frac{\lc \rho x+y_1\rc^{2\g}\p(1+\p)}{2\rho^2}\vert D\widehat{\sz}_1(x)\vert^2\right]\\
  {}&\cdot\left[1+(1-\g\p)\frac{\widehat{\sz}_1(x)+\widehat{\sq}_1(x)}{\lc \rho x+y_1\rc^{1-\g}}+\frac{(1-\g\p)\g(1-\p)}{2\lc \rho x+y_1\rc^{2-2\g}}\vert \widehat{\sz}_1(x)\vert^2\right]+o((\rho x+y_1)^{-1})\\
={}& \lc \rho x+y_1\rc+\frac{\lambda_1(y_{2}-y_1)}{\rho+\lambda_1+\lambda_{2}}+\frac{\g (1+\p)\lambda_1^2(y_2-y_1)^2}{2(\rho+\lambda_1+\lambda_2)^2}\lc \rho x+y_1\rc^{-1}\\
&-\frac{\g (1+\p)\lambda_1(y_2-y_1)^2}{2(\rho+\lambda_1+\lambda_2)}\lc1-\frac{\lambda_1\lambda_2+\lambda_2^2+\rho \lambda_1}{(\rho+\lambda_1+\lambda_2)^2}\rc \lc \rho x+y_1\rc^{-1}+o((\rho x+y_1)^{-1}) .
 \end{align*}
 The calculation for $c_2$ is similar. We can then deduce the expansion of $s_j$ in \cref{Eq: sj asmp}. 
\end{proof}
\subsection{Proof of \Cref{Prop: no sol FPK}}
\begin{proof} Suppose there exists an invariant measure $m=(m_1,m_2)$ when $r=\rho$, then it has densities $g_j$ defined on $[\ux,+\infty)$, $j\in \{2,1\}$, given by \cref{Eq: g2} and \cref{Eq: g1}. We write \cref{Eq: sj asmp} as 
\begin{align*}
s_1(x)={}&\frac{\g (1+\p)\lambda_1(y_2-y_1)^2}{2(\rho+\lambda_1+\lambda_2)}\left[\underbrace{1-\frac{\lambda_1\lambda_2+\lambda_2^2+\rho \lambda_1}{(\rho+\lambda_1+\lambda_2)^2}-\frac{\lambda_1}{(\rho+\lambda_1+\lambda_2)}}_{(I)}\right](\rho x+y_1)^{-1}\\
&+\frac{\lambda_1(y_{1}-y_2)}{\rho+\lambda_1+\lambda_{2}}+o((\rho x+y_1)^{-1}),\\
s_2(x)={}&\frac{\g (1+\p)\lambda_2(y_2-y_1)^2}{2(\rho+\lambda_1+\lambda_2)}\left[\underbrace{1-\frac{\lambda_1\lambda_2+\lambda_1^2+\rho \lambda_2}{(\rho+\lambda_1+\lambda_2)^2}-\frac{\lambda_2}{(\rho+\lambda_1+\lambda_2)}}_{(II)}\right](\rho x+y_2)^{-1}\\
&+\frac{\lambda_2(y_{2}-y_1)}{\rho+\lambda_1+\lambda_{2}}+o((\rho x+y_2)^{-1}).
 \end{align*}
We notice that in the equations for $s_1(x)$ and $s_2(x)$,
$$
(I)+(II)=1+\frac{\rho}{(\rho+\lambda_1+\lambda_2)}-\frac{\lambda_1+\lambda_2}{(\rho+\lambda_1+\lambda_2)}=\frac{2\rho}{(\rho+\lambda_1+\lambda_2)},
$$
it follows that as $x\to +\infty$, 
$$
\frac{\ld_2s_1(x)+\ld_1s_2(x)}{-s_1(x)s_2(x)}=\frac{\rho \g(1+\p)}{\rho x+y_1}+o((\rho x+y_1)^{-1}).
$$
With the formulas \cref{Eq: g2} and \cref{Eq: g1}, we obtain $g_2(x)$ and $g_1(x)$ tend to $+\infty$ as $x\to +\infty$, in contradiction with the fact that $m_1$ and $m_2$ have finite mass. 
 \end{proof}

\section{Asymptotic analysis near $x=\ux$}\label{Expansion ux}
\section{Asymptotic analysis near $x=\ux$}\label{Expansion ux}
We now consider the behavior of $s_1$ near $\ux$. From \cref{c optimal} and $s_1(\ux)=0$, we have 
\beq\label{Eq: Dv1 ux}
Dv_1(\ux)=\rho (r\ux+y_1)^{-\p^{-1}}((1-\g)v_1(\ux))^{\frac{\p^{-1}-\g}{1-\g}}.
\eeq
From \Cref{Prop: s1<0}, we can differentiate the HJB equation for $v_1$ at $x>\ux$ and pass to the limit $x\to \ux$, 
$$
\begin{aligned}
&(\z-r)Dv_1(\ux)+\ld_1(Dv_1(\ux)-Dv_2(\ux))\\
={}&\lim_{x\to \ux}s_1(x)D^2v_1(x)+H_v(\ux,y_1,v_1(\ux),Dv_1(\ux))Dv_1(\ux).
\end{aligned}
$$
Moreover, $\lim_{x\to \ux}s_1(x)D^2v_1(x)>0$. We denote the constant 
$$
\varkappa:=(\z-r)Dv_1(\ux)+\ld_1(Dv_1(\ux)-Dv_2(\ux))-H_v(\ux,y_1,v_1(\ux),Dv_1(\ux))Dv_1(\ux),
$$
and we can derive from \Cref{Eq: Hv ux} and $\rho> r$ that $\varkappa>0$.\par
For $x>\ux$, we write
\beq\label{Eq: Dv1 q1}
Dv_1(x)=\rho (r\ux+y_1)^{-\p^{-1}}((1-\g)v_1({x}))^{\frac{\p^{-1}-\g}{1-\g}}+q_1(x).
\eeq
It is clear from \Cref{Eq: Dv1 ux} that $q_1(\ux)=0$. From the continuity and boundedness of $v_1$ near $\ux$ we know 
\beq\label{Eq: v1 o1}
((1-\g)v_1(x))^{\frac{\g-\p^{-1}}{1-\g}}=((1-\g)v_1(\ux))^{\frac{\g-\p^{-1}}{1-\g}}+o(1),
\eeq
hence $q_1(x)=o(1)$ as $x\to \ux$.\par 
We notice that since $\p^{-1}-\g>0$, $((1-\g)v)^{\frac{\p^{-1}-\g}{1-\g}}$ is an increasing function of $v$ on $(-\infty,0)$. From the concavity and monotonicity of $v_1$ we have
$$
\begin{aligned}
Dv_1(x)<{}&Dv_1(\ux)=\rho (r\ux+y_1)^{-\p^{-1}}((1-\g)v_1(\ux))^{\frac{\p^{-1}-\g}{1-\g}}\\
<{}&\rho (r\ux+y_1)^{-\p^{-1}}((1-\g)v_1({x}))^{\frac{\p^{-1}-\g}{1-\g}},
\end{aligned}
$$
hence $q_1(x)<0$. From $D^2v_1(x)\to -\infty$ and 
\beq\label{Eq: D^2v1}
D^2v_1(x)=\underbrace{\rho (\p^{-1}-\g) (r\ux+y_1)^{-\p^{-1}}((1-\g)v_1({x}))^{\frac{\p^{-1}-1}{1-\g}}Dv_1({x})}_{>0}+Dq_1(x),
\eeq
we have $Dq_1(x)\to -\infty$ as $x\to \ux$. From \cref{c optimal} and \cref{Eq: Dv1 q1}, the consumption near $\ux$ is
$$
\begin{aligned}
c_1(x)={}&\rho^{\p}\lc\rho (r\ux+y_1)^{-\p^{-1}}((1-\g)v_1(x))^{\frac{\p^{-1}-\g}{1-\g}}+q_1(x)\rc^{-\p}((1-\g)v_1(x))^{\frac{1-\g \p}{1-\g}}\\
={}&(r\ux+y_1)\lc1+\underbrace{\rho^{-1} (r\ux+y_1)^{\p^{-1}}((1-\g)v_1(x))^{\frac{\g-\p^{-1}}{1-\g}}q_1(x)}_{=o(1),\,\,{\text{from}} \,q_1(x)=o(1)\, {\text{and}}\,\cref{Eq: v1 o1}}\rc^{-\p}\\
={}& r\ux+y_1-{\rho^{-1} }\p \lc r\ux+y_1\rc^{1+\p^{-1}}((1-\g)v_1(x))^{\frac{\g-\p^{-1}}{1-\g}}q_1(x)+o(q_1(x)),
\end{aligned}
$$
it follows
$$
s_1(x)= r(x-\ux)+\rho^{-1} \p \lc r\ux+y_1\rc^{1+\p^{-1}}((1-\g)v_1(x))^{\frac{\g-\p^{-1}}{1-\g}}q_1(x)+o(q_1(x)).
$$
Using again \cref{Eq: v1 o1}, we obtain
\beq\label{Eq: s1 ux}
\begin{aligned}
s_1(x)= r(x-\ux)+\rho^{-1} \p \lc r\ux+y_1\rc^{1+\p^{-1}}((1-\g)v_1(\ux))^{\frac{\g-\p^{-1}}{1-\g}}q_1(x)+o(q_1(x)).
\end{aligned}
\eeq
By differentiating the HJB equation for $v_1$ at $x>\ux$, we infer that in a right neighborhood of $\ux$,
\beq\label{Eq: s1Dq1}
\begin{aligned}
&s_1(x)Dq_1(x)\\
= {}&\varkappa+\underbrace{H_v(\ux,y_1,v_1(\ux),Dv_1(\ux))Dv_1(\ux)-H_v(x,y_1,v_1(x),Dv_1(x))Dv_1(x)}_{(I)}\\
&+\underbrace{(\z-r)(Dv_1(x)-Dv_1(\ux))+\ld_1(Dv_1(x)-Dv_2(x)-Dv_1(\ux)+Dv_2(\ux))}_{(II)}\\
&+\underbrace{s_1(x)\lc(Dq_1(x)-D^2v_1(x)\rc}_{(III)}\\
= {}&\varkappa+o(1).
\end{aligned}
\eeq
To obtain $(I)+(II)=o(1)$, we used the continuity of $Dv_j$ at $\ux$ (from \cref{Dv UC}) and the estimates on $H_{vv}$ and $H_{vp}$ (cf. \cref{Hvp} and \cref{Hvv}). Moreover, $(III)=o(1)$ comes from \cref{Eq: D^2v1} and $s_1(x)\to 0$ as $x\to \ux$. Next we denote
\beq\label{Eq: Q1}
Q_1(x):=q^2_1(x),\quad Q_1(\ux)=0,\quad DQ_1(x)=2q_1(x)Dq_1(x).
\eeq
From \cref{Eq: Q1}, \cref{Eq: s1 ux} and \cref{Eq: s1Dq1},
\beq\label{Eq: Q1 DQ1}
\begin{aligned}
&\underbrace{r(x-\ux)Dq_1(x)}_{<0}+\rho^{-1} \p \lc r\ux+y_1\rc^{1+\p^{-1}}((1-\g)v_1(\ux))^{\frac{\g-\p^{-1}}{1-\g}}\frac{DQ_1(x)}{2}+o(DQ_1(x))\\
={}& {\varkappa}+o(1).
\end{aligned}
\eeq
We obtain that in a right neighborhood of $\ux$,
$$
\rho^{-1} \p \lc r\ux+y_1\rc^{1+\p^{-1}}((1-\g)v_1(\ux))^{\frac{\g-\p^{-1}}{1-\g}}\frac{DQ_1(x)}{2}>{\varkappa/2}.
$$
This and $Q_1(0)=0$ yield, for $x-\ux$ sufficiently small,
$$
Q_1(x)>\frac{\rho{\varkappa}}{\p \lc r\ux+y_1\rc^{1+\p^{-1}}((1-\g)v_1(\ux))^{\frac{\g-\p^{-1}}{1-\g}}}(x-\ux),
$$
 hence $x-\ux=O(Q_1(x))=o(q_1(x))$ as $x\to \ux$. This implies $r(x-\ux)Dq_1(x)=o(DQ_1(x))$, and then \cref{Eq: Q1 DQ1} yields
$$
\frac{\rho^{-1} \p \lc r\ux+y_1\rc^{1+\p^{-1}}((1-\g)v_1(\ux))^{\frac{\g-\p^{-1}}{1-\g}}}{2}DQ_1(x)+o(DQ_1(x))=\varkappa+o(1).
$$
This yields
$$
Q_1(x)=\frac{2{\rho}\varkappa(x-\ux)}{\p \lc r\ux+y_1\rc^{1+\p^{-1}}}\cdot ((1-\g)v_1(\ux))^{\frac{\p^{-1}-\g}{1-\g}}+o(x-\ux).
$$
Therefore 
$$
q_1(x)=-\sqrt{\frac{2{\rho}\varkappa(x-\ux)}{\p \lc r\ux+y_1\rc^{1+\p^{-1}}}\cdot ((1-\g)v_1(\ux))^{\frac{\p^{-1}-\g}{1-\g}}}+o(\sqrt{x-\ux}).
$$
We have found that, as $x\to \ux$,
$$
\begin{aligned}
&Dv_1(x)\\
\sim {}&\rho (r\ux+y_1)^{-\p^{-1}}((1-\g)v_1(\ux))^{\frac{\p^{-1}-\g}{1-\g}}
-\sqrt{\frac{2{\rho}\varkappa(x-\ux)}{\p \lc r\ux+y_1\rc^{1+\p^{-1}}}\cdot ((1-\g)v_1(\ux))^{\frac{\p^{-1}-\g}{1-\g}}}.
\end{aligned}
$$
We observe that this is consistent with the fact that $Dv_1$ is uniformly continuous but not Lipschitz continuous, given by \Cref{Dv UC} and \Cref{Cor: D^2v1}.

\end{appendices}\par
\bigskip
\noindent {\bf Acknowledgment.}
\small{Yves Achdou was partially supported by the ANR (Agence Nationale de la Recherche) through project ANR-22-CE40-0010-01 and by the chair Finance and Sustainable Development and FiME Lab (Institut Europlace de Finance).\par 
This work was carried out while Qing Tang was a visitor at Laboratoire Jacques-Louis Lions, Universit\'e Paris Cit\'e, supported by a China Scholarship Council grant (No.~202406410221).}\\

\urlstyle{same}

\noindent {\bf ORCID}\\
\small{Yves Achdou \url{https://orcid.org/0000-0002-0514-3727}\\
 Qing Tang \url{https://orcid.org/0009-0003-6152-9549}}

\end{document}